\documentclass[10pt,oneside,shortlabels, reqno]{amsart}
\usepackage{color}
\usepackage{enumitem}
\usepackage{amsbsy}
\usepackage{amstext}
\usepackage{tikz-cd}
\usepackage{amsthm}
\usepackage{amssymb}
\usepackage[unicode=true,pdfusetitle,
bookmarks=true,bookmarksnumbered=false,bookmarksopen=false,
breaklinks=false,pdfborder={0 0 0},pdfborderstyle={},backref=page,colorlinks=true]
{hyperref}
\hypersetup{
	linkcolor=blue}

\makeatletter
\usepackage{mathrsfs}

\usepackage[all]{xy}
\newcommand{\xyL}[1]{%
	\xydef@\xymatrixrowsep@{#1}
} % end of \xyL
\newcommand{\xyC}[1]{%
	\xydef@\xymatrixcolsep@{#1}
} % end of \xyC

\theoremstyle{plain}
\newtheorem{thm}{\protect\theoremname}[section]
\theoremstyle{remark}
\newtheorem{rem}[thm]{\protect\remarkname}
\newcommand\thmsname{Theorem}
\newcommand\nm@thmtype{thm}
\theoremstyle{plain}

\newenvironment{namedthm}[1]{
	\renewcommand\thmsname{#1}\renewcommand\nm@thmtype{namedtheorem}
	\begin{\nm@thmtype}
	}
	{\end{\nm@thmtype}
}
\theoremstyle{remark}
\newtheorem{remark}[thm]{Remark}
\theoremstyle{definition}
\newtheorem{example}[thm]{\protect\examplename}
\theoremstyle{definition}
\newtheorem{defn}[thm]{\protect\definitionname}
\theoremstyle{plain}
\newtheorem{prop}[thm]{\protect\propositionname}
\newtheorem{cor}[thm]{Corollary}
\theoremstyle{plain}
\newtheorem{lem}[thm]{\protect\lemmaname}
\theoremstyle{definition}

\numberwithin{equation}{section}

\makeatother

\providecommand{\definitionname}{Definition}
\providecommand{\examplename}{Example}
\providecommand{\lemmaname}{Lemma}
\providecommand{\propositionname}{Proposition}
\providecommand{\remarkname}{Remark}
\providecommand{\theoremname}{Theorem}

\address[dino.peran@pmfst.hr]{Dino Peran, University of Split, Faculty of Science, Department of Mathematics, Ru\dj era Bo\v skovi\' ca 33, 21000 Split, Croatia}

\begin{document}
	\global\long\def\acc#1#2{\accentset#1#2}
	
	\title[ ]{Normal forms of parabolic logarithmic transseries}

	\author[D.\ Peran]{D.\ Peran}
	
	\thanks{This research of D. Peran is fully supported by the Croatian Science Foundation (HRZZ) grant UIP-2017-05-1020. It is also supported by the Hubert-Curien `Cogito’ grant 2021/2022 \emph{Fractal and transserial approach to differential equations}.}
	
	\begin{abstract}
		We give formal normal forms for parabolic logarithmic transseries $f=z+\cdots \, $, with respect to parabolic logarithmic normalizations. Normalizations are given algorithmically, using fixed point theorems, as limits of Picard's sequences in appropriate complete metric spaces, in contrast to transfinite \emph{term-by-term} eliminations described in former works. Furthermore, we give the explicit formula for the residual coefficient in the normal form and show that, in the larger logarithmic class, we can even eliminate the residual term from the normal form.
	\end{abstract}

	\maketitle
	\tableofcontents{}

	\section{Introduction}
			
		Transseries are generalizations of the standard power series. Shortly, transseries are formal sums of monomials obtained as formal products of powers, iterated exponentials and iterated logarithms (see \cite{Dries97}). Additionally, their sets of exponents (i.e., supports) are well-ordered sets, which is important for defining standard operations such as multiplication and composition of transseries. Transseries are nowdays an important tool for tackling problems in mathematics (see e.g. \cite{Ecalle92}, \cite{Ily84}) and physics (see \cite{abs19}). \\
		
		In this paper we consider the \emph{logarithmic transseries} which do not contain exponentials in their monomials. They are introduced in \cite{adh13} and called ``purely logarithmic transseries''. One of the reasons for this restriction comes from dynamics. The \emph{Dulac problem} of nonaccumulation of limit cycles on a hyperbolic or a semi-hyperbolic polycycle of an analytic planar vector field is closely related to the theory of logarithmic transseries through the first return map and its asymptotic expansion (see e.g. \cite{dulac}, \cite{Ecalle92}, \cite{Ily84}, \cite{Roussarie98}). The Dulac problem was solved by \emph{Ilyashenko} (see e.g. \cite{Ily84}, \cite{Ily91}) and \emph{Écalle} (see \cite{Ecalle92}) independently. Limit cycles can be detected as the fixed points of the \emph{first return map} (or \emph{Poincaré map}) of a polycycle. More precisely, it is shown in \cite{Ily91} that the first return maps of hyperbolic polycycles have logarithmic asymptotic expansions of the particular type, called the \emph{Dulac series} (see e.g. \cite[Section 24]{IlyYak08}). By \cite{Ily91} (or \cite{IlyYak08}) the first return maps of hyperbolic polycycles can be analytically extended to sufficiently large complex domains called the \emph{standard quadratic domains}. Therefore, by the \emph{Phragmen-Lindel\"of Theorem} (see e.g. \cite[Theorem 24.36]{IlyYak08}) the first return maps are uniquely determined by their asymptotic expansions (see e.g. \cite[Theorem 24.29]{IlyYak08}). It represents the part that was missing in the original Dulac's proof (see \cite{dulac}). This motivates the definition of \emph{Dulac germs} (or \emph{almost regular germs}) as analytic germs on standard quadratic domains which admit Dulac series as their asymptotic expansions uniformly on whole domains (see e.g. \cite[Definition 24.27]{IlyYak08}).
		
		Although Ilyashenko and Écalle solved the Dulac problem in a full generality (for hyperbolic and semi-hyperbolic polycycles), the semi-hyperbolic case is not well-understood yet. By \cite{Ily91}, unlike the hyperbolic case, it seems that the first return maps of semi-hyperbolic polycycles admit also iterated logarithms in their asymptotic expansions (see e.g. \cite{Ily91}). It seems that the most interesting polycycles are the ones with \emph{parabolic} (or \emph{tangent to the identity}) first return maps, while \emph{attractive} (\emph{hyperbolic} and \emph{strongly hyperbolic}) ones have trivial cyclicity (see \cite{Roussarie98}). \\
		
		Our long term plan is to classify analytically the first return maps of hyperbolic and semi-hyperbolic polycycles. An analytic normal form is, in some sense, the simplest representative of an analytic class.  
		
		Recently, there has been some progress in that direction. More precisely, in \cite{mr21}, \cite[Theorem B]{prrsDulac21}, \cite[Theorem C]{Peran21} the analytic normal forms of parabolic (resp. hyperbolic and strongly hyperbolic) Dulac germs are obtained. Note that these results provide the analytic classification of all Dulac germs and consequently, of all first return maps in the hyperbolic case. The first step in the proofs of all these results is obtaining formal normalizations and formal normal forms which are given in \cite{mrrz16} for transseries involving only powers and logarithms (i.e., without iterated logarithms).
		
		Motivated by the semi-hyperbolic case with iterated logarithms expected in the expansion, in \cite[Main Theorem]{prrs21} (resp. \cite[Theorem A]{Peran21}) the formal normal forms of hyperbolic (resp. strongly hyperbolic) logarithmic transseries (involving also iterated logarithms) are obtained.
		
		To complete the formal classification of logarithmic transseries, in this paper we obtain the formal classification of parabolic logarithmic transseries (see the Main Theorem). \\
		
		In this paper, to get the formal normal forms of parabolic logarithmic transseries, we use the similar fixed point techniques as in \cite{Peran21} and \cite{prrs21}. Consequently, normalizations are given as limits of Picards sequences in appropriate topologies. In \cite{prrsDulac21} and \cite{Peran21} this enabled the sufficient control to relate the formal and the analytic normalizations via the standard \emph{Poincaré asymptotic expansions}.
		
		However, in \cite{mrrz19} and \cite{mr21} it was shown that the formal normalizations of parabolic Dulac germs are, in general, more complicated logarithmic transseries than Dulac series. The main difference is that the supports of formal normalizations of parabolic Dulac germs are, in general, of order type strictly bigger than $\omega $. In order to well-define the transserial asymptotic expansion of such normalizations, the authors were forced to introduce the a particular \emph{summation rule} for partial sums at limit ordinal steps (i.e., a particular choice of germ related to a partial sum).
		
		The summation rule for parabolic Dulac germs is complicated, since it is based on the transfinite algorithm for their formal classification. We introduce here the fixed point technique motivated by a similar technique from \cite{prrs21}, which allows us a better control of the supports of formal normalizations, which, we belive, could simplify summation rules. In addition, we generalize \cite[Theorem A]{mrrz16} for parabolic logarithmic transseries in general, i.e., to those involving also the iterated logarithms.
		
		It is important to emphasize that in some point of the proof of our main result (see the Main Theorem) we use fixed point techniques to solve a particular type of nonlinear differential equations in the differential algebra of logarithmic transseries (see Proposition~\ref{Lemma3}) which is an interesting result hidden in the proof of the Main Theorem. \\
		
		The main result of this paper is given in the Main Theorem on page \pageref{TeoremA}. For a parabolic power series $f:=z+az^{n}+\cdots \, $, for $a\in \mathbb{C}\setminus \left\lbrace 0\right\rbrace $, $n\in \mathbb{N}_{\geq 2}$, it is known that its "shortest" formal normal form, with respect to power series normalizations, is $f_{c}:=z+az^{n}+cz^{2n-1}$, for a unique $c\in \mathbb{C}$ which we call the residual coefficient of $f$ (see e.g. \cite{CarlesonG93}). Furthermore, for a parabolic power-logarithmic transseries $f:=z+az^{\beta }\boldsymbol{\ell }^{n}+\cdots \, $, where $\boldsymbol{\ell }:=-\frac{1}{\log z}$, $a\in \mathbb{R}\setminus \left\lbrace 0\right\rbrace $, $n\in \mathbb{Z}$, $(\beta , n)>(1,0)$ lexicographically, it is shown in \cite[Theorem A]{mrrz16} that its formal normal form, with respect to the power-log normalizations, is $f_{c}:=z+az^{\beta }\boldsymbol{\ell }^{n}+cz^{2\beta -1}\boldsymbol{\ell }^{2n+1}$, for a unique $c\in \mathbb{R}$. In the fashion of these results, one would expect that formal normal forms, with respect to logarithmic normalizations, obtained in the Main Theorem can be chosen as finite series. However, it is shown in the Main Theorem that this is true in general only for parabolic logarithmic transseries $f$ such that the first term of $f-z$ has exponent of $z$ strictly bigger than $1$.
		
		Following \cite[Proposition 9.3]{mrrz19tubular}, in the Main Theorem we obtain an explicit formula for the residual coefficient of a parabolic logarithmic transseries. Furthermore, in Remark~\ref{Remark2} we prove that a parabolic logarithmic transseries can be \emph{embedded} in a formal vector field as the \emph{time-one} map. \\
		
		As opposed to (strongly) hyperbolic logarithmic transseries, if we restrict the depth of iterated logarithms, our normal form may change in parabolic case. Indeed, a standard parabolic power series $f$ with real coefficients can be formally normalized in the space of all power series to the normal form $f_{c}$ for a unique $c\in \mathbb{R}$ (which is not necessary $0$). If we allow parabolic changes of variables which involve logarithms, by \cite[Example 6.3]{mrrz16}, its formal normal form in the larger space is $f_{0}$ (i.e., $c=0$). That is, we can eliminate the residual coefficient in the larger space. In Corallary~\ref{Remark3} we generalize \cite[Example 6.3]{mrrz16} for logarithmic transseries of arbitrary depth of the iterated logarithm. \\
		
		Let us now explain briefly the structure of the paper. In Section~\ref{SectionNotionMThm} we recall the basic notions and we state the Main Theorem about the formal normalization of parabolic logarithmic transseries. Furthermore, Section~\ref{SectionPrereqForMThm} serves as a prerequisite for proving the Main Theorem in the following sections. The most important result stated in this section is the fixed point theorem from Proposition~\ref{KorBanach} that is originally proven in \cite[Proposition 3.2]{prrs21}. The proof of the Main Theorem is in Section~\ref{sec:proofA}. More precisely, the proof of the statement 1 of the Main Theorem is in Sections~\ref{SubsectionSketchMainThm}-\ref{sec:proofAb}, and the proofs of statements 2 and 3 of the Main Theorem are in Section~\ref{sec:ProofMinimality}. Finally, Section~\ref{sec:proofRemark} is dedicated to the proofs of Remark~\ref{Remark2} and Corollary~\ref{Remark3}.

	\section{Notation and the Main Theorem}\label{SectionNotionMThm}
		
		\subsection{Differential algebras $\mathfrak L$ and $\mathfrak L^{\infty }$}
			
			We repeat here shortly the notation and some facts introduced in e.g. \cite{prrs21} that we use in this paper. Let $\boldsymbol{\ell }_{1}:=-\frac{1}{\log z}$ and $\boldsymbol{\ell }_{k+1} :=\boldsymbol{\ell}_{k} \circ \boldsymbol{\ell } $, $k\in \mathbb{N}_{\geq 1}$, where $z$ is the \emph{formal variable}.
			
			By $\mathcal{L}_{k}$, $k\in \mathbb{N}$, we denote the set of all \emph{logarithmic transseries of depth} $k$:
			\begin{align}
				f & := \sum _{\left( \alpha ,n_{1},\ldots ,n_{k}\right) \in \mathbb{R}\times \mathbb{Z}^{k}} a_{\alpha ,n_{1},\ldots ,n_{k}} z^{\alpha }\boldsymbol{\ell }^{n_{1}}_{1}\cdots \boldsymbol{\ell }^{n_{k}}_{k} \, , \label{TransseriesDef}
			\end{align}
			where
			\begin{align*}
				\mathrm{Supp} (f) := \left\lbrace (\alpha , n_{1},\ldots ,n_{k}) \in \mathbb{R}\times \mathbb{Z}^{k} : a_{\alpha , n_{1},\ldots ,n_{k}} \neq 0 \right\rbrace ,
			\end{align*}
			is a well-ordered subset of $\mathbb{R} \times \mathbb{Z}^{k}$ with respect to the lexicographic order, such that $\min \, \mathrm{Supp}(f) > (0,0,\ldots ,0)_{k+1}$. If $\mathrm{Supp} (f)=\emptyset $, then we set $f:=0$ and call it the \emph{zero transseries}. We call $\mathrm{Supp} (f)$ the \emph{support of} $f$. 
			
			Note that $\mathcal{L}_{k}\subseteq \mathcal{L}_{k+1}$, for every $k\in \mathbb{N}$.
			
			Every
			\begin{align*}
				& a_{\alpha ,n_{1},\ldots ,n_{k}} z^{\alpha }\boldsymbol{\ell }^{n_{1}}_{1}\cdots \boldsymbol{\ell }^{n_{k}}_{k} , \quad a_{\alpha ,n_{1},\ldots ,n_{k}} \neq 0,
			\end{align*}
			in \eqref{TransseriesDef} is called the \emph{term} in $f$ of order $\left( \alpha ,n_{1},\ldots ,n_{k}\right) $. We denote by
			\begin{align*}
				\left[ f \right] _{\alpha ,n_{1},\ldots ,n_{k} } & := a_{\alpha ,n_{1},\ldots ,n_{k}}
			\end{align*}
			the \emph{coefficient} of the term $a_{\alpha ,n_{1},\ldots ,n_{k}}z^{\alpha }\boldsymbol{\ell }^{n_{1}}_{1}\cdots \boldsymbol{\ell }^{n_{k}}_{k} $. \\
			
			Additionally, if we allow the support $\mathrm{Supp}(f)$ of transseries $f$ in \eqref{TransseriesDef} to contain terms of order strictly smaller than $(0,0,\ldots ,0)_{k+1}$, we denote the set of all such $f$ by $\mathcal{L}_{k}^{\infty }$. \\
			
			We rewrite $f\in \mathcal{L}_{k}^{\infty }$, $k\in \mathbb{N}$, as:
			\begin{align*}
				f & :=\sum_{\alpha \in \mathbb{R}} z^{\alpha }R_{\alpha } ,
			\end{align*}
			where
			\begin{align*}
				R_{\alpha} & := \sum_{\left( \alpha ,n_{1},\ldots ,n_{k}\right) \in \mathbb{R}\times \mathbb{Z}^{k}} a_{\alpha ,n_{1},\ldots n_{k}}\boldsymbol{\ell }_{1}^{n_{1}}\cdots \boldsymbol{\ell}_{k}^{n_{k}} .
			\end{align*}
			By
			\begin{align*}
				\mathrm{Supp}_{z} (f) & := \left\lbrace \alpha \in \mathbb{R} : R_{\alpha } \neq 0 \right\rbrace ,
			\end{align*}
			we denote the well ordered set of all exponents of $z$ in $f$ that we call the \emph{support of} $f$ \emph{in} $z$. Every $z^{\alpha }R_{\alpha }$, for $R_{\alpha }\neq 0$, is called a \emph{block of order} $\alpha $ in $z$, or an $\alpha $-block.
			
			The \emph{order} of $f\in \mathcal{L}_{k}^{\infty }$, $k\in \mathbb{N}$, is defined as $\mathrm{ord}(f):=\min \mathrm{Supp}(f)$, if $\mathrm{Supp}(f) \neq \emptyset $. In case that $\mathrm{Supp}(f)=\emptyset $, we have $f=0$, so we put $\mathrm{ord}(f):=\infty $, where $\infty $ denotes an element that is lexicographically strictly bigger than any element of the set $\mathbb{R} \times \mathbb{Z}^{k}$.
			
			The term in $f\in \mathcal{L}_{k}^{\infty }$, $f\neq 0$, of order $\mathrm{ord}(f)$ is called the \emph{leading term} of $f$, and denoted by $\mathrm{Lt}(f)$.
			
			The \emph{order in} $z$ of $f\in \mathcal{L}_{k}^{\infty }$, $k\in \mathbb{N}$, is defined as $\mathrm{ord}_{z}(f):=\min \, \mathrm{Supp}_{z}(f)$, if $\mathrm{Supp}_{z}(f)\neq \emptyset $. If $\mathrm{Supp}_{z}(f)=\emptyset $, then $f=0$, so we put $\mathrm{ord}_{z}(f)=\infty $, where $\infty $ is the maximum of the set $\mathbb{R}\cup \left\lbrace \infty \right\rbrace $.
			
			The block in $f\in \mathcal{L}_{k}^{\infty }$, $f\neq 0$, $k\in \mathbb{N}$, of order $\mathrm{ord}_{z}(f)$ is called the \emph{leading block} of $f$ in $z$, and denoted by $\mathrm{Lb}_{z}(f)$. \\
			
			Using the Neumann Lemma (see \cite{Neumann49}, or \cite[Lemma 2.2]{mrrz16}), it is easy to see that the \emph{term-wise} multiplication is well defined on $\mathcal{L}_{k}^{\infty }$, $k\in \mathbb{N}$. Furthermore, it can be shown that $\mathcal{L}_{k}^{\infty }$ and $\mathcal{L}_{k}$, $k\in \mathbb{N}$, are differential algebras, with respect to the derivation $\frac{d}{dz}$.
			
			Put, as in \cite{prrs21},
			\begin{align*}
				& \mathfrak L^{\infty } :=\bigcup _{k\in \mathbb{N}}\mathcal{L}_{k}^{\infty } \quad \textrm{ and } \quad \mathfrak L :=\bigcup _{k\in \mathbb{N}}\mathcal{L}_{k} .
			\end{align*}
			It is easy to see that $\mathfrak L^{\infty }$ is a differential algebra and $\mathfrak L$ and $\mathcal{L}_{k}^{\infty }$, $\mathcal{L}_{k}$, $k\in \mathbb{N}$, are its subalgebras. We call $\mathfrak L$ the \emph{differential algebra of logarithmic transseries}, and $\mathcal{L}_{k}$, $k\in \mathbb{N}$, the \emph{differential algebra of logarithmic transseries of depth} $k$. \\
			
			Throughout the paper, we will use the following acronyms:
			\begin{enumerate}[1., font=\textup, topsep=0.4cm, itemsep=0.4cm, leftmargin=0.6cm]
				\item $f = g + \mathrm{h.o.b.}(z)$ (which means: \textit{higher order blocks in $z$}) if $\mathrm{ord}_{z}(f - g)$ is strictly bigger than the order in $z$ of any block in $g$.
				\item $f = g + \mathrm{h.o.t.}$ (which means: \textit{higher order terms}) if $\mathrm{ord}(f - g)$ is strictly bigger than the order of any term in $g$.
			\end{enumerate}
			To shorten the notation, for $k\in \mathbb{N}_{\geq 1}$, we use the following multi-indeces:
			\begin{enumerate}[1., font=\textup, topsep=0.4cm, itemsep=0.4cm, leftmargin=0.6cm]
				\item $\mathbf{n}:=(n_{1},\ldots ,n_{k}) \in \mathbb{R}^{k}$,
				\item $\mathbf{1}_{k}:=(1,\ldots ,1)_{k} \in \mathbb{R}^{k}$.
			\end{enumerate}
		
			\emph{The composition of transseries.} Since $\log \boldsymbol{\ell}_{k}=-\boldsymbol{\ell }_{k+1}^{-1}$, the differential algebra $\mathcal{L}_{k}$ is not closed for formal composition. Therefore, we consider the set $\mathcal{L}_{k}^{H}$, $k\in \mathbb{N}$, (see \cite{prrs21}, \cite{mrrz16}) as the set of all logarithmic transseries $f \in \mathcal{L}_{k}$ that do not contain logarithms in their leading terms, i.e., such that $f=\lambda z^{\alpha } + \mathrm{h.o.t.}$, for $\lambda , \alpha >0$.
			
			Let $f\in \mathcal{L}_{k}^{\infty }$ and $g \in \mathcal{L}_{k}^{H}$. The composition $f \circ g$ is defined standardly, using the \emph{Taylor Theorem} as:
			\begin{align*}
				f \circ g & :=f\left( \lambda z^{\alpha }\right) + \sum_{i\geq 1}\frac{f^{(i)}\left( \lambda z^{\alpha }\right) }{i!}g_{1}^{i} ,
			\end{align*}
			where $g:=\lambda z^{\alpha }+g_{1}$ and $\mathrm{ord}(g_{1}) > (\alpha , \mathbf{0}_{k})$. The above series converges in the \emph{weak topology} on $\mathcal{L}_{k}$. For more details, see \cite[Proposition 3.3]{prrs21}.
			
			Denote by $\mathcal{L}_{k}^{0}\subseteq \mathcal{L}_{k}$ the set of all \emph{parabolic} (or \emph{tangent to the identity}) logarithmic transseries $f=z+\mathrm{h.o.t.}$ in $\mathcal{L}_{k}$, $k\in \mathbb{N}$. Now, put
			\begin{align*}
				& \mathfrak L^{H} :=\bigcup _{k\in \mathbb{N}}\mathcal{L}_{k}^{H} \quad \textrm{ and } \quad \mathfrak L^{0} :=\bigcup _{k\in \mathbb{N}}\mathcal{L}_{k}^{0} .
			\end{align*}
			Similarly as in \cite[Section 7]{Dries97} it can be shown that $\mathfrak L^{H}$ is a group for composition, and $\mathfrak L^{0}$, $\mathcal{L}_{k}^{H}$, $\mathcal{L}_{k}^{0}$, $k\in \mathbb{N}$, are its subgroups. \\
			
			In the Main Theorem we find the "shortest" normal form in $\mathcal{L}_{k}^{0}$ for a parabolic logarithmic transseries $f\in \mathcal{L}_{k}^{0}$, where $k\in \mathbb{N}$ is minimal such that $f\in \mathcal{L}_{k}^{0}$. In Corollary~\ref{Remark3} we obtain the "shortest" normal form of $f\in \mathfrak L^{0}$ in the larger group of all parabolic logarithmic transseries $\mathfrak L^{0}$.

			\subsection{The Main Theorem}\label{SectionMain}
			
			Let $f \in \mathcal{L}_{k}^{0}$, $k\in \mathbb{N}$. Let $\left( \beta ,\mathbf{n}\right) := \mathrm{ord}(f- \mathrm{id}) > (1,\mathbf{0}_{k})$. Similarly as in the case of parabolic analytic diffeomorphisms at zero (see e.g. \cite{CarlesonG93}, \cite{Milnor06}), or logarithmic transseries of depth $1$ (see \cite{mrrz16}), the monomial (term) in $f$ of order
			\begin{align*}
				& \left( 2\beta -1,2\mathbf{n}+\mathbf{1}_{k}\right) 
			\end{align*}
			is called the \emph{residual monomial (term)} of $f$ and denoted by
			\begin{align*}
				\mathrm{Res}(f) & := z^{2\beta -1}\boldsymbol{\ell }_{1}^{2n_{1}+1}\ldots \boldsymbol{\ell }_{k}^{2n_{k}+1} .
			\end{align*}
			Moreover, the coefficient of the residual term will be called the \emph{residual coefficient}. The block $z^{2\beta -1}R_{2\beta -1}$ of $f$ of order $2\beta -1$ (in $z$) is similarly called the \emph{residual block} of $f$.
			
			Note that the residual term of $f$ depends on the chosen $k\in \mathbb{N}$ such that $f\in \mathcal{L}_{k}^{0}$. That is, it depends on the depth in logarithms of the group of the parabolic changes of variables.
			
			The Main Theorem below is a generalization of Theorem A in \cite{mrrz16} from the group $\mathcal{L}_{1}^{0}$ to $\mathcal{L}_{k}^{0}$, $k\in \mathbb{N}$ (i.e., to the group of parabolic logarithmic transseries with iterated logarithms of the maximal depth $k$). \\
			
			For $1\leq v\leq k$ and $\mathbf{n} \in \left\lbrace 0\right\rbrace ^{v-1}\times \mathbb{N}_{\geq 1} \times \mathbb{Z} ^{k-v}$, we define
			\begin{align}
				\mathbf{n'} & := \mathbf{n} + (\mathbf{1}_{v}, \mathbf{0}_{k-v}) . \label{NCrtica}
			\end{align}
			\begin{example}\hfill 
				\begin{enumerate}[1., font=\textup, topsep=0.4cm, itemsep=0.4cm, leftmargin=0.6cm]
					\item If $\mathbf{n} = (0,\ldots , 0, 2,1,\ldots ,1)$, then $\mathbf{n'} = (1,\ldots ,1,3,1,\ldots ,1)$.
					\item If $\mathbf{n}=(3,2,\ldots 2)$, then $\mathbf{n'}=(4,2,\ldots ,2)$.
				\end{enumerate}
			\end{example}
			
			\begin{namedthm}{Main Theorem}\label{TeoremA}
				Let $f \in \mathcal{L}_{k}^{0}$, $k\in \mathbb{N}$, be a parabolic logarithmic transseries. Let $(\beta , \mathbf{n}) := \mathrm{ord}(f- \mathrm{id}) > (1,\mathbf{0}_{k})$ and let $\mathbf{n}'$ be defined as in \eqref{NCrtica}. For
				\begin{align*}
					f & := \mathrm{id} +\displaystyle \sum_{\mathbf{n} \leq \mathbf{m}\leq \mathbf{n'}}a_{\mathbf{m}}z^{\beta }\boldsymbol{\ell}_{1}^{m_{1}}\cdots \boldsymbol{\ell}_{k}^{m_{k}} + \mathrm{h.o.t.},
				\end{align*}
				where $\mathbf{m}:=(m_{1},\ldots ,m_{k})$, put:
				$$
					L :=\left\lbrace \begin{array}{ll}
					a_{\mathbf{n}}\boldsymbol{\ell}_{1}^{n_{1}}\cdots \boldsymbol{\ell}_{k}^{n_{k}} , & \textrm{if } \beta >1, \\[0.5cm]
					\displaystyle \sum_{\mathbf{n} \leq \mathbf{m}\leq \mathbf{n'}}a_{\mathbf{m}}\boldsymbol{\ell}_{1}^{m_{1}}\cdots \boldsymbol{\ell}_{k}^{m_{k}} , & \textrm{if } \beta =1,
					\end{array} \right.
				$$
				and let
				\begin{align}
					f_{c} &:=\mathrm{id}+z^{\beta }L+c\mathrm{Res}(f) , \quad c\in \mathbb{R} . \label{DefnOfFc}
				\end{align}
				\begin{enumerate}[1., font=\textup, topsep=0.4cm, itemsep=0.4cm, leftmargin=0.6cm]
					\item There exists $c\in \mathbb{R} $, such that the \emph{conjugacy equation}:
					\begin{align}
						\varphi  \circ f \circ \varphi ^{-1} = f_{c} \label{ConjEq}
					\end{align}
					has a solution $\varphi  \in \mathcal{L}_{k}^{0}$.
					\item Real number $c$ is unique such that \eqref{ConjEq} has a solution in $\mathcal{L}_{k}^{0}$. Moreover, the \emph{residual coefficient} $c$ is explicitely given by:
					\begin{align}
						& c = \left[ \frac{a_{\mathbf{n}}^{2}}{f-\mathrm{id}} \right] _{-1,\mathbf{1}_{k}} - \left[ \frac{a_{\mathbf{n}}^{2}}{zL} \right] _{-1,\mathbf{1}_{k}} = \left[ a_{\mathbf{n}}^{2}\int \frac{dz}{f-\mathrm{id}}\right] _{\mathbf{0}_{k+1},-1} - \left[ \frac{a_{\mathbf{n}}^{2}}{zL}\right] _{-1,\mathbf{1}_{k}} , \label{EqRes}
					\end{align}
					if $\beta =1$, and
					\begin{align}
						& c=\left[ \frac{a_{\mathbf{n}}^{2}}{f-\mathrm{id}}\right] _{-1,\mathbf{1}_{k}}=\left[ a_{\mathbf{n}}^{2}\int \frac{dz}{f-\mathrm{id}}\right] _{\mathbf{0}_{k+1},-1} , \label{EqRes2}
					\end{align}
					if $\beta >1$.
					\item The \emph{normal form} $f_{c}$ is minimal\footnote{In the sense that it cannot be further reduced by changes of variables from $\mathcal L_{k}^{0}$. Moreover, the coefficients of $f_{c}$ cannot be changed by changes of variables from $\mathcal{L}_{k}^{0}$.} in $\mathcal L_{k}^{0}$.
				\end{enumerate}
			\end{namedthm}
		
			\begin{remark}[The embedding in the formal flow]\label{Remark2}
				Let $f$ and $c$ be as in the Main Theorem. Put:
				$$
					X_{c} := \left\lbrace \begin{array}{ll}
						\frac{a_{\mathbf{n}}z^{\beta }\boldsymbol{\ell }_{1}^{n_{1}}\cdots \boldsymbol{\ell }_{k}^{n_{k}}}{1+\frac{a_{\mathbf{n}}}{2}\left( z^{\beta }\boldsymbol{\ell }_{1}^{n_{1}}\cdots \boldsymbol{\ell }_{k}^{n_{k}}\right) ' - \frac{c}{a_{\mathbf{n}}} z^{\beta -1}\boldsymbol{\ell }_{1}^{n_{1}+1}\cdots \boldsymbol{\ell }_{k}^{n_{k}+1}} \frac{d}{dz} , & \textrm{if } \beta >1, \\[0.5cm]
						\frac{z\log (1+L)}{1- \frac{c\mathrm{Res}(f)}{z\log (1+ L)} + \frac{\boldsymbol{\ell }_{1}^{2}\frac{d}{d\boldsymbol{\ell }_{1}}(L)}{2(1+ L)}} \frac{d}{dz} , & \textrm{if } \beta =1.
					\end{array} \right.
				$$
				There exists a parabolic logarithmic transseries $\psi \in \mathcal{L}_{k}^{0}$ that reduces $f$ to a normal form given as the \emph{time-one map} of the formal vector field $X_{c}$. That is, such that
				\begin{align*}
					\psi \circ f \circ \psi ^{-1} & =\exp (X_{c}).\mathrm{id} .
				\end{align*}
				This proves that $f$ can be formally embedded in a formal flow.
			\end{remark}
		
			\begin{remark}\label{Remark1}\hfill
				\begin{enumerate}[1., font=\textup, topsep=0.4cm, itemsep=0.4cm, leftmargin=0.6cm]
					\item (Varying $k$ in the Main Theorem) Note that the Main Theorem and Remark~\ref{Remark2} hold for every $k\in \mathbb{N}$, $k\geq k_{0}$, where $k_{0}\in \mathbb{N}$ is minimal such that $f\in \mathcal{L}_{k_{0}}^{0}$. Note that the initial part $\mathrm{id}+zL$ of $f_{c}$ remains the same by varying the parametar $k\geq k_{0}$, as opposed to the residual coefficient $c$. In the proof of Corollary~\ref{Remark3} in Section~\ref{sec:proofRemark} we show that there are only two in general distinct values of a residual coefficient $c$. One is obtained by normalization in $\mathcal{L}_{k_{0}}^{0}$ and the other by normalization in any larger group $\mathcal{L}_{k}^{0}$, for $k>k_{0}$. 
					\item (The formal class of logarithmic transseries) In the case $\beta >1$, the quadruple $(\beta , \mathbf{n} , a_{\mathbf{n}} , c)$ determines the formal class of the parabolic logarithmic transseries $f\in \mathcal{L}_{k}^{0}$, $k\in \mathbb{N}$.
					
					In the case $\beta =1$, the formal class of $f\in \mathcal{L}_{k}^{0}$, $k\in \mathbb{N}$, is determined by (possibly infinite) sequence $(1 , \mathbf{n} , (a_{\mathbf{m}})_{\mathbf{n}\leq \mathbf{m} \leq \mathbf{n}'} , c)$.
					\item (Non-uniqueness of a normalization) The normalizing change of variables $\psi $ from Remark~\ref{Remark2} is not unique, since every precomposition of $\psi $ with a \emph{time-}$t$ map, $t\in \mathbb{R}$, of the vector field $X_{c}$ from Remark~\ref{Remark2} also reduces $f$ to the time-one map of $X_{c}$. Since the time-one map of the formal vector field $X_{c}$ and the minimal normal form $f_{c}$ defined in \eqref{DefnOfFc} belong to the same formal class, it implies that the solution $\varphi $ of the conjugacy equation \eqref{ConjEq} is not unique.
					\item Note that \eqref{EqRes} is a generalization of the formula for the residual coefficient for parabolic logarithmic transseries in $\mathcal{L}_{1}^{0}$ obtained in \cite[Proposition 9.3]{mrrz19tubular}, to the differential algebra $\mathcal{L}_{k}$, $k\in \mathbb{N}$.
				\end{enumerate}
			\end{remark}
		
			\begin{cor}[The normal form in the larger differential algebra $\mathfrak L$]\label{Remark3}
				Let $f\in \mathfrak L^{0}$ and let $k\in \mathbb{N}$ be minimal such that $f\in \mathcal{L}_{k}^{0}$. By a parabolic change of variables $\varphi $ in the larger differential algebra $\mathfrak L^{0}$, a parabolic logarithmic transseries $f\in \mathcal{L}_{k}^{0}$, $k\in \mathbb{N}$, can be further reduced to
				\begin{align*}
					f_{0} & := \mathrm{id} + zL ,
				\end{align*}
				for $L$ as defined in the Main Theorem.
				
				Moreover, every normalization $\varphi $ belongs to $\mathcal{L}_{k+1}^{0}$ and $f_{0}$ is minimal\footnote{In the sense that it cannot be further reduced in $\mathfrak L^{0}$, nor can the coefficients be changed.} in $\mathfrak L^{0}$.
			\end{cor}
		
			The proof of the Main Theorem is in Section~\ref{sec:proofA} and the proofs of Remark~\ref{Remark2} and Corollary~\ref{Remark3} are in Section~\ref{sec:proofRemark}.
			
			\begin{example}[Application of the Main Theorem and Corollary~\ref{Remark3}]\hfill
				\begin{enumerate}[1., font=\textup, topsep=0.4cm, itemsep=0.4cm, leftmargin=0.6cm]
					\item Let $f:=z+a_{1}z\boldsymbol{\ell }_{1}+ a_{2}z\boldsymbol{\ell }_{1}^{2}+a_{3}z\boldsymbol{\ell }_{1}^{3}+\mathrm{h.o.t.}\in \mathcal{L}_{1}$, for $a_{1},a_{2},a_{3}\in \mathbb{R}\setminus \left\lbrace 0\right\rbrace $. Here, $L=a_{1}\boldsymbol{\ell }_{1}+a_{2}\boldsymbol{\ell }_{1}^{2}$ and $\mathrm{Res}(f)=z\boldsymbol{\ell }_{1}^{3}$. By \eqref{EqRes}, it follows that:
					\begin{align*}
						& c:=\left[ \frac{a_{1}^{2}}{f-\mathrm{id}}\right] _{-1,1}-\left[ \frac{a_{1}^{2}}{zL}\right] _{-1,1} =-a_{3}+\frac{a_{2}^{2}}{a_{1}} - \frac{a_{2}^{2}}{a_{1}}=-a_{3} .
					\end{align*}
					Therefore, by the Main Theorem, it follows that
					\begin{align*}
						f_{c} & =\mathrm{id}+a_{1}z\boldsymbol{\ell }_{1}+a_{2}z\boldsymbol{\ell}_{1}^{2}+ \left( -a_{3}+\frac{a_{2}^{2}}{a_{1}} -\left( \frac{a_{2}}{a_{1}}\right) ^{2}\right) z\boldsymbol{\ell }_{1}^{3}
					\end{align*}
					is the normal form in $\mathcal{L}_{1}$ of the logarithmic transseries $f$. However, if we allow normalizations $\varphi $ from the larger group $\mathcal{L}_{2}^{0}$, we get just $f_{0}=\mathrm{id}+a_{1}z\boldsymbol{\ell }_{1}+a_{2}z\boldsymbol{\ell}_{1}^{2}$.
					\item Let $f:=z+a_{1}z^{2}\boldsymbol{\ell }_{2}+a_{2}z^{2}\boldsymbol{\ell }_{1}+a_{3}z^{3}\boldsymbol{\ell }_{1}\boldsymbol{\ell }_{2}^{2}+z^{4}+\mathrm{h.o.t.}\in \mathcal{L}_{2}$, for $a_{1},a_{2},a_{3}\in \mathbb{R}\setminus \left\lbrace 0\right\rbrace $. Note that $L=a_{1}\boldsymbol{\ell }_{2}$ and $\mathrm{Res}(f)=z^{3}\boldsymbol{\ell }_{1}\boldsymbol{\ell }_{2}^{3}$. By \eqref{EqRes2}, it follows that $c=\left[ \frac{a_{1}^{2}}{f-\mathrm{id}}\right] _{-1,\mathbf{1}_{2}}=0$. Therefore, by the Main Theorem, it follows that $f_{c}=\mathrm{id}+a_{1}z^{2}\boldsymbol{\ell }_{2}$ is the normal form in $\mathcal{L}_{2}$ of the logarithmic transseries $f$.
				\end{enumerate}
			\end{example}

		\section{Preliminaries to the proof of the Main Theorem}\label{SectionPrereqForMThm}
		
			This section contains prerequisites for the proof of the Main Theorem, Remark~\ref{Remark2} and Corollary~\ref{Remark3}. In final Subsection~\ref{SubsectionFixedPointThms} we state a fixed point theorem, originally stated in \cite[Proposition 3.2]{prrs21}, which is the main ingredient of the proof of the Main Theorem.

			\subsection{Metric on differential algebra $\mathfrak L^{\infty }$}\label{subsec:topology}
				
				As in \cite{prrs21}, we define the \emph{power-metric} $d_{z}$ on differential algebra $\mathfrak L^{\infty }$ by:
				$$
					d_{z}(f , g) :=\left\lbrace \begin{array}{ll}
						2^{-\mathrm{ord}_{z}(f-g)}, & f \neq g, \\ [1mm]
						0, & f = g .
					\end{array} \right.  
				$$
				We denote by $\mathcal{T}_{d_{z}}$ the induced \emph{power-metric topology}, which is standardly called the \emph{valuation topology} (see e.g. \cite{Dries97}) or the \emph{formal topology} in \cite{mrrz16}. \\
				
				Let $W\subseteq \mathbb{R}\times \mathbb{Z}^{k}$, $k\in \mathbb{N}$, and let $\mathcal L_{k}^{W}$ be the set of all $f\in \mathcal{L}_{k}^{\infty }$ such that $\mathrm{Supp} (f)\subseteq W$. Note that $\mathcal{L}_{k}^{W}$ is a linear subspace of $\mathcal{L}_{k}^{\infty }$. By \cite[Proposition 3.6]{prrs21}, it follows that $(\mathcal{L}_{k}^{W},d_{z})$ is complete, for every $W\subseteq \mathbb{R}\times \mathbb{Z}^{k}$ and $k\in \mathbb{N}$. In particular, $\mathcal{L}_{k}^{\infty }$ and $\mathcal{L}_{k}$, $k\in \mathbb{N}$, are complete.

			\subsection{Differential algebras of blocks $\mathcal{B}_{m}$}\label{SubsectionSpaces}
			
			In this subsection we recall the \emph{differential algebras of blocks} from \cite[Subsection 3.4]{prrs21}.
			
			Let $k\in \mathbb{N}_{\geq 1}$ be fixed. For $1\leq m\leq k$, we define $\mathcal{B}_{m}\subseteq \mathcal{L}_{k}^{\infty }$ to be the set of all $K\in \mathcal{L}_{k}^{\infty }$ of the form
			\begin{align*}
				K & := \sum_{\left( n_{m},\ldots ,n_{k}\right) \in \mathbb{Z}^{k-m+1}} a_{n_{m},\ldots ,n_{k}}\boldsymbol{\ell}_{m}^{n_{m}}\cdots \boldsymbol{\ell}_{k}^{n_{k}} .
			\end{align*}
			We call such $K$ a \emph{block of level} $m$. It is easy to see that $\mathcal{B}_{m}$, $m\in \mathbb{N}_{\geq 1}$, is a linear subspace that depends on the \emph{ambient space} $\mathcal{L}_{k}$, $k\in \mathbb{N}_{\geq 1}$. In the sequel we will write $\mathcal{B}_{m}\subseteq \mathcal{L}_{k}^{\infty }$ to emphasize which ambient space is considered. \\
			
			We define the \emph{order} of $K\in \mathcal{B}_{m}\setminus \left\lbrace 0\right\rbrace $ \emph{in} $\boldsymbol{\ell}_{m}$ as the minimal exponent of $\boldsymbol{\ell }_{m}$ in the block $K$, and denote it by $\mathrm{ord}_{\boldsymbol{\ell }_{m}}(K)$. If $K=0$, then we put $\mathrm{ord}_{\boldsymbol{\ell }_{m}}(K):=\infty $, where $\infty $ is the maximum of the set $\mathbb{Z}\cup \left\lbrace \infty \right\rbrace $. Moreover, for $K\neq 0$ and $K:=\sum _{n=n_{0}}^{+\infty }\boldsymbol{\ell }_{m}^{n}K_{n}\in \mathcal{B}_{m}$, $K_{n_{0}} \neq 0$, we call $K_{n_{0}}$ the \emph{leading block in} $\boldsymbol{\ell }_{m}$ of $K$ and denote it by $\mathrm{Lb}_{\boldsymbol{\ell }_{m}}(K)$. \\
			
			The function $d_{m}:\mathcal{B}_{m} \times \mathcal{B}_{m} \to \mathbb{R} $, such that:
			$$
				d_{m} \big( K_{1}, K_{2} \big) := \left\lbrace \begin{array}{ll}
					2^{-\mathrm{ord}_{\boldsymbol{\ell}_{m}}(K_{1}-K_{2})} , & K_{1} \neq K_{2}, \\
					0, & K_{1}=K_{1} ,
				\end{array} \right. 
			$$
			is a metric on $\mathcal{B}_{m}$, for $1\leq m\leq k$. It is easy to see that $\mathcal{B}_{m}$ is complete with respect to the metric $d_{m}$, and that $\mathcal{B}_{m+1} \subseteq \mathcal{B}_{m}$, for every $1\leq m\leq k-1$.
			
			Let
			\begin{align*}
				\mathcal{B}_{m}^{+} & :=\left\lbrace R \in \mathcal{B}_{m}\subseteq \mathcal{L}_{k}^{\infty } : \mathrm{ord}_{\boldsymbol{\ell}_{m}}(R)\geq 1 \right\rbrace ,
			\end{align*}
			and
			\begin{align*}
				\mathcal{B}_{\geq m}^{+} & :=\left\lbrace R \in \mathcal{B}_{m}\subseteq \mathcal{L}_{k}^{\infty } : \mathrm{ord}(R) > \mathbf{0}_{k+1} \right\rbrace ,
			\end{align*}
			for $1\leq m\leq k$. The spaces $\mathcal{B}_{\geq m}^{+}$ and $\mathcal{B}_{m}^{+}$ are complete, for $1\leq m\leq k$. Moreover, note that $\mathcal{B}_{m}^{+}\subseteq \mathcal{B}_{\geq m}^{+}\subseteq \mathcal{L}_{k}$, for $1\leq m\leq k$.
			
			The space $\mathcal{B}_{\geq m}^{+}$, $1\leq m\leq k$, can be decomposed as a direct sum of its complete subspaces:
			\begin{align*}
				\mathcal{B}_{\geq m}^{+} &= \mathcal{B}_{m}^{+} \oplus \cdots \oplus \mathcal{B}_{k-1}^{+} \oplus \mathcal{B}_{k}^{+} .
			\end{align*}
			\begin{example}
				Let $R:=5\boldsymbol{\ell}_{2}\boldsymbol{\ell}_{3}^{4}+\boldsymbol{\ell}_{2}^{3}\boldsymbol{\ell}_{3}^{-1} + \boldsymbol{\ell }_{1}^{2}\boldsymbol{\ell}_{2} \in \mathcal{B}_{\geq 1}^{+}\subseteq \mathcal{L}_{3}$. We have a decomposition $R=R_{1}+R_{2}+R_{3}$, where $R_{1}:= \boldsymbol{\ell }_{1}^{2}\boldsymbol{\ell}_{2} \in \mathcal{B}_{1}^{+}$, $R_{2}:=5\boldsymbol{\ell}_{2}\boldsymbol{\ell}_{3}^{4}+\boldsymbol{\ell}_{2}^{3}\boldsymbol{\ell}_{3}^{-1} \in \mathcal{B}_{2}^{+}$, and $R_{3}:=0$.
			\end{example}
			We define $D_{m}:\mathcal{B}_{m} \to \mathcal{B}_{m}$, $1\leq m\leq k$, such that
			\begin{align*}
				D_{m} & :=\boldsymbol{\ell}_{m}^{2} \cdot \frac{d}{d\boldsymbol{\ell}_{m}} .
			\end{align*}
			$D_{m}$ is a derivation and a $\frac{1}{2}$-contraction on the space $\left( \mathcal{B}_{m},d_{m}\right) $. The contractibility of the derivation $D_{m}$ is the main reason why we use $D_{m}$ on $\mathcal{B}_{m}$ instead of the usual derivation $\frac{d}{dz}$. Therefore, $\left( \mathcal{B}_{m},D_{m} \right) $ is a differential algebra and $\mathcal{B}_{m}^{+}, \mathcal{B}_{\geq m}^{+}$ are its subalgebras. Note that
			\begin{align*}
				& D_{m}\left( \mathcal{B}_{m+n}\right) \subseteq \mathcal{B}_{m}^{+} ,
			\end{align*}
			for every $0\leq m\leq k-1$ and $1\leq n\leq k-m$. \\
			
			For $R ,T \in \mathcal{B}_{m}$, $1\leq m\leq k$, we write:
			\begin{align*}
				R & = T + \mathrm{h.o.b.} \left( \boldsymbol{\ell}_{m} \right) ,
			\end{align*}
			(which means: \textit{higher order blocks in} $\boldsymbol{\ell}_{m}$) if $\mathrm{ord}_{\boldsymbol{\ell}_{m}}(R  - T)$ is strictly bigger than the order in $\boldsymbol{\ell}_{m}$ of any term in $T $. \\
			
			Note finally that there are two different objects named \emph{block}. We call $z^{\beta }R_{\beta }$, $\beta \geq 1$, a \emph{block of order} $\beta $. But also $R_{\beta }$, as an element of the space $\mathcal{B}_{1}$, is called a \emph{block of level} $1$.

			\subsection{A list of useful identities}\label{ListIdentities}
			In this subsection we state some important identites used for computing in differential algebras $\mathcal{L}_{k}$, $\mathcal{B}_{1}, \mathcal{B}_{1}^{+}$ and $\mathcal{B}_{\geq 1}^{+}$, which we use in the proof of the Main Theorem, see \cite[Subsection 3.6]{prrs21}. Proving these identities is elementary, so we state them without proofs.
			\begin{enumerate}[\(\bullet\), font=\textup, topsep=0.4cm, itemsep=0.4cm, leftmargin=0.6cm]
				\item For $K\in \mathcal{B}_{1}$ and $\alpha \neq 1$,
				\begin{align}
					(z^{\alpha }K)'&=z^{\alpha -1}(\alpha K+D_{1}(K)) , \nonumber \\
					(z^{\alpha }K)''&=z^{\alpha -2}(\alpha (\alpha -1) K+ \mathcal{C}_{2}(K)) , \nonumber \\
					(z^{\alpha }K)^{(i)}&=z^{\alpha -i}(\alpha (\alpha -1) \cdots (\alpha -i+1)K+\mathcal{C}_{i}(K)) , \,  i\in \mathbb{N}_{\geq 3}, \label{IdentDer1}
				\end{align}
				where $\mathcal{C}_{i}:\mathcal{B}_{1}\to \mathcal{B}_{1}$, $i\in \mathbb{N}_{\geq 1}$, are $\frac{1}{2}$-contractions on the space $(\mathcal{B}_{1},d_{1})$.
				
				\item For $K\in \mathcal{B}_{1}$ and $Q\in \mathcal{B}_{\geq 1}^{+}$,
				\begin{align}
					\sum_{i\geq 1}\frac{(z^{\alpha }K)^{(i)}}{i!} (zQ) ^{i} &=  z^{\alpha } \cdot \Big(K\cdot \sum_{i\geq 1}{\alpha \choose i} Q^{i} + \mathcal{C}(Q,K) \Big) , \label{Identitet1}
				\end{align}
				where $\mathcal{C}(Q,\cdot ):\mathcal{B}_{1}\to \mathcal{B}_{1}$ is a linear $\frac{1}{2^{1+\mathrm{ord}_{\boldsymbol{\ell }_{1}}(Q)}}$-contraction, for each $Q\in \mathcal{B}_{\geq 1}^{+}$, and $\mathcal{C}(\cdot ,K):\mathcal{B}_{\geq 1}^{+}\to \mathcal{B}_{1}$ is $\frac{1}{2^{1+\mathrm{ord}_{\boldsymbol{\ell }_{1}}(K)}}$-Lipschitz map (see Definition~\ref{DefinitionHomothetyLipschitz}), for each $K\in \mathcal{B}_{1}$, with respect to metric $d_{1}$.
				
				\item \begin{align}
					D_{m}(\boldsymbol{\ell }_{m}^{n}L_{m+1}) & = \boldsymbol{\ell }_{m}^{n+1}(nL_{m+1}+D_{m+1}(L_{m+1})) , \label{EqDer1}
				\end{align}
				where $L_{m+1}\in \mathcal{B}_{\geq m+1}^{+}$, $1\leq m\leq k-1$, and $n\in \mathbb{Z}$.
				
				\item \begin{align}
					D_{m}(K) & = \boldsymbol{\ell }_{m}\cdots \boldsymbol{\ell }_{m+n-1}D_{m+n}(K) , \label{EqDer2}
				\end{align}
				for $K\in \mathcal{B}_{m+n}$, $1\leq m\leq k-1$ and $1\leq n\leq k-m$.
				
				\item For $K\in \mathcal{B}_{\geq 1}^{+}$,
				\begin{align}
					(zK)'&=K+D_{1}(K) , \nonumber \\
					(zK)''&=z^{-1}(D_{1}(K)+ \mathcal{C}_{2}(K)) , \nonumber \\
					(zK)^{(i)}&=z^{-(i-1)}((-1)^{i}(i-2)!D_{1}(K)+\mathcal{C}_{i}(K)) , \,  i\in \mathbb{N}_{\geq 3}, \label{FormulaDerivationFirstBlock}
				\end{align}
				where $\mathcal{C}_{i}:\mathcal{B}_{\geq 1}^{+}\to \mathcal{B}_{\geq 1}^{+}$, $i\in \mathbb{N}_{\geq 2}$, are $\frac{1}{4}$-contractions on the space $(\mathcal{B}_{\geq 1}^{+},d_{1})$.
				
				\item For $K,Q\in \mathcal{B}_{\geq 1}^{+}$,
				\begin{align}
					\sum_{i\geq 1}\frac{(zK)^{(i)}}{i!} (zQ) ^{i} &=  z \cdot \Big(K\cdot Q+D_{1}(K)\cdot \Big( Q+\sum_{i\geq 2}\frac{(-1)^{i}}{(i-1)i}Q^{i} \Big) + \mathcal{C}(Q,K) \Big) \nonumber \\
					& = z\cdot \left( K\cdot Q+D_{1}(K)\cdot (1+Q)\log (1+Q) + \mathcal{C}(Q,K) \right) , \label{Identitet2}
				\end{align}
				where $\mathcal{C}(Q,\cdot ):\mathcal{B}_{\geq 1}^{+}\to \mathcal{B}_{\geq 1}^{+}$ is a linear $\frac{1}{2^{1+\mathrm{ord}_{\boldsymbol{\ell }_{1}}(Q)}}$-contraction, for each $Q\in \mathcal{B}_{\geq 1}^{+}$, and $\mathcal{C}(\cdot ,K):\mathcal{B}_{\geq 1}^{+}\to \mathcal{B}_{\geq 1}^{+}$ is a $\frac{1}{2^{2+\mathrm{ord}_{\boldsymbol{\ell }_{1}}(K)}}$-contraction, for each $K\in \mathcal{B}_{\geq 1}^{+}$, with respect to metric $d_{1}$.
			\end{enumerate}

			\subsection{Fixed point theorem}\label{SubsectionFixedPointThms}
			In the proof of the Main Theorem we use Proposition~\ref{KorBanach} below, which gives a fixed point theorem motivated by the \emph{Krasnoselskii Fixed Point Theorem} (see e.g. \cite{XiangGeorgiev16}), and which follows directly by the Banach Fixed Point Theorem. Proposition~\ref{KorBanach} is proven (very elementary) in \cite[Proposition 3.2]{prrs21}. We first recall some basic notions.
			\begin{defn}[Homothety]\label{DefinitionHomothetyLipschitz}\cite[Definition 3.1]{prrs21}
				Let $\lambda,\mu>0$ and let $(X,d)$, $(Y,\rho)$ be two metric
				spaces. 
					
				A map $\mathcal{T}:X\to Y$ such that 
				\[
					\rho\left(\mathcal{T}(x),\mathcal{T}(y)\right)=\lambda d(x,y),\quad\forall x,y\in X,
				\]
				is called a \emph{$\lambda$-homothety}.
					
				A map $\mathcal{S}:X\longrightarrow Y$ such that
				\[
					\rho\left(\mathcal{S}\left(x\right),\mathcal{S}\left(y\right)\right)\leq\mu d\left(x,y\right),\quad\forall x,y\in X,
				\]
				is called a \emph{$\mu$-Lipschitz map} and $\mu $ is called a \emph{Lipschitz coefficient} of $\mathcal{S}$.
				
				We say that $\mu $ is the \emph{minimal Lipschitz coefficient} of $\mathcal{S}$ if $\mu $ is the smallest Lipschitz coefficient of $\mathcal{S}$.
				
				In particular, if $\mu <1$, we call $\mathcal{S}$ a $\mu $\emph{-contraction}.
			\end{defn}
		
			Note that every $\lambda $-homothety is injective, and, if $\lambda =1$, it is an isometry.
			
			\begin{prop}[A fixed point theorem]\cite[Proposition 3.2]{prrs21}\label{KorBanach}
				Let $X$, $Y$ be two metric spaces and $X$ complete. Let $\mathcal{S},\,\mathcal{T}:X\to Y$, such that:
				\begin{enumerate}[1., font=\textup, topsep=0.4cm, itemsep=0.4cm, leftmargin=0.6cm]
					\item $\mathcal{S}$ is a $\mu $-Lipschitz map,
					\item $\mathcal{T}$ is a $\lambda $-homothety,
					\item $\mu < \lambda$,
					\item $\mathcal S(X)\subseteq \mathcal T(X)$. 
				\end{enumerate}
				Then there exists a unique point $x\in X$ such that $\mathcal{T}(x)=\mathcal{S}(x)$.
				
				Moreover, $x$ is the limit of a \emph{Picard sequence} $\left( \left( \mathcal{T}^{-1}\circ \mathcal{S}\right) ^{\circ n}(z)\right) _{n}$, for any \emph{initial point} $z\in X$.
			\end{prop}

		\section{Proof of the Main Theorem}\label{sec:proofA}
		
			\subsection{Sketch of the proof of statement 1 of the Main Theorem}\label{SubsectionSketchMainThm}
				
				Let $f,g\in \mathcal{L}_{k}^{0}$, $k\in \mathbb{N}$, be parabolic logarithmic transseries such that the following conjugacy equation is satisfied:
				\begin{align}
					\varphi  \circ f \circ \varphi ^{-1} & = g, \label{JednadzbaPrva}
				\end{align}
				for some $\varphi \in \mathcal{L}_{k}^{0}$. In particular, we are interested in $g$ to be the "shortest" normal form for $f$, but we do not know the normal form in advance. Therefore, we consider an arbitrary $g \in \mathcal{L}_{k}^{0}$ and we expect that the normal form will be revealed while solving the equation \eqref{JednadzbaPrva}. In order to apply the fixed point theorem from Proposition~\ref{KorBanach}, we transform equation \eqref{JednadzbaPrva} to an equivalent \emph{fixed point equation}.
				
				\subsubsection{Transformation of the conjugacy equation to a fixed point equation}
				
				\begin{lem}[Transformation of the conjugacy equation to a fixed point equation]\label{Lemma2}
					Let $f,g\in \mathcal{L}_{k}^{0}$, $k\in \mathbb{N}$, such that $f=\mathrm{id}+z^{\beta }R_{\beta } + \mathrm{h.o.b.}(z) $ and $g=\mathrm{id}+z^{\alpha }S_{\alpha }+\mathrm{h.o.b.}(z)$, for $\alpha , \beta \geq 1$, and $S_{\alpha }, R_{\beta }\in \mathcal{B}_{1}\subseteq \mathcal{L}_{k}^{\infty }$. A logarithmic transseries $\varphi \in \mathcal{L}_{k}^{0}$, $\varphi=\mathrm{id}+\varepsilon $, $\mathrm{ord}(\varepsilon)>(1,\mathbf{0}_{k})$, satisfies the \emph{conjugacy equation}
					$$
					\varphi  \circ f \circ \varphi ^{-1} = g
					$$
					if and only if the logarithmic transseries $\varepsilon \in \mathcal{L}_{k}$ satisfies the \emph{fixed point equation}:
					\begin{align}
						\mathcal{S}_{f}(\varepsilon ) & = \mathcal{T}_{f}(\varepsilon ), \label{FixedPointEq}
					\end{align}
					where the operators $\mathcal{S}_{f} , \mathcal{T}_{f}: \mathcal{L}_{k} \to \mathcal{L}_{k}$ are defined as:
					\begin{align}
						\mathcal{S}_{f}(\varepsilon ) & := \psi  - \mu  - \varepsilon ' \cdot \mu _{1} + \psi _{1}' \cdot \varepsilon  - \sum_{i\geq 2}\frac{\varepsilon  ^{(i)}}{i!}\mu ^{i} + \sum_{i\geq 2}\frac{\psi ^{(i)}}{i!}\varepsilon  ^{i} , \label{DefS}
					\end{align}
					and
					\begin{align}
						\mathcal{T}_{f}(\varepsilon ) & := \varepsilon ' \cdot z^{\beta }R_{\beta } - (z^{\alpha }S_{\alpha })' \cdot \varepsilon  , \label{Tdefined}
					\end{align}
					for $\varepsilon  \in \mathcal{L}_{k}$. Here,
					\begin{align}
						\mu :=f- \mathrm{id}, \quad & \mu _{1}:=f - \mathrm{id} - z^{\beta }R_{\beta } , \nonumber \\
						\psi :=g-\mathrm{id} , \quad & \psi _{1}:=g-\mathrm{id} -z^{\alpha }S_{\alpha } . \label{NotationLem} 
					\end{align}
				\end{lem}
				\begin{proof}
					Conjugacy equation \eqref{JednadzbaPrva} is equivalent to the equation
					\begin{align}
						\varphi  \circ f - g \circ \varphi  & =0. \label{Equation1}
					\end{align}
					Recall the notation from \eqref{NotationLem}. From the equation \eqref{Equation1}, using the Taylor Theorem (see \cite[Proposition 3.3]{prrs21}), we get that:
					\begin{align}
						0 &= \varphi \circ f- g \circ \varphi  \nonumber \\
						&= \left( z+\varepsilon \right) \circ \left( z+\mu \right) - (z+\psi )\circ \left( z+ \varepsilon\right) \nonumber \\
						&= z+\mu  + \varepsilon +\sum_{i\geq 1}\frac{\varepsilon  ^{(i)}}{i!}\mu ^{i}-\left( z+\varepsilon \right) - \psi  - \sum_{i\geq 1}\frac{\psi ^{(i)}}{i!}\varepsilon  ^{i} \nonumber \\
						&= \sum_{i\geq 1}\frac{\varepsilon  ^{(i)}}{i!}\mu ^{i} +\mu - \psi  - \sum_{i\geq 1}\frac{\psi ^{(i)}}{i!}\varepsilon  ^{i} \nonumber \\
						&= \varepsilon ' \cdot \mu  - \psi ' \cdot \varepsilon  + \sum_{i\geq 2}\frac{\varepsilon  ^{(i)}}{i!}\mu ^{i} + \mu  - \psi  - \sum_{i\geq 2}\frac{\psi ^{(i)}}{i!}\varepsilon  ^{i} \nonumber \\
						&= \varepsilon ' \cdot z^{\beta }R_{\beta } - \left( z^{\alpha }S_{\alpha } \right) ' \cdot \varepsilon + \mu  - \psi  + \varepsilon ' \cdot \mu _{1} - \psi _{1}' \cdot \varepsilon + \sum_{i\geq 2}\frac{\varepsilon  ^{(i)}}{i!}\mu ^{i} - \sum_{i\geq 2}\frac{\psi ^{(i)}}{i!}\varepsilon  ^{i} . \label{Equation2} 
					\end{align}
					Equation \eqref{Equation2} is equivalent to the following equation:
					\begin{align}
						\varepsilon ' \cdot z^{\beta }R_{\beta } - \left( z^{\alpha }S_{\alpha } \right) ' \cdot \varepsilon  & = \psi  - \mu  - \varepsilon ' \cdot \mu _{1} + \psi _{1}' \cdot \varepsilon  - \sum_{i\geq 2}\frac{\varepsilon  ^{(i)}}{i!}\mu ^{i} + \sum_{i\geq 2}\frac{\psi ^{(i)}}{i!}\varepsilon  ^{i} , \label{EquationLema2} 
					\end{align}
					i.e., $\mathcal{T}_{f}(\varepsilon ) = \mathcal{S}_{f}(\varepsilon )$.
				\end{proof}
			
				\begin{remark}
					In Lemma~\ref{Lemma2} we transformed conjugacy equation \eqref{JednadzbaPrva} into the equivalent equation $\mathcal{T}_{f}(\varepsilon ) = \mathcal{S}_{f}(\varepsilon )$, where $\mathcal{T}_{f}$ is a linear operator which is a \emph{block-wise} analogue to the \emph{Lie bracket operator} used in \cite{mrrz16} for \emph{term-wise} reduction of $f$ to $g$.
				\end{remark}
				
				Let $\mathcal{L}_{k}^{\delta }$, for $\delta \geq 1$, $k\in \mathbb{N}$, be the set of all logarithmic transseries $f\in \mathcal{L}_{k}$ such that $\mathrm{ord}_{z}(f)\geq \delta $ and $\mathrm{ord}(f)>(1,\mathbf{0}_{k})$. By the discussion in Subsection~\ref{subsec:topology}, it follows that $(\mathcal{L}_{k}^{\delta },d_{z})$ is a complete metric space. \\
				
				To be able to apply a fixed point theorem in Proposition~\ref{Lema3} below, we analyse the Lipschitz property of the operator $\mathcal{S}_{f}$ in equation \eqref{FixedPointEq}. We first state Lemma~\ref{Lema2} that will be used in the proof of Proposition~\ref{Lema3}. The proof is in the Appendix.
				
				\begin{lem}\label{Lema2}
					Let $\varphi  \in \mathcal{L}_{k}^{\alpha }$, for $\alpha \geq 1$ and $k\in \mathbb{N}$. Let the operator $\mathcal{S}:\mathcal{L}_{k}^{\delta } \to \mathcal{L}_{k}^{\delta }$, $\delta \geq 1$, be defined as:
					\begin{align}
						\mathcal{S}\left( \varepsilon \right) & :=\sum _{i\geq 2}\varphi ^{(i)} \cdot \varepsilon ^{i} , \label{LemmaDefinitionS}
					\end{align}
					for $\varepsilon  \in \mathcal{L}_{k}^{\delta }$. The operator $\mathcal{S}$ is $\frac{1}{2^{\delta + \alpha -2}}$-Lipschitz on the space $(\mathcal{L}_{k}^{\delta },d_{z})$.
					
					Moreover, $\frac{1}{2^{\delta + \alpha -2}}$ is the minimal Lipschitz coefficient of $\mathcal{S}$. 
				\end{lem}
				
				\begin{prop}[Properties of operator $\mathcal{S}_{f}$]\label{Lema3}
					Suppose that $f,g\in \mathcal{L}_{k}^{0}$, $k\in \mathbb{N}$, such that $f=\mathrm{id}+z^{\beta }R_{\beta } + \mathrm{h.o.b.}(z) $ and $g=\mathrm{id}+z^{\alpha }S_{\alpha }+\mathrm{h.o.b.}(z)$, for $\alpha , \beta \geq 1$, and \eqref{NotationLem}. Let $\delta \geq 1$ and put:
					\begin{align}
						\gamma & := \min \left\lbrace \mathrm{ord}_{z}(f- \mathrm{id} - z^{\beta }R_{\beta } ) , \mathrm{ord}_{z}(g-\mathrm{id} - z^{\alpha }S_{\alpha} )\right\rbrace , \nonumber \\
						\rho (\delta)  & := \min \left\lbrace \gamma -1 , 2(\beta -1) , \delta + \alpha -2\right\rbrace . \label{Bound}
					\end{align}
					Let $\mathcal{S}_{f}:\mathcal{L}_{k}^{1} \to \mathcal{L}_{k}^{1}$ be the operator defined as in \eqref{DefS}.
					\begin{enumerate}[1., font=\textup, topsep=0.4cm, itemsep=0.4cm, leftmargin=0.6cm]
						\item The operator $\mathcal{S}_{f}$ is $\frac{1}{2^{\rho (\delta)}}$-Lipschitz\footnote{In the sense that $\frac{1}{2^{\rho (\delta)}}$ is the minimal Lipschitz coefficient of $\mathcal{S}_{f}$.} on the space $\mathcal{L}_{k}^{\delta }$, for every $\delta \geq 1$.
						\item If $z^{\alpha }S_{\alpha }=z^{\beta }R_{\beta }$ and $\delta = \gamma - \beta +1$, then $\mathcal{S}_{f}(\mathcal{L}_{k}^{\delta }) \subseteq \mathcal{L}_{k}^{\gamma }$ and \\
						$\rho (\delta ) = \min \left\lbrace \gamma -1,2(\beta -1)\right\rbrace $.
					\end{enumerate}
				\end{prop}
				\begin{proof}
					1. Let $\delta \geq 1$ and $\varepsilon \in \mathcal{L}_{k}^{\delta }$. By \eqref{NotationLem} and \eqref{Bound}, it follows that \\
					$\mathrm{ord}_{z}\left( \mu _{1}\right) , \mathrm{ord}_{z}\left( \psi _{1}\right) \geq \gamma $. For the linear parts of $\mathcal{S}_{f}$, from \eqref{NotationLem} and \eqref{Bound}, we have the following bounds:
					\begin{align}
						& \mathrm{ord}_{z}(- \varepsilon ' \mu _{1} + \psi _{1}' \varepsilon ) = \mathrm{ord}_{z}(\varepsilon )+\gamma -1, \nonumber \\
						& \mathrm{ord}_{z} \Big(-\sum_{i\geq 2}\frac{\varepsilon  ^{(i)}}{i!}\mu ^{i}\Big) = \mathrm{ord}_{z}(\varepsilon )+2(\beta -1). \label{Ocjena1} 
					\end{align}
					For the nonlinear part of $\mathcal{S}_{f}$ we use Lemma~\ref{Lema2}. By \eqref{DefS}, Lemma~\ref{Lema2}, \eqref{Bound}, and \eqref{Ocjena1} we conclude that $\mathcal{S}_{f}$ is $\frac{1}{2^{\rho (\delta)}}$-Lipschitz and $\frac{1}{2^{\rho (\delta)}}$ is the minimal Lipschitz coefficient of $\mathcal{S}_{f}$. \\
					
					2. From $z^{\alpha}S_{\alpha}=z^{\beta }R_{\beta }$ and $\delta = \gamma - (\beta -1)$, we conclude that $\delta + \alpha -2=\gamma -1$. Consequently, it follows that
					\begin{align*}
						\rho (\delta ) &= \min \left\lbrace \gamma -1,2(\beta -1)\right\rbrace .
					\end{align*}
					Note that $\rho (\delta ) + \delta \geq \gamma $. This and $\mathrm{ord}_{z}(\psi - \mu ) \geq \gamma $ implies that $\mathcal{S}_{f}(\mathcal{L}_{k}^{\delta }) \subseteq \mathcal{L}_{k}^{\gamma }$.
				\end{proof}
				
				\subsubsection{A necessary condition for solvability of conjugacy equation in $\mathfrak{L}^{0}$}
				
				In the following proposition we give a necessary condition on $g$, such that conjugacy equation \eqref{JednadzbaPrva} is solvable in $\mathfrak{L}^{0}$.
				
				\begin{prop}\label{LeadingBlockLemma}
					Let $f,g\in \mathcal{L}_{k}^{0}$, $k\in \mathbb{N}$, be such that $f=\mathrm{id}+z^{\beta }R_{\beta } + \mathrm{h.o.b.}(z) $, and $g=\mathrm{id}+z^{\alpha }S_{\alpha }+\mathrm{h.o.b.}(z)$, where $\alpha , \beta \geq 1$. Let $(\beta ,\mathbf{n}):=\mathrm{ord}(f- \mathrm{id})$ and let $\mathbf{n}'$ be as in \eqref{NCrtica}. Let $\varphi  \in \mathfrak L^{0}$, such that $\varphi  \circ f \circ \varphi ^{-1} = g$. Then:
					\begin{enumerate}[1., font=\textup, topsep=0.4cm, itemsep=0.4cm, leftmargin=0.6cm]
						\item $\mathrm{Lt}(f- \mathrm{id}) = \mathrm{Lt}(g-\mathrm{id})$.
						\item If $\beta =1$, then $[f]_{1,\mathbf{m}} =[g]_{1,\mathbf{m}}$, for every $\mathbf{n} \leq \mathbf{m} \leq \mathbf{n}'$.
					\end{enumerate}
				\end{prop}
				\begin{proof}
					1. Suppose that $\mathrm{Lt}(f- \mathrm{id}) \neq \mathrm{Lt}(g-\mathrm{id})$. The conjugacy equation \eqref{JednadzbaPrva} is equivalent with the fixed point equation \eqref{FixedPointEq}. From \eqref{DefS} and \eqref{Tdefined} we conclude that
					\begin{align}
						\mathrm{ord}(\mathcal{S}_{f}(\varepsilon )) & = \min \left\lbrace \mathrm{ord}(f- \mathrm{id}),\mathrm{ord}(g-\mathrm{id} )\right\rbrace \label{OrdS}
					\end{align}
					and
					\begin{align}
						\mathrm{ord}(\mathcal{T}_{f}(\varepsilon )) & > \min \left\lbrace \mathrm{ord}(f- \mathrm{id}),\mathrm{ord}(g-\mathrm{id} )\right\rbrace . \label{OrdT} 
					\end{align}
					Now, \eqref{OrdS} and \eqref{OrdT} are in contradiction with $\mathrm{ord}(\mathcal{T}_{f}(\varepsilon ))=\mathrm{ord}(\mathcal{S}_{f}(\varepsilon ))$. This implies that $\mathrm{Lt}(f- \mathrm{id}) = \mathrm{Lt}(g-\mathrm{id})$. \\
					
					2. Let $\beta =1$ and $(1, \mathbf{n}):=\mathrm{ord}(f- \mathrm{id})$. From statement 1, it follows that $\mathrm{Lt}(f- \mathrm{id})=\mathrm{Lt}(g-\mathrm{id})$, which implies that $\alpha = \beta =1$ and $[f]_{1,\mathbf{n}}=[g]_{1,\mathbf{n}} $. Suppose that there exists $\mathbf{m}\in \mathbb{Z}^{k}$ such that $\mathbf{n} < \mathbf{m} \leq \mathbf{n}'$ and
					\begin{align}
						[f]_{1,\mathbf{m}} & \neq [g]_{1,\mathbf{m}} . \label{LemmaNecessaryEq1}
					\end{align}
					Since $\mathrm{Supp}(f)$ and $\mathrm{Supp}(g)$ are well ordered subsets of $\mathbb{R} _{>0} \times \mathbb{Z} ^{k}$, we can suppose that $\mathbf{m}$ is minimal such that \eqref{LemmaNecessaryEq1} holds. Now, suppose that there exists a solution $\varphi  \in \mathcal{L}_{k}^{0}$ of the conjugacy equation \eqref{JednadzbaPrva}. By Lemma~\ref{Lemma2}, the conjugacy equation is equivalent to the fixed point equation $\mathcal{T}_{f}(\varepsilon ) = \mathcal{S}_{f}(\varepsilon )$, for $\varepsilon := \varphi -\mathrm{id}$, where $\mathcal{S}_{f}$ and $\mathcal{T}_{f}$ are given in \eqref{DefS} and \eqref{Tdefined}, respectively. From \eqref{DefS}, it follows that $\mathrm{ord}(\mathcal{S}_{f}(\varepsilon )) = (1,\mathbf{m})$. So,
					\begin{align}
						\mathrm{ord}(\mathcal{T}_{f}(\varepsilon))=(1,\mathbf{m}) , \label{ordTcontradiction}
					\end{align}
					which implies that $\mathrm{ord}_{z}(\varepsilon )=1$. Let $\mathrm{Lb}_{z}(\varepsilon):=zT$. Since $\alpha =\beta =1$, let us use the notation $R:=R_{\beta }$ and $S:=S_{\alpha }$. From \eqref{Tdefined} and \eqref{FormulaDerivationFirstBlock}, it follows that:
					\begin{align}
						\mathcal{T}_{f}(\varepsilon) &= \varepsilon ' \cdot zR - (zS)' \cdot \varepsilon  \nonumber \\
						&= (zT) ' \cdot zR - (zS)' \cdot zT + \mathrm{h.o.b.}(z) \nonumber \\
						&= (T+D_{1}(T))\cdot zR - (S+D_{1}(S))\cdot zT + \mathrm{h.o.b.}(z) \nonumber \\
						&= zT\cdot (R-S)+(zR\cdot D_{1}(T)-zT\cdot D_{1}(S)) + \mathrm{h.o.b.}(z) . \label{orderT}
					\end{align}
					Since $\mathrm{Lt}(f-\mathrm{id})=\mathrm{Lt}(g-\mathrm{id})$, it follows that $\mathrm{ord}(zR)=\mathrm{ord}(zS)=(1,\mathbf{n})$. Hence,
					\begin{align}
						\mathrm{ord}(zR\cdot D_{1}(T)-zT\cdot D_{1}(S)) & >(1,\mathbf{n}') . \label{LemmaNecessaryEq2}
					\end{align}
					Now, from the minimality of $\mathbf{m}$ and the fact that $\mathrm{ord}(zT) > (1,\mathbf{0}_{k})$, it follows that
					\begin{align}
						\mathrm{ord}(zT\cdot (R-S)) & > (1,\mathbf{m}) . \label{LemmaNecessaryEq3}
					\end{align}
					Now, by \eqref{orderT}, \eqref{LemmaNecessaryEq2} and \eqref{LemmaNecessaryEq3}, it follows that $\mathrm{ord}(\mathcal{T}_{f}(\varepsilon )) > (1,\mathbf{m})$, which is a contradiction with \eqref{ordTcontradiction}. 
				\end{proof}
			
				\begin{rem}\hfill
					\begin{enumerate}[1., font=\textup, topsep=0.4cm, itemsep=0.4cm, leftmargin=0.6cm]
						\item Let $f,g\in \mathfrak L^{0}$ and let $L$ be as defined in the Main Theorem. Let $k\in \mathbb{N}$ be such that $f,g\in \mathcal{L}_{k}$. If conjugacy equation $\varphi \circ f\circ \varphi ^{-1}=g$ is solvable in $\mathfrak L^{0}$, by Proposition~\ref{LeadingBlockLemma}, it follows that $g=\mathrm{id}+zL+\mathrm{h.o.t.}$ This indicates that we cannot eliminate block $zL$ from the normal form.
						\item We pose the following question: Can we eliminate the remainder of the terms in the normal form by changes of variables from $\mathcal{L}_{k}^{0}$, where $k\in \mathbb{N}$ is minimal such that $f\in \mathcal{L}_{k}^{0}$ ?
						
						From the theory of standard power series we know that the normal form of a parabolic power series $f:=\mathrm{id}+az^{n}+\mathrm{h.o.t.}$, $a\neq 0$, $n\geq 2$, is $f_{c}:=\mathrm{id}+az^{n}+c\mathrm{Res}(f)$, for some $c\in \mathbb{R}$ (see e.g. \cite{CarlesonG93}, \cite{IlyYak08}, \cite{Loray98}, \cite{Loray05}). Here, $L=az^{n}$. Furthermore, from \cite[Theorem A]{mrrz16} it follows that the normal form of $f:=\mathrm{id}+az^{\beta }\boldsymbol{\ell }_{1}^{n}+\mathrm{h.o.t.}$, $a\neq 0$, $(\beta ,n)>(1,0)$, in the differential algebra $\mathcal{L}_{1}$ is $f_{c}:=\mathrm{id}+zL+c\mathrm{Res}(f)$, for some $c\in \mathbb{R}$. This is an indication that the normal form $f_{c}$ will be as defined in \eqref{DefnOfFc}.
						
						Furthermore, in Corollary~\ref{Remark3} we prove that $c=0$, if we consider changes of variables from a larger group $\mathcal{L}_{m}^{0}$, $m>k$.
						\item In general, the block $L$ is not finite even if we consider changes of variables in the larger group $\mathfrak L^{0}$. Transfinite $L$ may appear for parabolic logarithmic transseries $f\in \mathcal{L}_{k}^{0}$, $k\geq 2$, such that $\mathrm{ord}_{z}(f-\mathrm{id})=1$. This is the main difference when generalizing \cite[Theorem A]{mrrz16} from the differential algebra $\mathcal{L}_{1}$ to $\mathcal{L}_{k}$, for $k\in \mathbb{N}$.
					\end{enumerate}
				\end{rem}
			
					Let $f\in \mathcal{L}_{k}$, $k\in \mathbb{N}$, be such that $f=\mathrm{id}+z^{\beta }R_{\beta } + \mathrm{h.o.b.}(z) $, for $\beta \geq 1$ and $R_{\beta }\in \mathcal{B}_{1}\subseteq \mathcal{L}_{k}^{\infty }$. In Lemma~\ref{Lemma2} we have transformed conjugacy equation \eqref{JednadzbaPrva} to fixed point equation \eqref{FixedPointEq}. In order to apply the fixed point theorem from Proposition~\ref{KorBanach} to solve equation \eqref{FixedPointEq}, we have to find an appropriate complete space in which $\mathcal{T}_{f}$ is a $\lambda $-homothety and $\mathcal{S}_{f}$ is a $\mu $-Lipschitz map, such that $\mu < \lambda $. We distinguish two cases:
					
					\begin{enumerate}[(i), font=\textup, topsep=0.4cm, itemsep=0.4cm, leftmargin=0.6cm]
						\item If $\beta >1$. By Proposition~\ref{LeadingBlockLemma}, in order to conjugate $f$ to $g:=\mathrm{id}+z^{\alpha }S_{\alpha }+\mathrm{h.o.b.}(z)$, $\alpha \geq 1$, the necessary condition is $\mathrm{ord}_{z}(f-\mathrm{id})=\mathrm{ord}_{z}(g-\mathrm{id})$, that is, $\alpha = \beta $. Therefore, by \eqref{Tdefined}, the operator $\mathcal{T}_{f}$ is a $\frac{1}{2^{\beta -1}}$-homothety on the set of all logarithmic transseries in $\mathcal{L}_{k}^{1}$ that are not in its kernel. The operator $\mathcal{S}_{f}$, defined in \eqref{DefS}, is, by Proposition~\ref{Lema3}, a $\frac{1}{2^{\rho (\delta )}}$-Lipschitz map on the space $\mathcal{L}_{k}^{\delta }$, for $\delta \geq 1$. Note that $\frac{1}{2^{\rho (\delta )}} < \frac{1}{2^{\beta -1}}$ if and only if $\delta >1$. Therefore, in order to apply the fixed point theorem from Proposition~\ref{KorBanach}, we restrict ourselves to space $\mathcal{L}_{k}^{\delta }$, for $\delta >1$.
						
						Additionally, we have to exclude from $\mathcal{L}_{k}^{\delta }$ elements of the kernel of $\mathcal{T}_{f}$, which can be easily done only when $z^{\beta }R_{\beta }=\mathrm{Lt}(f-\mathrm{id})$ and $z^{\alpha }S_{\alpha }=\mathrm{Lt}(f-\mathrm{id})$. This is the reason why we first \emph{prenormalize} logarithmic transseries $f$, and then apply the fixed point theorem from Proposition~\ref{KorBanach} in case $\beta >1$. For the precise proof, see Subsection~\ref{sec:proofAa}.
						\item $\beta =1$. By Proposition~\ref{Lema3}, the operator $\mathcal{S}_{f}$ is an $1$-Lipschitz map on the space $(\mathcal{L}_{k}^{\delta },d_{z})$, for every $\delta \geq 1$, where $1$ is its minimal Lipschitz coefficient. On the other hand, $\mathcal{T}_{f}$ is $1$-homothety on the set of all logarithmic transseries in $\mathcal{L}_{k}^{\delta }$ which are not in its kernel, for every $\delta \geq 1$. Consequently, we cannot apply the fixed point theorem from Proposition~\ref{KorBanach} directly because the assumption $\mu < \lambda $ is not satisfied. Therefore, in the case $\beta =1$, we modify the operators $\mathcal{T}_{f}$ and $\mathcal{S}_{f}$. 
					\end{enumerate}
					
					That is why we split the proof of the Main Theorem in two different cases:
					\begin{enumerate}[(a), font=\textup, topsep=0.4cm, itemsep=0.4cm, leftmargin=0.6cm]
						\item $\beta = \mathrm{ord}_{z}(f- \mathrm{id}) >1$,
						\item $\beta = \mathrm{ord}_{z}(f- \mathrm{id}) =1$.
					\end{enumerate}
					The respective proofs of cases (a) and (b) are given in Subsections~\ref{sec:proofAa} and~\ref{sec:proofAb} respectively. For a visual representation of the sketch of the proof of statement 1 see Figure~\ref{FigureSketchProof}.
					
					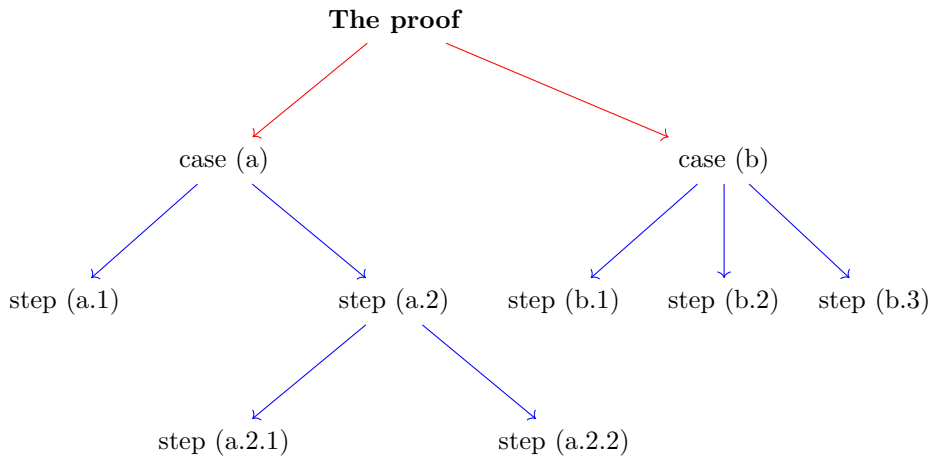
\begin{figure}[!h]
						$$
							\begin{tikzcd}[row sep=huge, column sep=tiny]
								& & \textbf{The proof} \arrow[ld,red,->] \arrow[rrd,red,->] & & & \\
								& \textrm{case (a)} \arrow[ld,blue,->] \arrow[rd,blue,->] & & & \textrm{case (b)} \arrow[d,blue,->] \arrow[ld,blue,->] \arrow[rd,blue,->] & \\
								\textrm{step (a.1)} & & \textrm{step (a.2)} \arrow[ld,blue,->] \arrow[rd,blue,->] & \textrm{step (b.1)} & \textrm{step (b.2)} & \textrm{step (b.3)} \\
								& \textrm{step (a.2.1)} & & \textrm{step (a.2.2)} & &
							\end{tikzcd}
						$$
						\caption{Diagram of the proof of the Main Theorem.}\label{FigureSketchProof}
					\end{figure}

				\subsection{Proof of case $\mathrm{ord}_{z}(f- \mathrm{id}) >1$ of statement 1}\label{sec:proofAa}
				
				\subsubsection{Sketch of the proof of case $\mathrm{ord}_{z}(f- \mathrm{id}) >1$ of statement 1}
				
				Here we give the sketch of the proof. The proof, after introducing all needed proposition and lemmas, is at the end of the subsection.
				
				Let $f\in \mathcal{L}_{k}$, $k\in \mathbb{N}$, be such that $f:=\mathrm{id}+z^{\beta }R_{\beta } + \mathrm{h.o.b.}(z) $, for $\beta >1$ and $R_{\beta }\in \mathcal{B}_{1}\setminus \left\lbrace 0\right\rbrace \subseteq \mathcal{L}_{k}^{\infty }$. Let $L$ be as defined in the Main Theorem (for $\beta >1$) and let $g\in \mathcal{L}_{k}^{0}$ be such that $g:=\mathrm{id}+z^{\beta}L+\mathrm{h.o.t.}$ Let $S_{\beta }$ be the whole $\beta $-block of $g$. That is, $g=\mathrm{id}+z^{\beta }S_{\beta }+\mathrm{h.o.b.}(z)$.
				
				By Lemma~\ref{Lemma2}, it follows that a parabolic logarithmic transseries $\varphi=\mathrm{id}+\varepsilon \in \mathcal{L}_{k}^{0}$, $\mathrm{ord}(\varepsilon)>(1,\mathbf{0}_{k})$, satisfies the conjugacy equation
				\begin{align}
					\varphi  \circ f \circ \varphi ^{-1} = g \label{ConjEquation1}
				\end{align}
				if and only if the logarithmic transseries $\varepsilon \in \mathcal{L}_{k}$ satisfies the \emph{fixed point equation}
				\begin{align*}
					\mathcal{S}_{f}(\varepsilon ) & = \mathcal{T}_{f}(\varepsilon ) ,
				\end{align*}
				where $\mathcal{S}_{f},\mathcal{T}_{f} : \mathcal{L}_{k} \to \mathcal{L}_{k}$ are given in \eqref{DefS} and \eqref{Tdefined} respectively.
				
				In general, $z^{\beta }S_{\beta } \neq z^{\beta }R_{\beta }$, so we first eliminate terms in the leading block of $f-\mathrm{id}-zL$, for $L$ defined in the Main Theorem. By \eqref{Tdefined},
				\begin{align}
					\mathrm{ord}_{z}(\mathcal{T}_{f}(\varepsilon )) & \geq \mathrm{ord}_{z}(\varepsilon )+ \beta -1 . \label{CoeffHomothecy}
				\end{align}
				By \eqref{DefS}, $\mathrm{ord}_{z}(\mathcal{S}_{f}(\varepsilon )) = \beta $. From $\mathcal{T}_{f}(\varepsilon ) = \mathcal{S}_{f}(\varepsilon )$, it necessarily follows that $\mathrm{ord}_{z}(\varepsilon )+ \beta -1= \beta $, so $\mathrm{ord}_{z}(\varepsilon ) = 1$.
					
				By Proposition~\ref{Lema3}, $\mathcal{S}_{f}$ is $\frac{1}{2^{\beta -1}}$-Lipschitz on the space $\mathcal{L}_{k}^{1}$. By \eqref{CoeffHomothecy}, the coefficient of $\mathcal{T}_{f}$ is $\frac{1}{2^{\beta -1}}$, which is the same as the minimal Lipschitz coefficient of $\mathcal{S}_{f}$ on the space $\mathcal{L}_{k}^{1}$. Therefore, we cannot apply the fixed point theorem from Proposition~\ref{KorBanach} directly, and we change operators $\mathcal{S}_{f}$ and $\mathcal{T}_{f}$ in our fixed point equation. \\
					
				To conclude, in order to eliminate every term in the leading block of $f - \mathrm{id}$ except the leading term (that cannot be eliminated), we put $g := \mathrm{id}+\mathrm{Lt}(f- \mathrm{id})+\mathrm{h.o.b.}(z)$. If the leading block of $f- \mathrm{id}$ contains at least two different terms, by the above discussion, we cannot solve the fixed point equation using the fixed point theorem from Proposition~\ref{KorBanach}. Therefore, we proceed in two steps (see Figure~\ref{FigureSketchProof}): \\
				
				\begin{itemize}
					\item[\emph{Step (a.1)}] We obtain a parabolic conjugacy $\varphi_{1}=\mathrm{id}+zT\in \mathcal{L}_{k}^{0}$, $T \in \mathcal{B}_{\geq 1}^{+}\subseteq \mathcal{L}_{k}$, as a solution of a \emph{prenormalization equation}:
					\begin{align}
						\varphi _{1}\circ f \circ \varphi _{1}^{-1}=\mathrm{id} +az^{\beta }\boldsymbol{\ell }_{1}^{n_{1}}\cdots \boldsymbol{\ell}_{k}^{n_{k}}+\mathrm{h.o.b.}(z) , \label{Eq3}
					\end{align}
					where $\mathrm{Lt}(f-\mathrm{id})=az^{\beta }\boldsymbol{\ell }_{1}^{n_{1}}\cdots \boldsymbol{\ell }_{k}^{n_{k}}$, $\beta >1$, $a\neq 0$. That is, we eliminate every term in the leading block of $f-\mathrm{id}$ except its leading term. Furthermore, such $T \in \mathcal{B}_{\geq 1}^{+}\subseteq \mathcal{L}_{k}$ is unique. The proof of step $(a.1)$ is in Subsection~\ref{subsubsection:proofA1}. \\
						
					\item[\emph{Step (a.2)}] We obtain a parabolic conjugacy $\varphi _{2}\in \mathcal{L}_{k}^{0}$, $\mathrm{ord}_{z}(\varphi_{1}-\mathrm{id})>1$ (not unique), as a solution of the \emph{normalization equation}:
					\begin{align}
						\varphi _{2}\circ ( \varphi _{1}\circ f \circ \varphi _{1}^{-1} ) \circ \varphi _{2}^{-1}=\mathrm{id} +az^{\beta }\boldsymbol{\ell }_{1}^{n_{1}}\cdots \boldsymbol{\ell}_{k}^{n_{k}}+c\mathrm{Res}(f) , \label{Eq4}
					\end{align}
					for a unique choice of the residual coefficient $c\in \mathbb{R} $. That is, we show that we can eliminate every term in $\varphi _{1}\circ f \circ \varphi _{1}^{-1} -\mathrm{id}-az^{\beta }\boldsymbol{\ell }_{1}^{n_{1}}\cdots \boldsymbol{\ell}_{k}^{n_{k}}$ except the residual term (which may be changed).
					
					Once prenormalized, we apply the fixed point theorem from Proposition~\ref{KorBanach} directly in order to obtain $\varphi _{2}$, but only after the formal invariant $c$ is choosen appropriately. So, we split the proof of the step $(a.2)$ into two steps: eliminations up to the residual term and after the residual term. The proof of step $(a.2)$ is in Subsection~\ref{subsubsection:proofA2}.
				\end{itemize}
			
				\noindent Finally, $\varphi :=\varphi _{2} \circ \varphi_{1}$ is a solution of the conjugacy equation \eqref{ConjEquation1}, for $g:=f_{c}$.

				\subsubsection{Step (a.1): Solving the prenormalization equation}\label{subsubsection:proofA1}
				
				We first transform the prenormalization equation \eqref{Eq3} to an appropriate fixed point equation, in order to apply the fixed point theorem from Proposition~\ref{KorBanach}.
				
				\begin{lem}[Fixed point equation for the prenormalization]\label{Prop1}
					Let $f\in \mathcal{L}_{k}$, $k\in \mathbb{N}_{\geq 1}$, be such that $f:=\mathrm{id}+z^{\beta }R_{\beta }+ \mathrm{h.o.b.}(z)$, $\beta >1$, $R_{\beta }\in \mathcal{B}_{1}\setminus \left\lbrace 0\right\rbrace \subseteq \mathcal{L}_{k}^{\infty }$, and let $az^{\beta }\boldsymbol{\ell}_{1}^{n_{1}}\cdots \boldsymbol{\ell }_{k}^{n_{k}}$, $a\neq 0$, be the leading term of $f- \mathrm{id}$. Let $\mathcal{T}_{1},\mathcal{S}_{1}:\mathcal{B}_{\geq 1}^{+} \to \mathcal{B}_{\geq 1}^{+} $ be the operators defined by:
					\begin{align}
						\mathcal{T}_{1}(Q) & :=  \displaystyle \Big( 1- \frac{\beta a\boldsymbol{\ell}_{1}^{n_{1}}\cdots \boldsymbol{\ell}_{k}^{n_{k}}}{R_{\beta }}\Big) \cdot Q - \frac{a\boldsymbol{\ell}_{1}^{n_{1}}\cdots \boldsymbol{\ell}_{k}^{n_{k}}}{R_{\beta }}\cdot \sum _{i\geq 2} {\beta \choose i}Q^{i} , \label{EqT} \\
						\mathcal{S}_{1}(Q) & := \frac{\mathcal{C}(Q)}{R_{\beta }} - D_{1}(Q) + \frac{a\boldsymbol{\ell}_{1}^{n_{1}}\cdots \boldsymbol{\ell}_{k}^{n_{k}}}{R_{\beta }} -1\, , \label{Eq2} 
					\end{align}
					$Q \in \mathcal{B}_{\geq 1}^{+}\subseteq \mathcal{L}_{k}$, where $\mathcal{C}$ is a suitable $\frac{1}{2^{1+n_{1}}}$-Lipschitz map on the space $\big( \mathcal{B}_{\geq 1}^{+} , d_{1}\big) $. Then $\varphi _{1} = \mathrm{id} + zQ$ is a solution of \emph{the prenormalization equation} \eqref{Eq3} if and only if $$\mathcal{T}_{1}(Q)= \mathcal{S}_{1}(Q) .$$
				\end{lem}
				\begin{proof}
					Since $z^{\beta }R_{\beta }=az^{\beta }\boldsymbol{\ell}_{1}^{n_{1}}\cdots \boldsymbol{\ell }_{k}^{n_{k}}+\mathrm{h.o.t.}$, it follows that
					\begin{align*}
						\mathrm{ord} \left( \frac{a\boldsymbol{\ell}_{1}^{n_{1}}\cdots \boldsymbol{\ell}_{k}^{n_{k}}}{R_{\beta }} -1\right) & > \mathbf{0}_{k+1} .
					\end{align*}
					Therefore, the operator $\mathcal{S}_{1}$ is well-defined on $\mathcal{B}_{\geq 1}^{+}\subseteq \mathcal{L}_{k}$.
					
					The equation \eqref{Eq3} with $g := z+az^{\beta }\boldsymbol{\ell }_{1}^{n_{1}}\cdots \boldsymbol{\ell}_{k}^{n_{k}}+\mathrm{h.o.b.}(z)$ can be equivalently written as:
					\begin{align*}
						& \varphi  \circ f - g \circ \varphi  = 0 ,
					\end{align*}
					This implies that
					\begin{align}
						\mathrm{Lb}_{z}\big( \varphi \circ f-g \circ \varphi  \big) & =0 . \label{Eq1}
					\end{align}
					Let $L:=a\boldsymbol{\ell }_{1}^{n_{1}}\cdots \boldsymbol{\ell}_{k}^{n_{k}}$. From \eqref{EquationLema2} we get that \eqref{Eq1} is equivalent with the equation:
					\begin{align}
						(zQ) '\cdot z^{\beta }R_{\beta } - \sum_{i\geq 1}\frac{( z^{\beta }L )^{(i)}}{i!}(zQ)^{i} +z^{\beta }\cdot ( R_{\beta } -L) =0. \label{Jjjed}
					\end{align}
					Dividing equation \eqref{Jjjed} by $z^{\beta }$ and using identities \eqref{Identitet1} and \eqref{FormulaDerivationFirstBlock}, we get that:
					\begin{align}
						(Q + D_{1}(Q)) \cdot R_{\beta } - L \cdot \sum_{i\geq 1}{\beta \choose i}Q^{i} - \mathcal{C}(Q) + R_{\beta } -L & = 0, \label{FixedPointLemmaEq1}
					\end{align}
					where $\mathcal{C}:=\mathcal{C}(L, \cdot ):\mathcal{B}_{1}\to \mathcal{B}_{1}$ is a $\frac{1}{2^{1+n_{1}}}$-Lipschitz map from \eqref{Identitet1}. We divide equation \eqref{FixedPointLemmaEq1} by $R_{\beta }$ and we get the following equation:
					\begin{align*}
						\Big( 1- \frac{\beta L}{R_{\beta }}\Big) \cdot Q - \frac{L}{R_{\beta }}\cdot \sum _{i\geq 2}{\beta \choose i}Q^{i} & = \frac{\mathcal{C}(Q)}{R_{\beta }} - D_{1}(Q)+ \frac{L}{R_{\beta }} -1 ,
					\end{align*}
					i.e., $\mathcal{T}_{1}(Q)= \mathcal{S}_{1}(Q)$.
				\end{proof}
			
				\begin{remark}[Minimality of the prenormalization]
					By Proposition~\ref{LeadingBlockLemma}, it follows that $\mathrm{Lt}(f- \mathrm{id})=\mathrm{Lt}(\varphi _{1}\circ f \circ \varphi _{1}^{-1}-\mathrm{id})$. Therefore, the prenormalization $\varphi _{1}\circ f \circ \varphi _{1}^{-1} = \mathrm{id}+az^{\beta }\boldsymbol{\ell}_{1}^{n_{1}}\cdots \boldsymbol{\ell}_{k}^{n_{k}} + \mathrm{h.o.b.}(z)$ is minimal\footnote{In a sense that the leading term of $\varphi _{1}\circ f\circ \varphi _{1}^{-1}-\mathrm{id}$ cannot be eliminated, nor its coefficients changed.}.
				\end{remark}
			
				In Proposition~\ref{Prop2} below, we prove that the operators $\mathcal{T}_{1}$ and $\mathcal{S}_{1}$ defined in \eqref{EqT} and \eqref{Eq2} satisfy assumptions of the fixed point theorem from Proposition~\ref{KorBanach}.
				
				\begin{prop}[Properties of operators $\mathcal{S}_{1}$ and $\mathcal{T}_{1}$]\label{Prop2}
					Let $f\in \mathcal{L}_{k}$, $k\in \mathbb{N}_{\geq 1}$, be such that $f:=\mathrm{id}+z^{\beta }R_{\beta }+ \mathrm{h.o.b.}(z)$, $\beta >1$, $R_{\beta }\in \mathcal{B}_{1}\setminus \left\lbrace 0\right\rbrace \subseteq \mathcal{L}_{k}^{\infty }$, and let $az^{\beta }\boldsymbol{\ell}_{1}^{n_{1}}\cdots \boldsymbol{\ell}_{k}^{n_{k}}$, $a\neq 0$, be the leading term of $f- \mathrm{id}$. Let $\mathcal{T}_{1},\mathcal{S}_{1}:\mathcal{B}_{\geq 1}^{+} \to \mathcal{B}_{\geq 1}^{+} $ be the operators defined in \eqref{EqT} and \eqref{Eq2}, respectively. Then
					\begin{enumerate}[1., font=\textup, topsep=0.4cm, itemsep=0.4cm, leftmargin=0.6cm]
						\item $\mathcal{S}_{1}$ is a $\frac{1}{2}$-contraction on the space $\big( \mathcal{B}_{\geq 1}^{+} , d_{1}\big) $.
						\item $\mathcal{T}_{1}$ is an isometry on the space $\big( \mathcal{B}_{\geq 1}^{+} , d_{1}\big) $.
						\item $\mathcal{T}_{1}$ is a surjection.
					\end{enumerate}
				\end{prop}
				
				First we state Lemma~\ref{Lem4} which is used in the proof of Proposition~\ref{Prop2}. In its proof we need Lemma~\ref{Lem5} also stated below.
			
				\begin{lem}\label{Lem4}
					Let $M \in \mathcal{B}_{\geq 1}^{+}\subseteq \mathcal{L}_{k}$, $k\in \mathbb{N}_{\geq 1}$, and let $R_{1}, R_{2} \in \mathcal{B}_{1}\setminus \left\lbrace 0\right\rbrace \subseteq \mathcal{L}_{k}^{\infty }$ such that $\mathrm{ord}(R_{1}), \mathrm{ord}(R_{2}) = \mathbf{0}_{k+1}$, i.e., $R_{1} $ and $R_{2} $ have non-zero constants as the leading terms. Let $h \in x^{2}\mathbb{R} \left[ \left[ x\right] \right] $ be a formal power series with real coefficints in the variable $x$. Then there exists a unique $Q \in \mathcal{B}_{\geq 1}^{+}\subseteq \mathcal{L}_{k}$, such that
					\begin{align}
						R_{1} \cdot Q +R_{2} \cdot h(Q) & = M . \label{EquationDif}
					\end{align}
				\end{lem}
			
				Lemma~\ref{Lem4} is a generalization of Lemma~\ref{Lem5} below. We prove Lemma~\ref{Lem4} inductively using Lemma~\ref{Lem5} for the inductive step.
				
				\begin{lem}\label{Lem5}
					Let $k\in \mathbb{N}_{\geq 1}$, $1\leq m\leq k$ and $M \in \mathcal{B}_{m}^{+}\subseteq \mathcal{L}_{k}$. Let $R_{1} , R_{2} \in \mathcal{B}_{m} \setminus \left\lbrace 0\right\rbrace \subseteq \mathcal{L}_{k}^{\infty }$ be such that $\mathrm{ord}(R_{1}) , \mathrm{ord}(R_{2})= \mathbf{0}_{k+1}$, i.e., $R_{1} $ and $R_{2} $ have non-zero constants as the leading terms. Let $h \in x^{2} \cdot \mathcal{B}_{\geq m}^{+}\left[ \left[ x\right] \right] $ be a formal power series with coefficients in $\mathcal{B}_{\geq m}^{+}\subseteq \mathcal{L}_{k}$ in the variable $x$. Then there exists a unique solution $Q \in \mathcal{B}_{m} ^{+}\subseteq \mathcal{L}_{k}$ of the equation
					\begin{eqnarray}
						R_{1} \cdot Q + R_{2} \cdot h(Q)= M . \label{Eq10}
					\end{eqnarray}
				\end{lem}
			
				\noindent The proof of Lemma~\ref{Lem5} is in the Appendix.
				
				\begin{proof}[Proof of Lemma~\ref{Lem4}]
					\emph{Existence.} We have the following decomposition:
					\begin{align*}
						M & = M_{1} + \cdots + M_{k} ,
					\end{align*}
					where $M_{i} \in \mathcal{B}_{i}^{+}$, for $1\leq i \leq k$. Similarly, we have the following decompositions:
					\begin{align*}
						R_{1} & = R_{1,1}+ \cdots + R_{1,k}
					\end{align*}
					and
					\begin{align*}
						R_{2} & = R_{2,1} + \cdots + R_{2,k} \, ,
					\end{align*}
					where $R_{1,m},R_{2,m}\in \mathcal{B}_{m}^{+}\subseteq \mathcal{L}_{k}$, for $1\leq m\leq k-1$, and $R_{1,k},R_{2,k}\in \mathcal{B}_{k}\subseteq \mathcal{L}_{k}^{\infty }$ such that $\mathrm{ord}(R_{1,k})=\mathrm{ord}(R_{2,k})=\mathbf{0}_{k+1}$. Without loss of generality assume that $M_{k} \neq 0$. If $M_{k}=0$, then replace $k$ with the biggest $m$ such that $1\leq m < k$, $M_{m} \neq 0$, and replace $R_{i,k}$ with $R_{i,m}\in \mathcal{B}_{m}\subseteq \mathcal{L}_{k}^{\infty }$ such that $\mathrm{ord}(R_{i,m})=\mathbf{0}_{k+1}$, for $i=1,2$.
					
					We use the decompositions above to solve the equation \eqref{EquationDif} inductively. Using Lemma~\ref{Lem5}, in each inductive step we solve an appropriate equation obtained from equation \eqref{EquationDif} by the projection on the space $\mathcal{B}_{i}^{+}$, $i=1, \ldots ,k$. \\
					
					\noindent \emph{1-step.} By Lemma~\ref{Lem5}, there exists a unique solution $Q_{k} \in \mathcal{B}_{k}^{+}\subseteq \mathcal{L}_{k}$ of the equation:
					\begin{align}
						R_{1,k} \cdot Q_{k} + R_{2,k} \cdot h (Q_{k} )= M_{k} \, . \label{EqDecompQ}
					\end{align}
					
					\noindent \emph{2-step.} Let us now consider the equation:
					\begin{align*}
						& (R_{1,k} + R_{1,k-1}) \cdot (Q_{k} + Q_{k-1} )  + (R_{2,k} + R_{2,k-1})\cdot  h (Q_{k} + Q_{k-1}) & \\
						&= M_{k} + M_{k-1}
					\end{align*}
					in the variable $Q_{k-1} \in \mathcal{B}_{k-1}^{+}$. By the Taylor Theorem
					\begin{align}
						h(Q_{k} + Q_{k-1} )= h (Q_{k}) + \sum _{i\geq 1}\frac{ h ^{(i)} (Q_{k} )}{i!}Q_{k-1}^{i} . \label{Eq13}
					\end{align}
					Using \eqref{EqDecompQ} and \eqref{Eq13}, we get the equation
					{\small \begin{align}
						& (R_{1,k} + R_{1,k-1} + (R_{2,k} + R_{2,k-1})\cdot h ' (Q_{k}) )\cdot Q_{k-1} + (R_{2,k} + R_{2,k-1})\cdot \sum _{i\geq 2}\frac{ h ^{(i)}(Q_{k} )}{i!}Q_{k-1}^{i} \nonumber \\
						&= M_{k-1} - R_{1,k-1} \cdot Q_{k} - R_{2,k-1}\cdot h(Q_{k}). \label{Eq14}
					\end{align} 
					}Since $R_{1,k}$ has non-zero constant as the leading term, $R_{1,k-1},R_{2,k-1}\in \mathcal{B}_{k-1}^{+}\subseteq \mathcal{L}_{k}$, and $R_{2,k}\cdot h ' (Q_{k}) \in \mathcal{B}_{k}^{+}\subseteq \mathcal{L}_{k}$, it follows that $R_{1,k} + R_{1,k-1} + (R_{2,k} + R_{2,k-1})\cdot h ' (Q_{k} )\neq 0$. By Lemma~\ref{Lem5}, there exists a unique solution $Q_{k-1} \in \mathcal{B}_{k-1}^{+}\subseteq \mathcal{L}_{k}$ of the equation \eqref{Eq14}.
				
					\noindent \emph{Steps 3,}$\ldots $\emph{,k.} Proceeding inductively, there exists a solution $Q := Q_{1} + \cdots + Q_{k}\in \mathcal{B}_{\geq 1}^{+}\subseteq \mathcal{L}_{k}$ of equation \eqref{EquationDif}. \\
				
					\emph{Uniqueness.} Let us now prove the uniqueness of $Q$ in $\mathcal{B}_{\geq 1}^{+}\subseteq \mathcal{L}_{k}$. Every solution $S\in \mathcal{B}_{\geq 1}^{+}\subseteq \mathcal{L}_{k}$ of the equation \eqref{EquationDif} can be decomposed as $S=S_{1}+\cdots +S_{k}$, for $S_{i}\in \mathcal{B}_{i}^{+}\subseteq \mathcal{L}_{k}$ satisfying the same equation in the inductive step as $Q_{i}$ does, for every $i=1,\ldots ,k$. Since $Q_{i}$ is a unique solution in $\mathcal{B}_{i}^{+}\subseteq \mathcal{L}_{k}$ of the appropriate equation in the inductive step, it follows that $S_{i}=Q_{i}$, for every $i=1,\ldots ,k$, and therefore, $S=Q$.  
				\end{proof}
				
				\begin{proof}[Proof of Proposition~\ref{Prop2}]
					1. Recall that $\mathcal{C}$ in \eqref{Eq2} is a $\frac{1}{2^{1+n_{1}}}$-Lipschitz map. Since $\mathrm{ord}(R_{\beta })=(0,n_{1},\ldots ,n_{k})$, it follows that
					\begin{align*}
						\mathrm{ord}_{\boldsymbol{\ell }_{1}}\Big( \frac{\mathcal{C}(Q)}{R_{\beta }} \Big) & \geq \mathrm{ord}_{\boldsymbol{\ell }_{1}}(Q)+1 .
					\end{align*}
					Since $D_{1}$ is a linear $\frac{1}{2}$-contraction, from \eqref{Eq2} it follows that $\mathcal{S}_{1}$ is a $\frac{1}{2}$-contraction on the space $(\mathcal{B}_{\geq 1}^{+},d_{1})$. \\
					
					2. From \eqref{EqT}, since $\beta >1$, and by Lemma~\ref{Lema1}, it follows that
					$$
						\mathrm{ord}_{\boldsymbol{\ell }_{1}} \big( \mathcal{T}_{1}(Q_{1})-\mathcal{T}_{1}(Q_{2})\big) = \mathrm{ord}_{\boldsymbol{\ell }_{1}}(Q_{1}-Q_{2}),
					$$
					for every $Q_{1},Q_{2}\in \mathcal{B}_{\geq 1}^{+}$. Therefore, $\mathcal{T}_{1}$ is an isometry on $(\mathcal{B}_{\geq 1}^{+},d_{1})$. \\
					
					3. Recall that $\beta = \mathrm{ord}_{z}(f- \mathrm{id})$. Take $h\in x^{2}\mathbb{R} \left[ \left[ x\right] \right] $ such that
					\begin{align*}
						h & := \sum _{i\geq 2} {\beta \choose i}x^{i}.
					\end{align*}
					Rewrite \eqref{EqT} as:
					\begin{align*}
						\mathcal{T}_{1}(Q) & = \Big( 1- \frac{\beta a\boldsymbol{\ell }_{1}^{n_{1}}\cdots \boldsymbol{\ell }_{k}^{n_{k}}}{R_{\beta }}\Big) \cdot Q - \frac{a\boldsymbol{\ell }_{1}^{n_{1}}\cdots \boldsymbol{\ell }_{k}^{n_{k}}}{R_{\beta }} \cdot h(Q) , \quad Q \in \mathcal{B}_{\geq 1}^{+}\subseteq \mathcal{L}_{k} .
					\end{align*}
					Since $az^{\beta }\boldsymbol{\ell}_{1}^{n_{1}}\cdots \boldsymbol{\ell}_{k}^{n_{k}}$ is the leading term of $z^{\beta }R_{\beta }$, and $\beta >1$, it follows that $\mathrm{ord}\Big( 1- \frac{\beta a\boldsymbol{\ell }_{1}^{n_{1}}\cdots \boldsymbol{\ell }_{k}^{n_{k}}}{R_{\beta }} \Big) , \mathrm{ord}\Big( \frac{a\boldsymbol{\ell }_{1}^{n_{1}}\cdots \boldsymbol{\ell }_{k}^{n_{k}}}{R_{\beta }} \Big) = \mathbf{0}_{k+1}$, which is the order of a constant term. The surjectivity of $\mathcal{T}_{1}$ now follows by Lemma~\ref{Lem4}.
				\end{proof}
			
				\begin{proof}[Proof of step $(a.1)$ of statement 1]
					By Lemma~\ref{Prop1}, we transform the prenormalization equation
					\begin{align}
						\varphi _{1}^{-1} \circ f \circ \varphi _{1} = \mathrm{id} + az^{\beta }\boldsymbol{\ell }_{1}^{n_{1}}\cdots \boldsymbol{\ell}_{k}^{n_{k}} + \mathrm{h.o.b.}(z) \, , \label{Equation10}
					\end{align}
					for $\varphi _{1}:=\mathrm{id}+ zQ$, $Q\in \mathcal{B}_{\geq 1}^{+}\subseteq \mathcal{L}_{k}$, to the equivalent fixed point equation $\mathcal{T}_{1}(Q)= \mathcal{S}_{1}(Q)$, where $az^{\beta }\ell _{1}^{n_{1}}\cdots \boldsymbol{\ell }_{k}^{n_{k}}=\mathrm{Lt}(f-\mathrm{id})$, for $a\neq 0$. By Proposition~\ref{Prop2} and the fixed point theorem from Proposition~\ref{KorBanach}, there exists a unique $T \in \mathcal{B}_{\geq 1}^{+}\subseteq \mathcal{L}_{k}$, such that $\mathcal{T}_{1}(T)= \mathcal{S}_{1}(T)$. Therefore, $\varphi _{1}:=\mathrm{id}+zT$ is a unique solution of the prenormalization equation \eqref{Equation10} up to higher order blocks in $z$. 
				\end{proof}
			
				\vspace{0.2cm}

				\subsubsection{Step (a.2): Solving the normalization equation}\label{subsubsection:proofA2}
				
				For simplicity, we denote the prenormalized logarithmic transseries $\varphi _{1}^{-1} \circ f \circ \varphi _{1} = \mathrm{id} + az^{\beta }\boldsymbol{\ell }_{1}^{n_{1}}\cdots \boldsymbol{\ell}_{k}^{n_{k}} + \mathrm{h.o.b.}(z)$, $\beta >1$, $a\neq 0$, again by $f$. By Proposition~\ref{LeadingBlockLemma}, it follows that $\mathrm{Res}(\varphi \circ f\circ \varphi ^{-1})=\mathrm{Res}(f)$.
				
				In this step, we eliminate all terms in $f- \mathrm{id}-az^{\beta }\boldsymbol{\ell }_{1}^{n_{1}}\cdots \boldsymbol{\ell}_{k}^{n_{k}}$ except the residual term. \\
				
				Let $f = \mathrm{id}+az^{\beta }\boldsymbol{\ell}_{1}^{n_{1}}\cdots \boldsymbol{\ell}_{k}^{n_{k}} + \mathrm{h.o.b.}(z)$, and let $g \in \mathcal{L}_{k}^{0}$. Consider the conjugacy equation $\varphi \circ f \circ \varphi ^{-1}=g$, for $\varphi \in \mathcal{L}_{k}^{0}$. By Lemma~\ref{Lemma2}, it is equivalent to the fixed point equation $\mathcal{T}_{f}(\varepsilon) = \mathcal{S}_{f}(\varepsilon )$, $\varepsilon  := \varphi - \mathrm{id}\in \mathcal{L}_{k}$, where $\mathcal{S}_{f}$ and $\mathcal{T}_{f}$ are explicitely given in \eqref{DefS} and \eqref{Tdefined}. Since $f$ is already prenormalized, we consider $g\in \mathcal{L}_{k}^{0}$, such that $\mathrm{Lb}_{z}(g-\mathrm{id})=az^{\beta }\boldsymbol{\ell}_{1}^{n_{1}}\cdots \boldsymbol{\ell}_{k}^{n_{k}} $.
				
				Notice that $\mathcal{T}_{f}$ is now a generalization of the \emph{Lie bracket operator} defined in \cite[Section 3]{mrrz16}:
				\begin{align}
					\mathcal{T}_{f}(\varepsilon ) = \varepsilon ' \cdot az^{\beta }\boldsymbol{\ell}_{1}^{n_{1}}\cdots \boldsymbol{\ell}_{k}^{n_{k}} - (az^{\beta }\boldsymbol{\ell}_{1}^{n_{1}}\cdots \boldsymbol{\ell}_{k}^{n_{k}})' \cdot \varepsilon  , \quad \varepsilon  \in \mathcal{L}_{k}^{1} . \label{EquationT}
				\end{align}
				Let us denote by
				\begin{align*}
					& \mathcal{L}_{k}^{\delta ,1}, \quad \delta \geq 1,
				\end{align*}
				a subspace of $\mathcal{L}_{k}^{\delta }$ which consists of all logarithmic transseries in $\mathcal{L}_{k}^{\delta }$ that do not contain a term of order $(\beta , n_{1},\ldots ,n_{k})=(\beta , \mathbf{n})$. Moreover, we denote by
				\begin{align*}
					& \mathcal{L}_{k}^{\delta ,2}, \quad  \delta \geq 1,
				\end{align*}
				the subspace of $\mathcal{L}_{k}^{\delta }$ which consists of all logarithmic transseries in $\mathcal{L}_{k}^{\delta }$ that do not contain a term of order $\mathrm{ord}(\mathrm{Res}(f))=(2\beta -1, 2\mathbf{n}+\mathbf{1}_{k})$. The reason for this restriction is explained in Remark~\ref{Remark} below.
				
				\begin{remark}[Injectivity of the operator $\mathcal{T}_{f}$]\label{Remark}
					Let $f\in \mathcal{L}_{k}$, $k\in \mathbb{N}$, such that $f=\mathrm{id}+az^{\beta }\boldsymbol{\ell }_{1}^{n_{1}}\cdots \boldsymbol{\ell }_{k}^{n_{k}}+\mathrm{h.o.t.}$, for $\beta >1$ and $a\neq 0$. It is easy to see that $\mathcal{T}_{f}$, given by \eqref{Tdefined}, is a linear operator with the kernel
					\begin{align*}
						\ker (\mathcal{T}_{f}) &= \left\lbrace Cz^{\beta }\pmb{\ell}_{1}^{n_{1}}\cdots \pmb{\ell}_{k}^{n_{k}} : C\in \mathbb{R} \right\rbrace .
					\end{align*}
					Indeed, suppose that $\mathcal{T}_{f}(\varepsilon) = 0$, i.e.
					\begin{align*}
						& \varepsilon ' \cdot az^{\beta }\boldsymbol{\ell}_{1}^{n_{1}}\cdots \boldsymbol{\ell}_{k}^{n_{k}} - (az^{\beta }\boldsymbol{\ell}_{1}^{n_{1}}\cdots \boldsymbol{\ell}_{k}^{n_{k}})' \cdot \varepsilon  = 0, \\
						& \varepsilon ' \cdot az^{\beta }\boldsymbol{\ell}_{1}^{n_{1}}\cdots \boldsymbol{\ell}_{k}^{n_{k}} = (az^{\beta }\boldsymbol{\ell}_{1}^{n_{1}}\cdots \boldsymbol{\ell}_{k}^{n_{k}})' \cdot \varepsilon  , \\
						& \frac{\varepsilon '}{\varepsilon } = \frac{(z^{\beta }\boldsymbol{\ell}_{1}^{n_{1}}\cdots \boldsymbol{\ell}_{k}^{n_{k}})'}{z^{\beta }\boldsymbol{\ell}_{1}^{n_{1}}\cdots \boldsymbol{\ell}_{k}^{n_{k}} } , \\
						& \left( \log \varepsilon \right) ' = \big( \log (z^{\beta }\boldsymbol{\ell}_{1}^{n_{1}}\cdots \boldsymbol{\ell}_{k}^{n_{k}}) \big) ', \\
						& \varepsilon = Cz^{\beta }\boldsymbol{\ell}_{1}^{n_{1}}\cdots \boldsymbol{\ell}_{k}^{n_{k}} , \quad C\in \mathbb{R} . 
					\end{align*}
					Therefore,
					\begin{enumerate}[1., font=\textup, topsep=0.4cm, itemsep=0.4cm, leftmargin=0.6cm]
						\item the operator $\mathcal{T}_{f}$ is not a homothety on any subspace of $\mathcal{L}_{k}^{1}$ that has nonempty intersection with $\ker (\mathcal{T}_{f}) \setminus \left\lbrace 0\right\rbrace $,
						\item the restriction $\mathcal{T}_{f}: \mathcal{L}_{k}^{\delta ,1} \to \mathcal{L}_{k}^{\delta ,1}$ is injective, because $\mathcal{L}_{k}^{\delta ,1}$ does not contain terms in $\ker (\mathcal{T}_{f})\setminus \left\lbrace 0\right\rbrace $, for every $\delta \geq 1$. Moreover, $\mathcal{L}_{k}^{\delta ,1}$ is a complete subspace of $\mathcal{L}_{k}$ (see \cite[Proposition 3.6]{prrs21}).
					\end{enumerate} 
				\end{remark}
			
				In Proposition~\ref{Lema3} we proved contractibility of the operator $\mathcal{S}_{f}$ on $\mathcal{L}_{k}^{\delta }$, for $\delta >1$. For the sake of applying the fixed point theorem from Proposition~\ref{KorBanach}, in Proposition~\ref{PropositionT} below we prove that the operator $\mathcal{T}_{f}$ is a linear $\frac{1}{2^{\beta -1}}$-homothety on the space $\mathcal{L}_{k}^{\delta ,1}$, for every $\delta \geq 1$.
			
				\begin{prop}[Properties of the operator $\mathcal{T}_{f}$]\label{PropositionT}
					Let $f\in \mathcal{L}_{k}$, $k\in \mathbb{N}$, be such that $f:=\mathrm{id}+az^{\beta }\boldsymbol{\ell}_{1}^{n_{1}}\cdots \boldsymbol{\ell}_{k}^{n_{k}}+\mathrm{h.o.b.}(z)$, $\beta >1$, $a\neq 0$, and $g\in \mathcal{L}_{k}^{0}$ such that $\mathrm{Lb}_{z}(g-\mathrm{id})=az^{\beta }\boldsymbol{\ell}_{1}^{n_{1}}\cdots \boldsymbol{\ell}_{k}^{n_{k}}$. Let $\delta \geq 1$ and the operator $\mathcal{T}_{f}:\mathcal{L}_{k}^{\delta ,1} \to \mathcal{L}_{k}^{\delta ,1}$ be as in \eqref{EquationT}. Then:
					\begin{enumerate}[1., font=\textup, topsep=0.4cm, itemsep=0.4cm, leftmargin=0.6cm]
						\item $\mathcal{T}_{f}$ is a linear $\frac{1}{2^{\beta -1}}$-homothety on the space $\mathcal{L}_{k}^{\delta ,1}$, for every $\delta \geq 1$.
						\item $\mathcal{T}_{f}\big( \mathcal{L}_{k}^{\delta ,1}\big) = \mathcal{L}_{k}^{\delta + \beta -1,2}$.
					\end{enumerate}
				\end{prop}
			
				We repeat here without proof Lemma~4.4 from \cite{prrs21} that is used in the proof of Proposition~\ref{PropositionT} below.
				
				\begin{lem}\cite[Lemma 4.4]{prrs21}\label{Lemma1}\hfill
					\begin{enumerate}[1., font=\textup, topsep=0.4cm, itemsep=0.4cm, leftmargin=0.6cm]
						\item $\int z^{-1}\boldsymbol{\ell}_{1}\cdots \boldsymbol{\ell}_{k}\ dz= -\boldsymbol{\ell}_{k+1}^{-1},\ k\in\mathbb N$,
						\item For $(\delta ,\mathbf{m})\in\mathbb R\times\mathbb Z^k$, $k\in \mathbb{N}$ and $(\delta,\mathbf{m})\neq (-1,\mathbf{1}_{k})$, 
						\begin{align*}
							& \int z^{\delta}\boldsymbol{\ell}_{1}^{m_{1}}\cdots \boldsymbol{\ell}_{k}^{m_{k}}\, dz\in\mathcal{L}_{k}^\infty .
						\end{align*}
					\end{enumerate}
				\end{lem}
			
				\begin{proof}[Proof of Proposition~\ref{PropositionT}]
					1. Let $\delta \geq 1$, $\varepsilon  \in \mathcal{L}_{k}^{\delta ,1}$, $\varepsilon  \neq 0$. Put $\varepsilon  := z^{\alpha }T_{\alpha} + \mathrm{h.o.b.}(z)$, for $\alpha \geq \delta $ and $T_{\alpha }\in \mathcal{B}_{1}\subseteq \mathcal{L}_{k}^{\infty }$. From \eqref{EquationT}, by linearity of $\mathcal{T}_{f}$, we get:
					\begin{align}
						\mathcal{T}_{f}(\varepsilon) = \mathcal{T}_{f}(z^{\alpha }T_{\alpha }) + \mathrm{h.o.b.}(z). \label{Jednn}
					\end{align}
					From \eqref{Jednn} and the fact that $z^{\alpha }T_{\alpha }$ does not contain a term of order $(\beta ,\mathbf{n})$, it follows that:
					\begin{align*}
						& \mathrm{ord}_{z}\left( \mathcal{T}_{f}(\varepsilon )\right) = \mathrm{ord}_{z}(z^{\alpha }T_{\alpha }) + \beta -1 = \mathrm{ord}_{z}(\varepsilon) + \beta -1 .
					\end{align*}
					This implies that:
					\begin{align}
						d_{z}\left( 0,\mathcal{T}_{f}(\varepsilon)\right) = \frac{1}{2^{\beta -1}}d_{z}(0,\varepsilon ) . \label{Homotet}
					\end{align}
					By linearity of $\mathcal{T}_{f}$ and \eqref{Homotet}, we conclude that $\mathcal{T}_{f}$ is a $\frac{1}{2^{\beta -1}}$-homothety. \\
					
					2. Let $\delta \geq 1$. First, we prove that $\mathcal{L}_{k}^{\delta + \beta -1,2} \subseteq \mathcal{T}_{f}(\mathcal{L}_{k}^{\delta , 1})$. Let $g \in \mathcal{L}_{k}^{\delta + \beta -1,2}$ be arbitrary. Consider the equation $\mathcal{T}_{f}(\varepsilon ) = g$, i.e.,
					\begin{align*}
						\varepsilon ' - \frac{(az^{\beta}\boldsymbol{\ell}_{1}^{n_{1}}\cdots \boldsymbol{\ell}_{k}^{n_{k}})' }{az^{\beta}\boldsymbol{\ell}_{1}^{n_{1}}\cdots \boldsymbol{\ell}_{k}^{n_{k}} }\cdot \varepsilon  & = \frac{g}{az^{\beta}\boldsymbol{\ell}_{1}^{n_{1}}\cdots \boldsymbol{\ell}_{k}^{n_{k}}} .
					\end{align*}
					This is a linear ordinary differential equation of order one. Its solutions are given by:
					{\small \begin{align}
						\varepsilon _{C} &= \exp \left( \int \frac{(az^{\beta}\boldsymbol{\ell}_{1}^{n_{1}}\cdots \boldsymbol{\ell}_{k}^{n_{k}})' }{az^{\beta}\boldsymbol{\ell}_{1}^{n_{1}}\cdots \boldsymbol{\ell}_{k}^{n_{k}} } dz \right) \Bigg( C+\int \frac{g}{az^{\beta}\boldsymbol{\ell}_{1}^{n_{1}}\cdots \boldsymbol{\ell}_{k}^{n_{k}}} \exp \left( - \int \frac{(az^{\beta}\boldsymbol{\ell}_{1}^{n_{1}}\cdots \boldsymbol{\ell}_{k}^{n_{k}})' }{az^{\beta}\boldsymbol{\ell}_{1}^{n_{1}}\cdots \boldsymbol{\ell}_{k}^{n_{k}}} dz \right) dz \Bigg) \nonumber \\
						&= az^{\beta}\boldsymbol{\ell}_{1}^{n_{1}}\cdots \boldsymbol{\ell}_{k}^{n_{k}} \cdot \Big( C+\int \frac{g}{(az^{\beta}\boldsymbol{\ell}_{1}^{n_{1}}\cdots \boldsymbol{\ell}_{k}^{n_{k}})^{2}} dz \Big) , \label{BitnaJednadzba}
					\end{align}
					}for $C\in \mathbb{R} $. Now we choose $C\in \mathbb{R} $, such that $\varepsilon_{C}$ does not contain a term of order $(\beta , \mathbf{n})$ and put $\varepsilon := \varepsilon_{C}$.
				
					The integration above is the usual integration \textit{monomial by monomial}. Since $g \in \mathcal{L}_{k}^{\delta + \beta -1,2}$, then $g$ does not contain a term of order $(2\beta -1,2\mathbf{n}+\mathbf{1}_{k})$. Therefore, the logarithmic transseries $\frac{g}{(z^{\beta }\boldsymbol{\ell}_{1}^{n_{1}}\cdots \boldsymbol{\ell}_{k}^{n_{k}})^{2}} $ does not contain a term of order $(-1 , \mathbf{1}_{k})$. By Lemma~\ref{Lemma1}, we conclude that
					\begin{align*}
						& \int \frac{g}{(az^{\beta}\boldsymbol{\ell}_{1}^{n_{1}}\cdots \boldsymbol{\ell}_{k}^{n_{k}})^{2}} dz 
					\end{align*}
					belongs to $\mathcal{L}_{k}^{\infty }$. By the formal integration we get that:
					\begin{align*}
						\mathrm{ord}_{z} \Big( \int \frac{g}{(az^{\beta}\boldsymbol{\ell}_{1}^{n_{1}}\cdots \boldsymbol{\ell}_{k}^{n_{k}})^{2}} dz \Big) & =  \mathrm{ord}_{z}(g)-2\beta +1 ,
					\end{align*}
					which implies, by \eqref{BitnaJednadzba}, that:
					\begin{align*}
						\mathrm{ord}_{z} (\varepsilon ) & = \mathrm{ord}_{z}(g)- \beta +1 .
					\end{align*}
					Since $g \in \mathcal{L}_{k}^{\delta + \beta -1 , 2}$, it follows that $\varepsilon  \in \mathcal{L}_{k}^{\delta }$. Recall that $\varepsilon $ is choosen such that $\varepsilon $ does not contain a term of order $(\beta , \mathbf{n})$, which implies that $\varepsilon  \in \mathcal{L}_{k}^{\delta ,1}$.
					
					It is left to prove that $\mathcal{T}_{f}(\mathcal{L}_{k}^{\delta ,1}) \subseteq \mathcal{L}_{k}^{\delta + \beta -1,2}$. For an arbitrary term $bz^{\alpha }\boldsymbol{\ell}_{1}^{m_{1}}\cdots \boldsymbol{\ell}_{k}^{m_{k}} \in \mathcal{L}_{k}^{\delta ,1}$, $b\neq 0$, we get that:
					\begin{align}
						& \mathcal{T}_{f}(bz^{\alpha }\boldsymbol{\ell}_{1}^{m_{1}}\cdots \boldsymbol{\ell}_{k}^{m_{k}} ) \nonumber \\
						&= az^{\beta }\boldsymbol{\ell}_{1}^{n_{1}}\cdots \boldsymbol{\ell}_{k}^{n_{k}} \cdot (bz^{\alpha }\boldsymbol{\ell}_{1}^{m_{1}}\cdots \boldsymbol{\ell}_{k}^{m_{k}} )' - (az^{\beta }\boldsymbol{\ell}_{1}^{n_{1}}\cdots \boldsymbol{\ell}_{k}^{n_{k}})' \cdot bz^{\alpha }\boldsymbol{\ell}_{1}^{m_{1}}\cdots \boldsymbol{\ell}_{k}^{m_{k}} \nonumber \\
						& = abz^{\alpha + \beta -1} \cdot ((\alpha - \beta )\boldsymbol{\ell}_{1}^{n_{1}+m_{1}}\cdots \boldsymbol{\ell}_{k}^{n_{k}+m_{k}} + (m_{1}-n_{1})\boldsymbol{\ell}_{1}^{n_{1}+m_{1}+1}\boldsymbol{\ell}_{2}^{n_{2}+m_{2}}\cdots \boldsymbol{\ell}_{k}^{n_{k}+m_{k}}+\cdots \nonumber \\
						& \cdots + (m_{k}-n_{k})\boldsymbol{\ell}_{1}^{n_{1}+m_{1}+1}\cdots \boldsymbol{\ell}_{k}^{n_{k}+m_{k}+1}) . \label{Equation11}
					\end{align}
					Since $z^{\alpha }\boldsymbol{\ell}_{1}^{m_{1}}\cdots \boldsymbol{\ell}_{k}^{m_{k}}\neq z^{\beta }\boldsymbol{\ell}_{1}^{n_{1}}\cdots \boldsymbol{\ell}_{k}^{n_{k}}$, from \eqref{Equation11} it follows that a residual term is not in the image of $\mathcal{T}_{f}$. In particular, $\mathcal{T}_{f}(bz^{\alpha }\boldsymbol{\ell}_{1}^{m_{1}}\cdots \boldsymbol{\ell}_{k}^{m_{k}}) \in \mathcal{L}^{\delta + \beta -1,2}$, for any term $bz^{\alpha }\boldsymbol{\ell}_{1}^{m_{1}}\cdots \boldsymbol{\ell}_{k}^{m_{k}} \in \mathcal{L}_{k}^{\delta ,1}$, $b\neq 0$. Since $\mathcal{T}_{f}$ is linear, we conclude that $\mathcal{T}_{f}(\mathcal{L}_{k}^{\delta ,1})\subseteq \mathcal{L}_{k}^{\delta + \beta -1,2}$.
				\end{proof}
				
				Unlike $\mathcal{T}_{f}(\mathcal{L}_{k}^{\delta ,1})$ described in Proposition~\ref{PropositionT}, the logarithmic transseries in the image $\mathcal{S}_{f}(\mathcal{L}_{k}^{\delta ,1})$ in general contain residual terms (i.e., terms of order $(2\beta -1,2\mathbf{n}+\mathbf{1}_{k})$). Therefore we cannot directly apply Proposition~\ref{KorBanach} and we split the proof of step $(a.2)$ into additional two substeps (see Figure~\ref{FigureSketchProof}): \\
				
					\begin{itemize}
						\item[(a.2.1)] Using Proposition~\ref{KorBanach}, we obtain $c\in \mathbb{R}$ and a parabolic logarithmic transseries $\varphi _{2,1} \in \mathcal{L}_{k}^{0}$ such that
						\begin{align}
							\varphi _{2,1} \circ f \circ \varphi _{2,1}^{-1} = \mathrm{id} + az^{\beta }\boldsymbol{\ell }_{1}^{n_{1}}\cdots \boldsymbol{\ell}_{k}^{n_{k}} + c\mathrm{Res}(f) + \mathrm{h.o.b.}(z) , \label{EqConj1}
						\end{align}
						which means that we eliminate all blocks of $f$ between the first and the residual block and all terms in the residual block except maybe the residual term. \\
						
						\item[(a.2.2)] Using Proposition~\ref{KorBanach}, we obtain a parabolic logarithmic transseries $\varphi _{2,2} \in \mathcal{L}_{k}^{0}$ such that
						\begin{align}
							\varphi _{2,2}\circ (\varphi _{2,1} \circ f \circ \varphi _{2,1}^{-1}) \circ \varphi _{2,2}^{-1} = \mathrm{id} + az^{\beta }\boldsymbol{\ell }_{1}^{n_{1}}\cdots \boldsymbol{\ell}_{k}^{n_{k}} + c\mathrm{Res}(f) , \label{EqConj3}
						\end{align}
						which means that we eliminate all blocks in $\varphi _{2,1} \circ f \circ \varphi _{2,1}^{-1} $ after the residual block (i.e., the $2\beta -1$-block). \\
					\end{itemize}
				
					Finally, $\varphi _{2}:=\varphi _{2,2} \circ \varphi _{2,1} $ is a solution of the normalization equation $\varphi _{2}\circ f\circ \varphi _{2}^{-1}=\mathrm{id}+az^{\beta }\boldsymbol{\ell }_{1}^{n_{1}}\cdots \boldsymbol{\ell }_{k}^{n_{k}}+c\mathrm{Res}(f)$. \\
					
					In the proofs of steps $(a.2.1)$ and $(a.2.2)$, we again transform equations \eqref{EqConj1} and \eqref{EqConj3} to the appropriate fixed point equations and then use the fixed point theorem from Proposition~\ref{KorBanach} to prove the existence of solutions.
					
					\begin{proof}[Proof of step (a.2.1).]
						Let $f\in \mathcal{L}_{k}$, $k\in \mathbb{N}$, such that $f := \mathrm{id}+ az^{\beta }\boldsymbol{\ell}_{1}^{n_{1}}\cdots \boldsymbol{\ell}_{k}^{n_{k}} + \mathrm{h.o.b.}(z)$, $\beta >1$, $a\neq 0$, be prenormalized. We prove that there exists $c\in \mathbb{R} $ and $\varphi _{2,1} \in \mathcal{L}_{k}^{0}$, such that:
						\begin{align*}
							\varphi _{2,1} \circ f \circ \varphi _{2,1}^{-1} & = \mathrm{id} +az^{\beta }\boldsymbol{\ell}_{1}^{n_{1}}\cdots \boldsymbol{\ell}_{k}^{n_{k}}+c\mathrm{Res}(f) + \mathrm{h.o.b.}(z).
						\end{align*}
						Put
						\begin{align*}
							g & := \mathrm{id} + az^{\beta }\boldsymbol{\ell }_{1}^{n_{1}}\cdots \boldsymbol{\ell}_{k}^{n_{k}} .
						\end{align*}
						Suppose for a moment that we want to solve the conjugacy equation $\varphi \circ f\circ \varphi^{-1}=g$, which is, by Proposition~\ref{Lema3}, equivalent to the fixed point equation $\mathcal{T}_{f}(\varepsilon ) = \mathcal{S}_{f}(\varepsilon )$, for $\varepsilon :=\varphi -\mathrm{id}$, $\varepsilon  \in \mathcal{L}_{k}^{1}$, such that $\mathrm{ord}(\varepsilon) > (1,\mathbf{0}_{k})$. The operators $\mathcal{S}_{f}$ and $\mathcal{T}_{f}$ here are defined in \eqref{DefS} and \eqref{Tdefined}, respectively. Put
						\begin{align*}
							\gamma & :=\mathrm{ord}_{z}(f- \mathrm{id}-az^{\beta }\boldsymbol{\ell }_{1}^{n_{1}}\cdots \boldsymbol{\ell}_{k}^{n_{k}} ) .
						\end{align*}
						Notice that $\gamma > \beta >1$. If the fixed point equation $\mathcal{T}_{f}(\varepsilon ) = \mathcal{S}_{f}(\varepsilon )$ is solvable, comparing the orders of the left and the right-hand side of the fixed point equation $\mathcal{T}_{f}(\varepsilon ) = \mathcal{S}_{f}(\varepsilon )$, since $\mathrm{ord}_{z}(\mathcal{S}_{f}(\varepsilon )) = \gamma $, it necessarily follows that $\mathrm{ord}_{z}(\mathcal{T}_{f}(\varepsilon )) = \gamma $. By Proposition~\ref{PropositionT}, operator $\mathcal{T}_{f}$ is a linear $\frac{1}{2^{\beta -1}}$-homothety on the space $\mathcal{L}_{k}^{1,1}$. To conclude, if $\varepsilon$ is a solution of the fixed point equation, it necessarily follows that $\mathrm{ord}_{z}(\varepsilon )=\gamma - (\beta -1)>1$. Therefore, take:
						\begin{align*}
							\delta & := \gamma - (\beta -1) .
						\end{align*}
						By Proposition~\ref{PropositionT} and Proposition~\ref{Lema3}, the operator $\mathcal{T}_{f}$ is a linear $\frac{1}{2^{\beta -1}}$-homothety on the space $\mathcal{L}_{k}^{\delta , 1}$ and operator $\mathcal{S}_{f}$ is a $\frac{1}{2^{\rho }}$-contraction on the space $\mathcal{L}_{k}^{\delta ,1}$, where $\rho := \min \left\lbrace \gamma - 1,2(\beta -1)\right\rbrace $. By Proposition~\ref{Lema3}, it follows that $\mathcal{S}_{f}(\mathcal{L}_{k}^{\delta , 1}) \subseteq \mathcal{L}_{k}^{\gamma }$, and, by Proposition~\ref{PropositionT}, $\mathcal{T}_{f}(\mathcal{L}_{k}^{\delta , 1})=\mathcal{L}_{k}^{\gamma , 2}$. Notice that, in general, $\mathcal{S}_{f}(\mathcal{L}_{k}^{\delta , 1}) \nsubseteq \mathcal{L}_{k}^{\gamma , 2}$. So, in order to apply the fixed point theorem from Proposition~\ref{KorBanach}, we compose the operator $\mathcal{S}_{f}$ with the projection operator $\left[ \, \cdot \, \right] : \mathcal{L}_{k}^{\gamma } \to \mathcal{L}_{k}^{\gamma ,2}$, where $[ h]$ is the transseries obtained from $h$ by removing its term of order $(2\beta -1,2\mathbf{n}+\mathbf{1}_{k})$, $h \in \mathcal{L}_{k}^{\gamma }$. So, instead of the fixed point equation $\mathcal{T}_{f}(\varepsilon) = \mathcal{S}_{f}(\varepsilon)$, we solve the modified fixed point equation $\mathcal{T}_{f}(\varepsilon) = \left[ \mathcal{S}_{f} (\varepsilon) \right] $, $\varepsilon  \in \mathcal{L}_{k}^{\delta ,1}$. Note that the projection operator $\left[ \, \cdot \, \right] $ is $1$-Lipschitz on $\mathcal{L}_{k}^{\gamma }$, i.e.
						\begin{align}
							d_{z}\left( \left[ g_{1}\right] , \left[ g_{2} \right] \right) & \leq d_{z}\left( g_{1} , g_{2} \right) , \quad g_{1} , g_{2} \in \mathcal{L}_{k}^{\gamma } . \label{Dva}
						\end{align}
						We define the operator $\left[ \mathcal{S}_{f} \right] :\mathcal{L}_{k}^{\delta , 1} \to \mathcal{L}_{k}^{\gamma , 2}$, as a composition $\left[ \mathcal{S}_{f}\right] =\left[ \, \cdot \, \right] \circ \mathcal{S}_{f}$. Using \eqref{Dva} and Proposition~\ref{Lema3}, we conclude that the operator $\left[ \mathcal{S}_{f}\right] $ is a $\frac{1}{2^{\rho }}$-contraction on $\mathcal{L}_{k}^{\delta ,1}$. Furthermore, $\left[ \mathcal{S}_{f}\right] (\mathcal{L}_{k}^{\delta , 1})\subseteq \mathcal{L}_{k}^{\gamma , 2}$, which is, by Proposition~\ref{PropositionT}, equal to $\mathcal{T}_{f}(\mathcal{L}_{k}^{\delta ,1})$. By definition, $\rho > \beta -1$, which implies that $\frac{1}{2^{\rho }} < \frac{1}{2^{\beta -1}}$. By the fixed point theorem from Proposition~\ref{KorBanach}, there exists a unique $\varepsilon  \in \mathcal{L}_{k}^{\delta , 1}$, such that $\mathcal{T}_{f}(\varepsilon ) =\left[ \mathcal{S}_{f}\right] (\varepsilon ) $ (see Figure~\ref{Figure1}).
						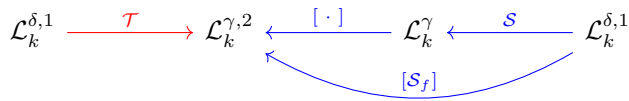
\begin{figure}[!h]
							$$
								\begin{tikzcd}[row sep=tiny,column sep=huge]
									\mathcal{L}_{k}^{\delta , 1} \arrow[r,red,"\mathcal{T}",->] & \mathcal{L}_{k}^{\gamma , 2} & \mathcal{L}_{k}^{\gamma } \arrow[l,blue,swap,"\textrm{[ } \cdot \textrm{ ]}",->] & \mathcal{L}_{k}^{\delta , 1} \arrow[l,blue,swap,"\mathcal{S}",->] \arrow[ll,blue,bend left=30,swap,"\textrm{[}\mathcal{S}_{f} \textrm{]}",->] 
								\end{tikzcd}
							$$
							\caption{Relation between the operators $\left[ \mathcal{S}_{f}\right] $ and $\mathcal{T}_{f}$.}\label{Figure1}
						\end{figure}
						
						\vspace{4mm}
						
						Note that $\left[ \mathcal{T}_{f}(\varepsilon)\right] = \mathcal{T}_{f}(\varepsilon )$, because $\mathcal{T}_{f}(\mathcal{L}_{k}^{\delta , 1}) = \mathcal{L}_{k}^{\gamma , 2}$. Therefore, solving the equation $\mathcal{T}_{f}(\varepsilon )=\left[ \mathcal{S}_{f}\right] (\varepsilon )$ in $\mathcal{L}_{k}^{\delta ,1}$ is equivalent to solving the equation:
						\begin{align*}
							\left[ \mathcal{T}_{f}(\varepsilon )\right] & = \left[ \mathcal{S}_{f} (\varepsilon )\right] .
						\end{align*}
						Operator $\left[ \, \cdot \, \right] $ is linear, so the above equation is equivalent to
						\begin{align*}
							\left[ \mathcal{T}_{f}(\varepsilon ) - \mathcal{S}_{f}(\varepsilon )\right] & = 0 ,
						\end{align*}
						which is equivalent to the existence of $c\in \mathbb{R} $, such that
						\begin{align}
							\mathcal{T}_{f}(\varepsilon ) - \mathcal{S}_{f}(\varepsilon ) = c\mathrm{Res}(f) . \label{Jenadzbica}
						\end{align}
						Finally, put $\varphi _{2,1} := \mathrm{id} + \varepsilon $. As in \eqref{Equation2}, we get:
						\begin{align}
							\mathcal{T}_{f}(\varepsilon ) - \mathcal{S}_{f}(\varepsilon ) = \varphi _{2,1} \circ f - g \circ \varphi _{2,1} . \label{Jednakost}
						\end{align}
						From \eqref{Jenadzbica} and \eqref{Jednakost}, it follows that there exists $c\in \mathbb{R}$ such that:
						\begin{align}
							\varphi _{2,1} \circ f - g \circ \varphi _{2,1} = c\mathrm{Res}(f) . \label{Jednakost1} 
						\end{align}
						Now, by transforming \eqref{Jednakost1}, since $\mathrm{ord}_{z}(\varepsilon )>1$, by the Taylor Theorem (\cite[Proposition 3.3]{prrs21}), we get that:
						\begin{align*}
							\varphi _{2,1} \circ f \circ \varphi _{2,1}^{-1} &= g + c\mathrm{Res}(f) \circ \varphi _{2,1}^{-1} \\
							&= \mathrm{id} +az^{\beta }\boldsymbol{\ell}_{1}^{n_{1}}\cdots \boldsymbol{\ell}_{k}^{n_{k}} + c\mathrm{Res}(f) + \mathrm{h.o.b.}(z) .
						\end{align*}
						Thus, we have eliminated all blocks in $f$ of order (in $z$) between $\beta $ and $2\beta -1$, and all terms in the residual block, except the residual term.
					\end{proof}
				
					\begin{proof}[Proof of step (a.2.2).]
						Let $f$, $\varphi_{2,1}$ and $a,c\in \mathbb{R} $ be as in step $(a.2.1)$ above. We prove that there exists $\varphi _{2,2} \in \mathcal{L}_{k}^{0}$, such that:
						\begin{align*}
							\varphi _{2,2} \circ (\varphi _{2,1} \circ f \circ \varphi _{2,1}^{-1}) \circ \varphi _{2,2}^{-1} = \mathrm{id} +az^{\beta }\boldsymbol{\ell}_{1}^{n_{1}}\cdots \boldsymbol{\ell}_{k}^{n_{k}}+c\mathrm{Res}(f) .
						\end{align*}
						Put:
						\begin{align*}
							h & :=\varphi _{2,1} \circ f \circ \varphi _{2,1}^{-1} , \\
							g & := \mathrm{id} +az^{\beta }\boldsymbol{\ell}_{1}^{n_{1}}\cdots \boldsymbol{\ell}_{k}^{n_{k}} + c\mathrm{Res}(f) , \\
							\gamma & := \mathrm{ord}_{z}(h-g) , \\
							\delta & := \gamma - (\beta -1) .  
						\end{align*}
						The conjugacy equation
						\begin{align}
							\varphi_{2,2}\circ h \circ \varphi_{2,2}^{-1}=g , \label{EquationConj3}
						\end{align}
						is, by Lemma~\ref{Lemma2}, equivalent to the fixed point equation $\mathcal{S}_{f}(\varepsilon ) = \mathcal{T}_{f}(\varepsilon )$, where $\varepsilon  :=\varphi _{2,2}- \mathrm{id}$. As in the proof of step $(a.2.1)$, if $\varepsilon $ is a solution of the equation $\mathcal{S}_{f}(\varepsilon ) = \mathcal{T}_{f}(\varepsilon )$, we conclude that $\mathrm{ord}_{z}(\mathcal{T}_{f}(\varepsilon ))= \gamma $. Therefore, if $\varepsilon $ is a solution, necessarily $\mathrm{ord}_{z}(\varepsilon )= \delta $. Since $\gamma > 2\beta -1$ and $\delta > \beta $, it follows that $\mathcal{L}_{k}^{\delta } =\mathcal{L}_{k}^{\delta ,1} $ and $\mathcal{L}_{k}^{\gamma }=\mathcal{L}_{k}^{\gamma , 2}$. By Proposition~\ref{PropositionT}, the operator $\mathcal{T}_{f}:\mathcal{L}_{k}^{\delta } \to \mathcal{L}_{k}^{\gamma }$ is a $\frac{1}{2^{\beta -1}}$-homothety and $\mathcal{T}_{f}(\mathcal{L}_{k}^{\delta }) = \mathcal{L}_{k}^{\gamma }$. By Proposition~\ref{Lema3}, the operator $\mathcal{S}_{f}:\mathcal{L}_{k}^{\delta } \to \mathcal{L}_{k}^{\gamma }$ is a $\frac{1}{2^{\rho }}$-contraction, with $\rho := \min \left\lbrace \gamma -1, 2(\beta -1)\right\rbrace $. Since $\mathcal{S}_{f}(\mathcal{L}_{k}^{\delta }) \subseteq \mathcal{L}_{k}^{\gamma } = \mathcal{T}_{f}(\mathcal{L}_{k}^{\delta })$ and $\frac{1}{2^{\rho }} < \frac{1}{2^{\beta -1}}$, by the fixed point theorem from Proposition~\ref{KorBanach}, it follows that there exists a unique $\varepsilon  \in \mathcal{L}_{k}^{\delta }$, such that $\mathcal{T}_{f}(\varepsilon) = \mathcal{S}_{f}(\varepsilon)$. Now, $\varphi _{2,2}:= \mathrm{id} + \varepsilon $ is a solution of the conjugacy equation \eqref{EquationConj3}.
					\end{proof}

					\subsection{Proof of case $\mathrm{ord}_{z}(f- \mathrm{id}) =1$ of statement 1}\label{sec:proofAb}
					
					Let $f\in \mathcal{L}_{k}^{0}$, $k\in \mathbb{N}_{\geq 1}$, such that $f=\mathrm{id}+zR+\mu $, $R\in \mathcal{B}_{\geq 1}^{+}\setminus \left\lbrace 0\right\rbrace \subseteq \mathcal{L}_{k}$, $\mu \in \mathcal{L}_{k}^{\alpha }$, $\alpha :=\mathrm{ord}_{z}(f-\mathrm{id}-zR)>1$. Write $(1,\mathbf{n}) := \mathrm{ord} \, (f-\mathrm{id})$, and $\mathbf{n}=(\mathbf{0}_{m-1},n_{m},\ldots ,n_{k})$.
					
					That means that the first term of $f-\mathrm{id}$ is equal to $a_{1,\mathbf{n}}z\boldsymbol{\ell }_{m}^{n_{m}}\cdots \boldsymbol{\ell }_{k}^{n_{k}}$, where $a_{1,\mathbf{n}} \neq 0$, $n_{m}\geq 1$, for some $1\leq m\leq k$. Note that the residual term here, as defined in Section~\ref{SectionMain}, is also of order $1$ in $z$.
					
					We consider a general conjugacy equation:
					\begin{align}
						\varphi \circ f\circ \varphi ^{-1} & =g, \quad \varphi \in \mathcal{L}_{k}^{0}, \label{ConjEquationCaseB}
					\end{align}
					where $g:=\mathrm{id}+zT+\mathrm{h.o.b.}(z)$, $T\in \mathcal{B}_{\geq 1}^{+}\subseteq \mathcal{L}_{k}$. By Subsection~\ref{SubsubsectionDifference}, we cannot apply Proposition~\ref{KorBanach} directly to the fixed point equation obtained from the conjugacy equation as described in Lemma~\ref{Lemma2}. Since $\mathrm{ord}_{z}(f-\mathrm{id})=1$, we have to transform the conjugacy equation $\varphi \circ f\circ \varphi ^{-1}=g$, $\varphi \in \mathcal{L}_{k}^{0}$, in order to apply Proposition~\ref{KorBanach}. By Proposition~\ref{LeadingBlockLemma}, which gives a necessary condition for solvability of conjugacy equation \eqref{ConjEquationCaseB}, it follows that necessarily $\left[ f\right] _{1,\mathbf{m}}=\left[ g\right] _{1,\mathbf{m}}$, $\mathbf{n}\leq \mathbf{m}\leq \mathbf{n}'$, where $\mathbf{n}'$ is as defined in \eqref{NCrtica}. Therefore, let $L$ be a formal sum (possibly infinite) of all terms in $R$ of order $(0,\mathbf{m})$, for $\mathbf{n}\leq \mathbf{m}\leq \mathbf{n}'$. Note that $L$ is a formal sum of all terms in $R$ that cannot be eliminated, nor changed. \\
					
					Let $\mathcal{L}_{k}^{<(\gamma , n_{1},\ldots ,n_{m})}\subseteq \mathcal{L}_{k}$, for $(\gamma , n_{1},\ldots ,n_{m})\in \mathbb{R}_{>0}\times \mathbb{Z} ^{m}$, $1\leq m\leq k$, be the set of all transseries $h \in \mathcal{L}_{k}$, such that every term in $h$ is of strictly smaller order than $(\gamma , n_{1},\ldots ,n_{m})$ in the first $m+1$ variables $z, \boldsymbol{\ell }_{1},\ldots , \boldsymbol{\ell }_{m}$. 
					
					Let
					\begin{align*}
						\mathcal{P}_{<(\gamma , n_{1},\ldots ,n_{m})} : \mathcal{L}_{k} \to \mathcal{L}_{k}^{<(\gamma , n_{1},\ldots ,n_{m})}
					\end{align*}
					be the projection operator from space $\mathcal{L}_{k}$ to the space $\mathcal{L}_{k}^{<(\gamma , n_{1},\ldots ,n_{m})}$. Similarly we define spaces $\mathcal{L}_{k}^{\leq (\gamma , n_{1},\ldots ,n_{m})}$, $\mathcal{L}_{k}^{>(\gamma , n_{1},\ldots ,n_{m})}$, $\mathcal{L}_{k}^{\geq (\gamma , n_{1},\ldots ,n_{m})}\subseteq \mathcal{L}_{k}$ and projection operators:
					\begin{align*}
						\mathcal{P}_{\leq (\gamma , n_{1},\ldots ,n_{m})} : \mathcal{L}_{k} \to \mathcal{L}_{k}^{\leq (\gamma , n_{1},\ldots ,n_{m})} , \\
						\mathcal{P}_{>(\gamma , n_{1},\ldots ,n_{m})} : \mathcal{L}_{k} \to \mathcal{L}_{k}^{>(\gamma , n_{1},\ldots ,n_{m})} , \\
						\mathcal{P}_{\geq (\gamma , n_{1},\ldots ,n_{m})} : \mathcal{L}_{k} \to \mathcal{L}_{k}^{\geq (\gamma , n_{1},\ldots ,n_{m})} .
					\end{align*}
					
					As in case $(a)$, Proposition~\ref{KorBanach} cannot be applied directly. Therefore, we proceed in three steps $(b.1)-(b.3)$ described below (see Figure~\ref{FigureSketchProof}). \\
					
					Let $f\in \mathcal{L}_{k}^{0}$, $k\in \mathbb{N}_{\geq 1}$, such that $f=\mathrm{id}+zL+\mathrm{h.o.t.}$, for $L$ defined in the Main Theorem. Suppose that $(1,\mathbf{n}):=\mathrm{ord}(f-\mathrm{id})$ and $\mathbf{n}=(\mathbf{0}_{m-1},n_{m},\ldots ,n_{k})$. \\
					
					\begin{itemize}
						\item[\emph{Step (b.1)}] We eliminate all the terms in the leading block of $f-\mathrm{id}$ which are not in $zL$ and are of order less than or equal to $(\mathbf{1}_{m},n_{m}+1)$ in the first $m+1$ variables $z, \boldsymbol{\ell }_{1}, \ldots , \boldsymbol{\ell }_{m}$. That is, we find a solution (not unique) $\varphi _{1} \in \mathcal{L}_{k}^{0}$ of the equation
						\begin{align}
							\mathcal{P}_{\leq (\mathbf{1}_{m},n_{m}+1)}(\varphi _{1}\circ f\circ \varphi _{1}^{-1}) &= \mathrm{id}+ zL . \label{FirstConj}
						\end{align}
						Note that we can skip this step if $m=k$ because in that case $\mathcal{P}_{\leq (\mathbf{1}_{k},n_{k}+1)}(f)=z+zL$.
						
						Moreover, we prove in Remark~\ref{RemCaseB1} that $\varphi _{1}$ is a unique solution of \eqref{FirstConj} if we impose the \emph{canonical form} $\varphi _{1}=\mathrm{id}+zS$, $S\in \mathcal{B}_{\geq m+1}^{+}\subseteq \mathcal{L}_{k}$. \\
						
						\item[\emph{Step (b.2)}] Put $r:=\mathrm{ord}(\mathrm{Res}(f))=(\mathbf{1}_{m},2n_{m}+1,\ldots ,2n_{k}+1)$. We eliminate all the terms in the leading block of $\varphi _{1}\circ f\circ \varphi _{1}^{-1}-\mathrm{id}$ which are not in $zL$ up to the residual term. That is, we find a solution (not unique) $\varphi _{2}\in \mathcal{L}_{k}^{0}$ of the equation
						\begin{align}
							\mathcal{P}_{<r}(\varphi _{2}\circ (\varphi _{1}\circ f\circ \varphi _{1}^{-1})\circ \varphi _{2}^{-1}) &= \mathrm{id}+ zL . \label{SecondConj}
						\end{align}
						By \eqref{SecondConj}, it follows that there exists $c\in \mathbb{R}$ (which is unique and given explicitely in Subsection~\ref{sec:ProofMinimality}) such that:
						\begin{align*}
							\varphi _{2}\circ (\varphi _{1}\circ f\circ \varphi _{1}^{-1})\circ \varphi _{2}^{-1} &= \mathrm{id}+ zL +c\mathrm{Res} \, (f)+\mathrm{h.o.t.}
						\end{align*}
						
						Moreover, we prove in Remark~\ref{RemCaseB2} that $\varphi _{2}$ is a unique solution of \eqref{SecondConj} if we impose the \emph{canonical form} $\varphi _{2}=\mathrm{id}+zS$, $S\in \widetilde{\mathcal{B}}$, where $\widetilde{\mathcal{B}}$ is the set of all logarithmic transseries in $\mathcal{B}_{m}^{+}\subseteq \mathcal{L}_{k}$ which contain only terms of strictly smaller order than $(\mathbf{0}_{m},n_{m},\ldots ,n_{k})=(0,\mathbf{n})$. \\
						
						\item[\emph{Step (b.3)}] Let $c\in \mathbb{R}$ be the coefficient of the residual term in $\varphi _{2}\circ \varphi _{1}\circ f\circ \varphi _{1}^{-1}\circ \varphi _{2}^{-1}$. We eliminate all the terms in $\varphi _{2}\circ (\varphi _{1}\circ f\circ \varphi _{1}^{-1})\circ \varphi _{2}^{-1}) - \mathrm{id}+ zL$ except the residual term. That is, we find a solution (not unique) $\varphi _{3}\in \mathcal{L}_{k}^{0}$ of the equation
						\begin{align}
							\varphi _{3}\circ (\varphi _{2}\circ (\varphi _{1}\circ f\circ \varphi _{1}^{-1})\circ \varphi _{2}^{-1}) \circ \varphi _{3}^{-1} &= \mathrm{id}+ zL + c\mathrm{Res}(f) . \label{ThirdConj}
						\end{align}
						
						Moreover, we prove that $\varphi _{3}$ is a unique solution of \eqref{ThirdConj} if we impose the \emph{canonical form} $\varphi _{3}=\mathrm{id}+zS+\varepsilon $, where $S\in \mathcal{B}_{\geq 1}^{+}\subseteq \mathcal{L}_{k}$, $\mathrm{ord} \, (zS) > \mathrm{ord} \, (\mathrm{Res} \, (f))$, and $\varepsilon \in \mathcal{L}_{k}$, $\mathrm{ord}_{z} \, (\varepsilon ) \geq \mathrm{ord}_{z} \, (f_{1}-\mathrm{Lb}_{z}(f_{1}))$, where $f_{1}:=\varphi _{2}\circ (\varphi _{1}\circ f\circ \varphi _{1}^{-1})\circ \varphi _{2}^{-1}$. \\
					\end{itemize}
					
					The general idea in all steps is to transform the conjugacy equation, using the Taylor Theorem (see \cite[Proposition 3.3]{prrs21}), to an equivalent fixed point equation or differential equation that is then solved using the fixed point theorem from Proposition~\ref{KorBanach}. The crucial point of this approach is solving various nonlinear differential equations on differential algebras of blocks, using fixed point method described in Propositions~\ref{Prop3},~\ref{Prop4} and~\ref{Lemma4}. \\
					
					We illustrate steps (b.1)-(b.3) on the following example.
					\begin{example}
						Let $f\in \mathcal{L}_{3}$ be given by:
						\begin{align*}
							f &=z+\overbrace{z\boldsymbol{\ell }_{2}\Big( \sum _{i=-2}^{+\infty }\boldsymbol{\ell }_{3}^{i}\Big) +z\boldsymbol{\ell }_{1}+z\boldsymbol{\ell }_{1}\boldsymbol{\ell }_{2}^{2}\boldsymbol{\ell }_{3}^{-2}}^{zL} \\
							&+\underbrace{z\boldsymbol{\ell }_{1}\boldsymbol{\ell }_{2}^{2}(\boldsymbol{\ell }_{3}^{-1}+1+\boldsymbol{\ell }_{3}^{6}+\boldsymbol{\ell }_{3}^{14})}_{\textrm{elimination of this part in step }(b.1)} +\underbrace{z\boldsymbol{\ell }_{1}\boldsymbol{\ell }_{2}^{3}\boldsymbol{\ell }_{3}^{-3}}_{\textrm{residual term}}+z^{3}\boldsymbol{\ell }_{2}^{-5}+\mathrm{h.o.t.}
						\end{align*}
						Note that $\mathbf{n}=(0,1,-2)$ and $m=2$. Therefore, by \eqref{NCrtica}, it follows that $\mathbf{n}'=(1,2,-2)$ and
						\begin{align*}
							L=\boldsymbol{\ell }_{2}\sum _{i=-2}^{+\infty }\boldsymbol{\ell }_{3}^{i}+\boldsymbol{\ell }_{1}+\boldsymbol{\ell }_{1}\boldsymbol{\ell }_{2}^{2}\boldsymbol{\ell }_{3}^{-2} .
						\end{align*}
						Furthermore, $\mathrm{Res}(f)=z\boldsymbol{\ell }_{1}\boldsymbol{\ell }_{2}^{3}\boldsymbol{\ell }_{3}^{-3}$. In step (b.1) we find a solution $\varphi _{1}\in \mathcal{L}_{3}^{0}$ of the equation \eqref{FirstConj}.
						
						Now,
						\begin{align*}
							\varphi _{1}\circ f\circ \varphi _{1}^{-1} &=z+zL+\underbrace{\sum _{u_{3}=v_{3}}^{-4}a_{1,1,3,u_{3}}z\boldsymbol{\ell }_{1}\boldsymbol{\ell }_{2}^{3}\boldsymbol{\ell }_{3}^{u_{3}}}_{\textrm{elimination of this part in step }(b.2)}+b\mathrm{Res}(f)+\mathrm{h.o.t.}, 
						\end{align*}
						for some $v_{3}\leq -4$. Some of these coefficients $a_{1,1,3,u_{3}}$ and $b$ can be equal to zero. \\
						Note that $r=\mathrm{ord}(\mathrm{Res}(f))=(1,1,3,-3)$. Now, we apply step (b.2) and obtain $\varphi _{2}\in \mathcal{L}_{3}^{0}$ and $c\in \mathbb{R}$ such that \eqref{SecondConj} holds, i.e.
						\begin{align*}
							\varphi _{2}\circ (\varphi _{1}\circ f\circ \varphi_{1}^{-1})\circ \varphi _{2}^{-1} = z+zL+c\mathrm{Res}(f)+\underbrace{\mathrm{h.o.t.}}_{\textrm{elimination }(b.3)} .
						\end{align*}
					
						Finally, we apply step (b.3) to eliminate $\mathrm{h.o.t}$ by $\varphi _{3}\in \mathcal{L}_{3}^{0}$, not changing any terms in $\mathrm{id}+zL+c\mathrm{Res}(f)$.
					\end{example}

					\subsubsection{Proof of step (b.1)}
					
					The following Lemma~\ref{CaseBLemma7} and Proposition~\ref{Prop3} are the main parts of the proof of step (b.1). As explained above, if $m=k$ we skip this step. Therefore, we assume here that $m\leq k-1$ and $k\geq 2$.
					
					\begin{lem}[Transforming equation \eqref{FirstConj} to a differential equation]\label{CaseBLemma7}
						Let $f\in \mathcal{L}_{k}^{0}$, $k\in \mathbb{N}_{\geq 1}$, such that $f=\mathrm{id}+zR+\mathrm{h.o.b.}(z)$, $R\in \mathcal{B}_{\geq 1}^{+}\subseteq \mathcal{L}_{k}$. Let $R\neq 0$ and $(1,\mathbf{n}):=\mathrm{ord}(f-\mathrm{id})$, where $\mathbf{n}=(\mathbf{0}_{m-1},n_{m},\ldots ,n_{k})$, $n_{m}\in \mathbb{N}_{\geq 1}$, $1\leq m\leq k-1$. Let $L$ be defined as in the Main Theorem and  suppose that $L_{m+1},T_{m+1}\in \mathcal{B}_{m+1}\setminus \left\lbrace 0\right\rbrace $\footnote{Note that $L_{m+1}\neq 0$ since $n_{m}\neq 0$. Without the loss of generality, we assume that $T_{m+1} \neq 0$. Otherwise step $(b.1)$ is omitted.}, such that:
						\begin{align*}
							& L =\boldsymbol{\ell }_{m}^{n_{m}}L_{m+1} + \mathrm{h.o.t.} , \\
							& R-L =\boldsymbol{\ell }_{1}\cdots \boldsymbol{\ell }_{m-1}\boldsymbol{\ell }_{m}^{n_{m}+1}T_{m+1} + \mathrm{h.o.t.}
						\end{align*}
						Logarithmic transseries $\varphi \in \mathcal{L}_{k}$ is a solution of the equation:
						\begin{align}
							\mathcal{P}_{\leq (\mathbf{1}_{m},n_{m}+1)}(\varphi \circ f\circ \varphi ^{-1}) &= \mathrm{id}+ zL \label{LemaB1Equation}
						\end{align}
						if and only if $S_{m+1}$ satisfies the differential equation:
						{\small \begin{align}
							& L_{m+1}\cdot D_{m+1}(S_{m+1})-(n_{m}L_{m+1}+D_{m+1}(L_{m+1})) \cdot (1+S_{m+1})\cdot \log (1+S_{m+1}) \nonumber \\
							& + T_{m+1}\cdot S_{m+1} = -T_{m+1} . \label{EquationLemaB1}
						\end{align}
						}Here, we write $\varphi =\mathrm{id}+zS+\varepsilon $, $S\in \mathcal{B}_{\geq 1}^{+}\subseteq \mathcal{L}_{k}$, $\varepsilon \in \mathcal{L}_{k}$ such that $\mathrm{ord}_{z} \, (\varepsilon )>1$, and we decompose $S=S_{m+1}+S_{m}+\cdots +S_{1}$, for $S_{i}\in \mathcal{B}_{i}^{+}\subseteq \mathcal{L}_{k}$, $1\leq i\leq m$, and $S_{m+1}\in \mathcal{B}_{\geq m+1}^{+}\subseteq \mathcal{L}_{k}$.
					\end{lem}
					
					\begin{proof}
						The conjugacy equation \eqref{LemaB1Equation} with $\varphi \in \mathcal{L}_{k}^{0}$ is equivalent to the equation:
						\begin{align}
							\mathcal{P}_{\leq (\mathbf{1}_{m},n_{m}+1)}(\varphi \circ f) & =\mathcal{P}_{\leq (\mathbf{1}_{m},n_{m}+1)}\big( (\mathrm{id}+zL)\circ \varphi \big) , \quad \varphi \in \mathcal{L}_{k}^{0} . \label{LemaB1Equation2}
						\end{align}
						Put $\mu :=f-(\mathrm{id}+zR)$. Then $\mathrm{ord}_{z} \, (\mu ) >1$. Put $\varphi =\mathrm{id}+zS+\varepsilon $, for $S\in \mathcal{B}_{\geq 1}^{+}$, and $\varepsilon \in \mathcal{L}_{k}$ such that $\mathrm{ord}_{z} \, (\varepsilon )>1$. By the Taylor Theorem (see \cite[Proposition 3.3]{prrs21}) we get:
						{\small \begin{align}
								\varphi \circ f & = \mathrm{id}+zR+\mu + \sum _{i\geq 1}\frac{(zS+\varepsilon )^{(i)}(\mathrm{id}+zR)}{i!}\mu ^{i} \nonumber \\
								& + zS+\varepsilon+\sum _{i\geq 1}\frac{(zS)^{(i)}}{i!}(zR)^{i} + \sum _{i\geq 1}\frac{\varepsilon ^{(i)}}{i!}(zR)^{i} \nonumber \\
								(\mathrm{id} +zL)\circ \varphi & = \mathrm{id}+zS+\varepsilon+zL+\sum _{i\geq 1}\frac{(zL)^{(i)}(\mathrm{id}+zS)}{i!}\varepsilon^{i}+\sum _{i\geq 1}\frac{(zL)^{(i)}}{i!}(zS)^{i}, \label{LemCaseBEq1}
						\end{align}
						}Analyzing orders and applying operator $\mathcal{P}_{\leq (\mathbf{1}_{m},n_{m}+1)}$ on \eqref{LemCaseBEq1}, we get that the equation \eqref{LemaB1Equation2} is equivalent to the equation:
						\begin{align}
							\mathcal{P}_{\leq (\mathbf{1}_{m},n_{m}+1)}\Big( \sum _{i\geq 1}\frac{(zS)^{(i)}}{i!}(zR)^{i} - \sum _{i\geq 1}\frac{(zL)^{(i)}}{i!}(zS)^{i} + z(R-L) \Big) &= 0 . \label{LemCaseBEq2}
						\end{align}
						By \eqref{Identitet2}, it follows that:
						\begin{align}
							\sum _{i\geq 1}\frac{(zL)^{(i)}}{i!}(zS)^{i} &= z\big( L\cdot S+D_{1}(L)\cdot (1+S)\cdot \log (1+S) + \mathcal{C}_{1}(S) \big) , \nonumber \\
							\sum _{i\geq 1}\frac{(zS)^{(i)}}{i!}(zR)^{i} &= z\big( S\cdot R+D_{1}(S)\cdot (1+R)\cdot \log (1+R) + \mathcal{K}_{1}(S) \big) , \label{LemaB1Identities}
						\end{align}
						for $\frac{1}{2^{2+\mathrm{ord}_{\boldsymbol{\ell }_{1}}(R)}}$-contractions $\mathcal{C}_{1},\mathcal{K}_{1}:(\mathcal{B}_{\geq 1}^{+},d_{1})\to (\mathcal{B}_{\geq 1}^{+},d_{1})$. Note that
						\begin{align}
							& \mathcal{P}_{\leq (\mathbf{1}_{m},n_{m}+1)}(\mathcal{K}_{1}(S))=\mathcal{P}_{\leq (\mathbf{1}_{m},n_{m}+1)}(\mathcal{C}_{1}(S))=0, \label{LemaB1Identities2}
						\end{align}
						for each $S\in \mathcal{B}_{\geq 1}^{+}\subseteq \mathcal{L}_{k}$. Applying $\mathcal{P}_{\leq (\mathbf{1}_{m},n_{m}+1)}$ to the equation \eqref{LemCaseBEq2} and using \eqref{LemaB1Identities} and \eqref{LemaB1Identities2}, after dividing by $z$, we get the following equivalent equation:
						\begin{align}
							& \mathcal{P}_{\leq (0,\mathbf{1}_{m-1},n_{m}+1)}\Big( S\cdot (R-L) + (R-L)+(1+R)\cdot \log (1+R)\cdot D_{1}(S) \Big) \nonumber \\
							& + \mathcal{P}_{\leq (0,\mathbf{1}_{m-1},n_{m}+1)}\Big( -(1+S)\cdot \log (1+S) \cdot D_{1}(L) \Big) = 0. \label{LemaB1IdentityEq}
						\end{align}
						Since we eliminated $z$ from the equation \eqref{LemCaseBEq2}, we use the projection operator $\mathcal{P}_{\leq (0,\mathbf{1}_{m-1},n_{m}+1)}$ instead of $\mathcal{P}_{\leq (\mathbf{1}_{m},n_{m}+1)}$ in the equation \eqref{LemaB1IdentityEq}.
						
						Recall decomposition $S=S_{m+1}+S_{m}+\cdots +S_{1}$. By \eqref{EqDer1} and \eqref{EqDer2} we have:
						\begin{align}
							& D_{1}(S_{m+1}) =\boldsymbol{\ell }_{1}\cdots \boldsymbol{\ell }_{m-1}D_{m}(S_{m+1}) , \nonumber \\
							& D_{m}(S_{m+1}) =\boldsymbol{\ell }_{m}D_{m+1}(S_{m+1}) , \label{EqLemmaOne}
						\end{align}
						$\mathrm{Lb}_{\boldsymbol{\ell }_{m}}(R)=\mathrm{Lb}_{\boldsymbol{\ell }_{m}}(L)=\boldsymbol{\ell }_{m}^{n_{m}}L_{m+1}$ and $\mathrm{Lb}_{\boldsymbol{\ell }_{m}}\big( \frac{R-L}{\boldsymbol{\ell }_{1}\cdots \boldsymbol{\ell }_{m-1}} \big) =\boldsymbol{\ell }_{m}^{n_{m}+1}T_{m+1}$, analyzing the orders of the terms, we get:
						{\small \begin{align}
								& \mathcal{P}_{\leq (0,\mathbf{1}_{m-1},n_{m}+1)}\Big( (1+R)\cdot \log (1+R)\cdot D_{1}(S)-(1+S)\cdot \log (1+S) \cdot D_{1}(L) + S\cdot (R-L) + (R-L)\Big) \nonumber \\
								& = \boldsymbol{\ell }_{1}\cdots \boldsymbol{\ell }_{m-1}\Big( \boldsymbol{\ell }_{m}^{n_{m}}L_{m+1}\cdot D_{m}(S_{m+1})-(1+S_{m+1})\cdot \log (1+S_{m+1})\cdot D_{m}(\boldsymbol{\ell }_{m}^{n_{m}}L_{m+1}) \Big) \nonumber \\
								& + \boldsymbol{\ell }_{1}\cdots \boldsymbol{\ell }_{m-1}\Big( S_{m+1}\cdot \boldsymbol{\ell }_{m}^{n_{m}+1}T_{m+1} + \boldsymbol{\ell }_{m}^{n_{m}+1}T_{m+1} \Big) . \label{LemaB1Equationlast}
						\end{align}
						}Dividing by $\boldsymbol{\ell }_{1}\cdots \boldsymbol{\ell }_{m-1}$ and using \eqref{LemaB1Equationlast}, equation \eqref{LemaB1IdentityEq} is equivalent to the equation:
						\begin{align}
							& \boldsymbol{\ell }_{m}^{n_{m}}L_{m+1}\cdot D_{m}(S_{m+1})-(1+S_{m+1})\cdot \log (1+S_{m+1})\cdot D_{m}(\boldsymbol{\ell }_{m}^{n_{m}}L_{m+1}) \nonumber \\
							& + S_{m+1}\cdot \boldsymbol{\ell }_{m}^{n_{m}+1}T_{m+1} = -\boldsymbol{\ell }_{m}^{n_{m}+1}T_{m+1} . \label{EqLeFixed}
						\end{align}
						Using \eqref{EqDer1} and \eqref{EqLemmaOne}, we get that equation \eqref{EqLeFixed} is equivalent to the equation:
						\begin{align*}
							& L_{m+1}\cdot D_{m+1}(S_{m+1})-(1+S_{m+1})\cdot \log (1+S_{m+1})\cdot (n_{m}L_{m+1}+D_{m+1}(L_{m+1})) \\
							& + S_{m+1}\cdot T_{m+1} = -T_{m+1} .
						\end{align*}
					\end{proof}
				
					\begin{prop}[A solution of a differential equation in $\mathcal{B}_{\geq m+1}^{+}$]\label{Prop3}
						Let $N,T\in \mathcal{B}_{m+1}\subseteq \mathcal{L}_{k}^{\infty }$, $k\in \mathbb{N}_{\geq 1}$, such that $N\neq 0$ and $\mathrm{ord}(T)> \mathrm{ord}(N)$. Let $n \in \mathbb{N}_{\geq 1}$ and let $h\in x^{2}\mathbb{R}\left[ \left[ x\right] \right] $ be a power series in the variable $x$, with real coefficients, such that $h(0)=h'(0)=0$. Then there exists a unique solution $S\in \mathcal{B}_{\geq m+1}^{+}\subseteq \mathcal{L}_{k}$ of the equation:
						\begin{align}
							N\cdot D_{m+1}(S)-(nN+D_{m+1}(N))\cdot (S+h(S)) + T\cdot S &= T . \label{PropEq1}
						\end{align}	
					\end{prop}
					\begin{proof}
						Dividing by $N\neq 0$ both sides of equation \eqref{PropEq1} and by regrouping, we get the equivalent equation:
						\begin{align*}
							\Big( \frac{T}{N}-n \Big) S-n\cdot h(S) &= (S+h(S))\cdot \frac{D_{m+1}(N)}{N}-D_{m+1}(S) + \frac{T}{N} .
						\end{align*}
						Let $\mathcal{T}_{m+1},\mathcal{S}_{m+1}:\mathcal{B}_{\geq m+1}^{+}\to \mathcal{B}_{\geq m+1}^{+}$ be operators defined by:
						\begin{align}
							\mathcal{S}_{m+1}(S) & := (S+h(S))\cdot \frac{D_{m+1}(N)}{N}-D_{m+1}(S) + \frac{T}{N} , \nonumber \\
							\mathcal{T}_{m+1}(S) & := \Big( \frac{T}{N}-n \Big) S-n\cdot h(S) , \quad S\in \mathcal{B}_{\geq m+1}^{+}\subseteq \mathcal{L}_{k} . \nonumber 
						\end{align}
						Let $S_{1},S_{2}\in \mathcal{B}_{\geq m+1}^{+}\subseteq \mathcal{L}_{k}$ be arbitrary such that $S_{1}\neq S_{2}$. Since $\mathrm{ord}(T)>\mathrm{ord}(N)$, it follows that $\mathrm{ord}\left( \frac{T}{N} \right) > \mathbf{0}_{k+1}$. Note that
						\begin{align*}
							\mathcal{T}_{m+1}(S_{1}) - \mathcal{T}_{m+1}(S_{2}) & = \Big( \frac{T}{N}-n \Big) (S_{1}-S_{2})-n\cdot (S_{1}-S_{2})\cdot \sum _{i\geq 2}\Big( \sum _{j=0}^{i-1}S_{1}^{i-j}\cdot S_{2}^{j} \Big) \\
							& = -n\mathrm{Lt}(S_{1}-S_{2})+\mathrm{h.o.t.}
						\end{align*}
						Therefore, $\mathrm{ord}_{\boldsymbol{\ell }_{m+1}}(\mathcal{T}_{m+1}(S_{1}-S_{2}))=\mathrm{ord}_{\boldsymbol{\ell }_{m+1}}(S_{1}-S_{2})$. This implies that $\mathcal{T}_{m+1}$ is an isometry on the space $(\mathcal{B}_{\geq m+1}^{+},d_{m+1})$. Similarly as in Lemma~\ref{Lem4}, it can be proven that $\mathcal{T}_{m+1}$ is a surjection. Since
						\begin{align*}
							& \mathrm{ord}_{\boldsymbol{\ell }_{m+1}}\Big( \frac{D_{m+1}(N)}{N}\Big) \geq 1 , \\
							& \mathrm{ord}_{\boldsymbol{\ell }_{m+1}}(D_{m+1}(S)) \geq \mathrm{ord}_{\boldsymbol{\ell }_{m+1}}(S)+1 ,
						\end{align*}
						it follows that $\mathcal{S}_{m+1}$ is a $\frac{1}{2}$-contraction on the space $(\mathcal{B}_{\geq m+1}^{+},d_{m+1})$. By the fixed point theorem from Proposition~\ref{KorBanach}, it follows that there exists a unique solution $S\in \mathcal{B}_{\geq m+1}^{+}\subseteq \mathcal{L}_{k}$ of the equation $\mathcal{T}_{m+1}(S)=\mathcal{S}_{m+1}(S)$, i.e., equation \eqref{PropEq1}.	
					\end{proof}
				
					\begin{proof}[Proof of step (b.1)]
						Suppose that $f\in \mathcal{L}_{k}^{0}$, $k\in \mathbb{N}_{\geq 1}$, such that $f=\mathrm{id}+zR+\mathrm{h.o.b.}(z)$, for $R\in \mathcal{B}_{\geq 1}^{+}\subseteq \mathcal{L}_{k}$. We assume here that $R\neq 0$, since otherwise, it follows that $\mathrm{ord}_{z}(f-\mathrm{id}) >1$ and we can apply case $(a)$ instead. Let $\mathrm{ord} \, (f-\mathrm{id}) = (1,\mathbf{0}_{m-1},n_{m},\ldots ,n_{k})$, $n_{m}\geq 1$, $1\leq m\leq k-1$. Let $L$ be as defined in the Main Theorem. In this step we eliminate all terms in $f-\mathrm{id}$ except the terms in $zL$, up to the order $(\mathbf{1}_{m},n_{m}+1)$ in the first $m+1$ variables $z,\boldsymbol{\ell }_{1},\ldots ,\boldsymbol{\ell }_{m}$. Note that, if $m=k$, all such terms are in $zL$. In this case we have nothing to eliminate and we skip case (b.1).
						
						Suppose that $m<k$. By Lemma~\ref{CaseBLemma7} we transform equation \eqref{FirstConj} to differential equation \eqref{EquationLemaB1}. Let us decompose $L$ and $R-L$ as in Lemma~\ref{CaseBLemma7}. If $T_{m+1}=0$, then we have nothing to eliminate and we skip case (b.1). Suppose that $T_{m+1}\neq 0$. Now, put $T:=-T_{m+1}$, $N:=L_{m+1}$ and $h:=(1+x)\log (1+x)-x$. Since $N\neq 0$ and $\mathrm{ord}(T)>\mathrm{ord}(N)$, by Proposition~\ref{Prop3}, there exists a unique solution $S_{m+1}\in \mathcal{B}_{\geq m+1}^{+}\subseteq \mathcal{L}_{k}$ of equation \eqref{EquationLemaB1}. Now, we put $\varphi _{1}:=\mathrm{id}+zS_{m+1}$ and, by Lemma~\ref{CaseBLemma7}, $\varphi _{1}\in \mathcal{L}_{k}^{0}$ is a solution of conjugacy equation \eqref{FirstConj}.  
					\end{proof}
				
					\begin{remark}[Non-uniqueness of the conjugacy $\varphi _{1}$]\label{RemCaseB1}
						Let $\varphi _{1}:=\mathrm{id}+zS+\varepsilon $, $\varepsilon \in \mathcal{L}_{k}$, $\mathrm{ord}_{z} \, (\varepsilon ) >1$, and  $S\in \mathcal{B}_{\geq 1}^{+}\subseteq \mathcal{L}_{k}$, be a solution of conjugacy equation \eqref{FirstConj} for a logarithmic transseries $f\in \mathcal{L}_{k}^{0}$, $k\in \mathbb{N}_{\geq 1}$. Now, put $\mathrm{ord} \, (f - \mathrm{id}):=(1,\mathbf{0}_{m-1},n_{m},\ldots ,n_{k})$, for $n_{m}\geq 1$ and $1\leq m\leq k-1$. We decompose $S$ as $S=S_{1}+\cdots +S_{m}+S_{m+1}$, for $S_{i}\in \mathcal{B}_{i}^{+}\subseteq \mathcal{L}_{k}$, $1\leq i\leq m$, and $S_{m+1}\in \mathcal{B}_{\geq m+1}^{+}\subseteq \mathcal{L}_{k}$. By Lemma~\ref{CaseBLemma7}, $\varphi _{1}$ is a solution of conjugacy equation \eqref{FirstConj} if and only if $S_{m+1}$ is a solution of equation \eqref{EquationLemaB1}. Therefore, we can choose an arbitrary $S_{i}\in \mathcal{B}_{i}^{+}$, $1\leq i\leq m$, and $\varepsilon \in \mathcal{L}_{k}$, $\mathrm{ord}_{z}(\varepsilon ) >1$, such that $\varphi _{1}$ is still a solution of conjugacy equation \eqref{FirstConj}. Although $\varphi _{1}$ is obviously not a unique solution of conjugacy equation \eqref{FirstConj}, if we request the \emph{canonical form} of $\varphi _{1}$, i.e., $\varphi _{1}:=\mathrm{id}+zS_{m+1}$, where $S_{m+1}\in \mathcal{B}_{\geq m+1}^{+}\subseteq \mathcal{L}_{k}$, then, by Lemma~\ref{CaseBLemma7} and Proposition~\ref{Prop3}, it follows that $S_{m+1}\in \mathcal{B}_{\geq m+1}^{+}\subseteq \mathcal{L}_{k}$ is the unique solution of equation \eqref{EquationLemaB1}. Consequently, $\varphi _{1}=\mathrm{id}+zS_{m+1}$ is a unique \emph{canonical} solution of conjugacy equation \eqref{FirstConj}.
					\end{remark}

					\subsubsection{Proof of step (b.2)}
					
					For simplicity of notation we suppose that step $(b.1)$ has already been applied on $f\in \mathcal{L}_{k}^{0}$, for $k\in \mathbb{N}_{\geq 1}$, $\mathrm{ord} \, (f-\mathrm{id})=(1,\mathbf{0}_{m-1},n_{m},\ldots ,n_{k})$, $n_{m}\geq 1$, i.e., that $\mathcal{P}_{\leq (\mathbf{1}_{m},n_{m}+1)}(f)=\mathrm{id}+zL$, where $L$ is as defined in the Main Theorem.
					
					We first transform the conjugacy equation
					\begin{align}
						\mathcal{P}_{<r}(\varphi \circ f\circ \varphi ^{-1}) & = \mathrm{id} + zL, \quad \varphi \in \mathcal{L}_{k}^{0} , \label{SecondConj2}
					\end{align}
					to the differential equation in Lemma~\ref{CaseBLemma6}. Then we apply Proposition~\ref{Prop4} and Proposition~\ref{KorBanach} to solve the mentioned differential equation.
					
					\begin{lem}[Transforming equation \eqref{SecondConj2} into a differential equation]\label{CaseBLemma6}
						Let $f\in \mathcal{L}_{k}^{0}$, $k\in \mathbb{N}_{\geq 1}$, such that $f=\mathrm{id}+zR+\mathrm{h.o.b.}(z)$, for $R\in \mathcal{B}_{\geq 1}^{+}\subseteq \mathcal{L}_{k}$. Put $(1,\mathbf{n}):=\mathrm{ord}(f-\mathrm{id})$ and $\mathbf{n}:=(\mathbf{0}_{m-1},n_{m},\ldots ,n_{k})$, $n_{m}\geq 1$, for $1\leq m\leq k$. Let $L$ be as defined in the Main Theorem. Suppose that $z(R-L)$ contains at least one term of order strictly smaller than $\mathrm{ord} \, (\mathrm{Res} \, (f))$\footnote{Otherwise we skip case $(b.2)$.}. Put\footnote{Note that $\mathrm{ord}(R-L)>(\mathbf{1}_{m},n_{m}+1,n_{m+1},\ldots ,n_{k})$.}
						\begin{align}
							& L=L_{m}+\cdots +L_{1} , \nonumber \\
							& R-L=\boldsymbol{\ell }_{1}\cdots \boldsymbol{\ell }_{m}T_{m}+\boldsymbol{\ell }_{1}\cdots \boldsymbol{\ell }_{m-1}T_{m-1}+\cdots +\boldsymbol{\ell }_{1}T_{1} , \label{EquationLemaDecomposition}
						\end{align}
						where $L_{i},T_{i}\in \mathcal{B}_{i}^{+}\subseteq \mathcal{L}_{k}$, for $i=1,\ldots ,m$. Put $r:=\mathrm{ord}(\mathrm{Res}(f))=(1,2\mathbf{n}+\mathbf{1}_{k})$ and $r_{0}:=(\mathbf{0}_{m},2n_{m}+1,\ldots ,2n_{k}+1)$.
						
						Here, we put $\varphi :=\mathrm{id}+zS+\varepsilon $, for $\varepsilon \in \mathcal{L}_{k}$ such that $\mathrm{ord}_{z} \, (\varepsilon ) >1$, $S\in \mathcal{B}_{\geq 1}^{+}\subseteq \mathcal{L}_{k}$ and decompose $S$ as $S=S_{m}+\cdots +S_{1}$, for $S_{i}\in \mathcal{B}_{i}^{+}\subseteq \mathcal{L}_{k}$, $1\leq i\leq m-1$, and $S_{m}\in \mathcal{B}_{\geq m}^{+}\subseteq \mathcal{L}_{k}$. The logarithmic transseries $\varphi \in \mathcal{L}_{k}^{0}$ is a solution of the equation:
						\begin{align}
							\mathcal{P}_{<r}(\varphi \circ f\circ \varphi ^{-1}) = \mathrm{id}+ zL \nonumber
						\end{align}
						if and only if $S_{m}$ belongs to $\mathcal{B}_{m}^{+}\setminus \left\lbrace 0\right\rbrace \subseteq \mathcal{L}_{k}$ and $S_{m}$ satisfies the equation:
						{\small \begin{align}
							& \mathcal{P}_{<r_{0}}\Big( (1+L_{m})\log (1+L_{m})D_{m}(S_{m}) - D_{m}(L_{m})\cdot (1+S_{m})\log (1+S_{m})+\boldsymbol{\ell }_{m}T_{m}S_{m}-\boldsymbol{\ell }_{m}T_{m} \Big) \nonumber \\
							& = \mathcal{P}_{<r_{0}}(\mathcal{C}(S_{m})) , \label{B2LemaEquation}
						\end{align}
						}where $\mathcal{C}:\mathcal{B}_{\geq 1}^{+}\to \mathcal{B}_{\geq 1}^{+}$ is a $\frac{1}{2^{2+\mathrm{ord}_{\boldsymbol{\ell }_{1}} \, (R)}}$-contraction on the space $(\mathcal{B}_{\geq 1}^{+},d_{1})$.
					\end{lem}
				
					\begin{proof}
						We first apply the similar procedure as in the proof of Lemma~\ref{CaseBLemma7}. By comparing orders and applying projection operator $\mathcal{P}_{<r}$, it is easy to see that the conjugacy equation $\mathcal{P}_{<r}( \varphi \circ f\circ \varphi ^{-1})=\mathrm{id}+zL$ is equivalent to the equation:
						\begin{align*}
							\mathcal{P}_{<r}(\varphi \circ f ) & =\mathcal{P}_{<r}\big( ( \mathrm{id}+zL)\circ \varphi \big) ,
						\end{align*}
						where $\varphi :=\mathrm{id}+zS+\varepsilon $, for $S\in \mathcal{B}_{\geq 1}^{+}\subseteq \mathcal{L}_{k}$ and $\varepsilon \in \mathcal{L}_{k}$, $\mathrm{ord}_{z} \, (\varepsilon ) >1$. Put $\mu :=f-(\mathrm{id}+zR)$. By the Taylor Theorem (see \cite[Proposition 3.3]{prrs21}) we get the equivalent equation:
						{\small \begin{align}
								& \mathcal{P}_{<r}\Big( zS+\varepsilon+\sum _{i\geq 1}\frac{(zS+\varepsilon )^{(i)}(\mathrm{id}+zR)}{i!}\mu ^{i}+\mathrm{id}+zR+\mu  + \sum _{i\geq 1}\frac{(zS)^{(i)}}{i!}(zR)^{i} + \sum _{i\geq 1}\frac{\varepsilon ^{(i)}}{i!}(zR)^{i} \Big) \nonumber \\
								&= \mathcal{P}_{<r}\Big( \mathrm{id}+zS+\varepsilon+zL+\sum _{i\geq 1}\frac{(zL)^{(i)}(\mathrm{id}+zS)}{i!}\varepsilon^{i}+\sum _{i\geq 1}\frac{(zL)^{(i)}}{i!}(zS)^{i} \Big) . \label{B2LemCaseBEq1}
						\end{align}
						}Subtracting and applying operator $\mathcal{P}_{<r}$ on \eqref{B2LemCaseBEq1}, analyzing the orders of terms, we get that the equation $\mathcal{P}_{<r}(\varphi \circ f\circ \varphi ^{-1})=\mathrm{id}+zL$ is equivalent to the equation:
						\begin{align}
							\mathcal{P}_{<r}\Big( \sum _{i\geq 1}\frac{(zS)^{(i)}}{i!}(zR)^{i} - \sum _{i\geq 1}\frac{(zL)^{(i)}}{i!}(zS)^{i} + z(R-L) \Big) &= 0 . \label{B2LemCaseBEq2}
						\end{align}
						By \eqref{Identitet2}, it follows that:
						\begin{align}
							\sum _{i\geq 1}\frac{(zL)^{(i)}}{i!}(zS)^{i} &= z\left( LS+D_{1}(L)(1+S)\log (1+S) + \mathcal{C}_{L}(S) \right) , \nonumber \\
							\sum _{i\geq 1}\frac{(zS)^{(i)}}{i!}(zR)^{i} &= z\left( SR+D_{1}(S)(1+R)\log (1+R) + \mathcal{K}_{R}(S) \right) , \label{EquationLemaProjectorIdenti}
						\end{align}
						for a $\frac{1}{2^{2+\mathrm{ord}_{\boldsymbol{\ell }_{1}}\, (R)}}$-contraction $\mathcal{C}_{L}:=\mathcal{C}(\cdot ,L):(\mathcal{B}_{\geq 1}^{+},d_{1})\to (\mathcal{B}_{\geq 1}^{+},d_{1})$ and a $\frac{1}{2^{2+2\cdot \mathrm{ord}_{\boldsymbol{\ell }_{1}} \, (R)}}$-contraction $\mathcal{K}_{R}:=\mathcal{C}(R,\cdot ):(\mathcal{B}_{\geq 1}^{+},d_{1})\to (\mathcal{B}_{\geq 1}^{+},d_{1})$. \\
						Since
						\begin{align*}
							\mathrm{ord}_{\boldsymbol{\ell }_{1}} (\mathcal{K}_{R}(S)) & \geq 2+2\cdot \mathrm{ord}_{\boldsymbol{\ell }_{1}} (R) ,
						\end{align*}
						it follows that $\mathcal{P}_{<r}(\mathcal{K}_{R}(S)) =0$. Consequently, by \eqref{B2LemCaseBEq2} and \eqref{EquationLemaProjectorIdenti}, we get that the equation $\mathcal{P}_{<r}(\varphi \circ f\circ \varphi ^{-1})=\mathrm{id}+zL$ is equivalent to the equation:
						{\small \begin{align}
							& \mathcal{P}_{<r}\Big( z\Big( S\cdot (R-L) + (R-L)+(1+R)\log (1+R)D_{1}(S)-(1+S)\log (1+S) D_{1}(L) \Big) \Big) \nonumber \\
							& = \mathcal{P}_{<r}(z\mathcal{C}_{L}(S)). \label{B2LemCaseBEq3}
						\end{align}
						}Now decompose $S=S_{m}+\cdots +S_{1}$, for $S_{i}\in \mathcal{B}_{i}^{+}\subseteq \mathcal{L}_{k}$, $1\leq i\leq m-1$, $S_{m}\in \mathcal{B}_{\geq m}^{+}\subseteq \mathcal{L}_{k}$, and $R-L$ as in \eqref{EquationLemaDecomposition}. It follows from \eqref{B2LemCaseBEq3} that $S_{m}\in \mathcal{B}_{m}^{+}\setminus \left\lbrace 0\right\rbrace \subseteq \mathcal{L}_{k}$.
					
						On the contrary, suppose that $S_{m} =0$, or $\mathrm{ord}_{\boldsymbol{\ell }_{m}}(S_{m}) =0$, i.e., $S_{m}\in \mathcal{B}_{\geq m+1}^{+}$.
					
						Suppose that $m=1$. Then $S=S_{1}$. Note that $\mathrm{ord}_{\boldsymbol{\ell }_{1}}(S\cdot (R-L)) \geq 2+\mathrm{ord}_{\boldsymbol{\ell }_{1}}(R)$ and $\mathrm{ord}_{\boldsymbol{\ell }_{1}}(\mathcal{C}_{L}(S)) \geq 2+\mathrm{ord}_{\boldsymbol{\ell }_{1}}(R)$. Since $S\in \mathcal{B}_{\geq 2}^{+}$ and $R\in \mathcal{B}_{1}^{+}$, it follows that:
						\begin{align*}
							\mathrm{ord}_{\boldsymbol{\ell }_{1}} \big( (1+R)\cdot \log (1+R)\cdot D_{1}(S)-(1+S)\cdot \log (1+S) \cdot D_{1}(L) \big) & = \mathrm{ord}_{\boldsymbol{\ell }_{1}} (R)+1 .
						\end{align*}
						Now we get that the order of the left-hand side of equation \eqref{B2LemCaseBEq3} is strictly smaller than the order of the right-hand side of the same equation, which is a contradiction.
						
						Suppose that $m>1$. It can be seen that $\mathrm{ord}(z\mathcal{C}_{L}(S))\geq r$, and therefore, the right-hand side of equation \eqref{B2LemCaseBEq3} is equal to zero. Note that $\mathrm{ord}\, (S\cdot (R-L)) > \mathrm{ord}\, (R) + \mathrm{ord} \, (S)+(0,\mathbf{1}_{m},\mathbf{0}_{k-m})$. Since $S_{m}\in \mathcal{B}_{\geq m+1}^{+}\subseteq \mathcal{L}_{k}$ and $L_{m}\in \mathcal{B}_{m}^{+}\subseteq \mathcal{L}_{k}$, it follows that:
						\begin{align*}
							& \mathrm{ord} \, \big( (1+R)\log (1+R)D_{1}(S)-(1+S)\log (1+S) D_{1}(L) \big) \\
							& = \mathrm{ord}\, (R)+\mathrm{ord} \, (S)+(0,\mathbf{1}_{m},\mathbf{0}_{k-m}) .
						\end{align*}
						Now, the order of the left-hand side of equation \eqref{B2LemCaseBEq3} is strictly smaller than the order of the right-hand side of the same equation, which is a contradiction. Thus, we proved that $S_{m}\in \mathcal{B}_{m}^{+}\setminus \left\lbrace 0\right\rbrace \subseteq \mathcal{L}_{k}$. \\
						
						We finish the proof considering separately cases $m=1$ and $m>1$. Suppose that $m=1$. Then we have $S=S_{1}$, $R-L=\boldsymbol{\ell }_{1}T_{1}$ and $L=L_{1}$. Therefore, in case $m=1$ equation \eqref{B2LemaEquation} follows by dividing equation \eqref{B2LemCaseBEq3} by $z$.
						
						Now, suppose that $m>1$. Note that $D_{1}(S_{i})=\boldsymbol{\ell }_{1}\cdots \boldsymbol{\ell }_{i-1}D_{i}(S_{i})$, for $2\leq i\leq m$, $(1+R)\log (1+R)=R+\sum _{i\geq 2}\frac{(-1)^{i}}{i(i-1)}R^{i}$ and $R=L_{m}+L_{m-1}+\cdots +L_{1}+\mathrm{h.o.t.}$ Note that $S_{m}\neq 0$ by above discussion, and $L_{m}\neq 0$. Analyzing orders of terms, it follows (for $m>1$) that:
						\begin{align}
							& \mathcal{P}_{<r}(z(1+R)\log (1+R)D_{1}(S)) \nonumber \\
							&=\mathcal{P}_{<r}(z(1+L_{m})\log (1+L_{m})D_{1}(S_{m})) \nonumber \\
							&=\mathcal{P}_{<r}(z\boldsymbol{\ell }_{1}\cdots \boldsymbol{\ell }_{m-1}(1+L_{m})\log (1+L_{m})D_{m}(S_{m})) . \label{LemaCaseB2Equat31}
						\end{align}
						By \eqref{EquationLemaDecomposition}, it follows that:
						\begin{align}
							\mathcal{P}_{<r}(z(R-L)) &=\mathcal{P}_{<r}(z\boldsymbol{\ell }_{1}\cdots \boldsymbol{\ell }_{m}T_{m}) . \label{LemaCaseB2Equat32}
						\end{align}
						Since $\mathrm{ord}_{\boldsymbol{\ell }_{1}}(\mathcal{C}_{L}(S)) \geq 2+\mathrm{ord}_{\boldsymbol{\ell }_{1}}(R) >1$, it follows that $\mathcal{P}_{<r}(z\mathcal{C}_{L}(S)) =0$, $S\in \mathcal{B}_{\geq 1}^{+}\subseteq \mathcal{L}_{k}$. Therefore:
						\begin{align}
							& \mathcal{P}_{<r}(z(1+S)\log (1+S)D_{1}(R) + \mathcal{C}_{L}(S)) \nonumber \\
							&=\mathcal{P}_{<r} \big( z(1+S_{m})\log (1+S_{m})D_{1}(L_{m}) \big) \nonumber \\
							&=\mathcal{P}_{<r}(z\boldsymbol{\ell }_{1}\cdots \boldsymbol{\ell }_{m-1}(1+S_{m})\log (1+S_{m})D_{m}(L_{m})) . \label{LemaCaseB2Equat33}
						\end{align}
						After dividing by $z\boldsymbol{\ell }_{1}\cdots \boldsymbol{\ell }_{m-1}$ and using \eqref{LemaCaseB2Equat31}, \eqref{LemaCaseB2Equat32} and \eqref{LemaCaseB2Equat33}, we get that equation \eqref{B2LemCaseBEq3} is equivalent to the equation \eqref{B2LemaEquation}. Therefore, the equation \eqref{B2LemaEquation} is equivalent to the equation $\mathcal{P}_{<r}(\varphi \circ f\circ \varphi ^{-1})=\mathrm{id}+zL$.
					\end{proof}
					
					\begin{prop}[Solution of a differential equation]\label{Prop4}
						Let $L\in \mathcal{B}_{m}^{+}\setminus \left\lbrace 0\right\rbrace \subseteq \mathcal{L}_{k}$, $\mathrm{ord}(L)=(\mathbf{0}_{m}, n_{m},\ldots ,n_{k})$, $n_{m}\geq 1$, and $r_{0}:=(\mathbf{0}_{m},2n_{m}+1,\ldots ,2n_{k}+1)$, $1\leq m\leq k$, $k\in \mathbb{N}_{\geq 1}$. Let $V\in \mathcal{B}_{m}^{+}\subseteq \mathcal{L}_{k}$, $\mathrm{ord}_{\boldsymbol{\ell }_{m}}(V)\geq  n_{m}+2$, $\mathrm{ord}(V)<r_{0}$. Furthermore, let $h\in x^{2}\mathbb{R}\left[ \left[ x\right] \right] $ be a power series in the variable $x$, with real coefficients, such that $h(0)=h'(0)=0$, and let $\mathcal{C}_{m}:\mathcal{B}_{m}^{+}\to \mathcal{B}_{m}^{+}$ be a $\frac{1}{2^{2+n_{m}}}$-contraction, with respect to the metric $d_{m}$. Then there exists a solution $S\in \mathcal{B}_{m}^{+}\subseteq \mathcal{L}_{k}$, $\mathrm{ord}(S)<\mathrm{ord}(L)$, of the equation:
						\begin{align}
							\mathcal{P}_{<r_{0}}\Big( (L+h(L))D_{m}(S) - D_{m}(L)\cdot (S+h(S))+VS\Big) & = \mathcal{P}_{<r_{0}}(V+\mathcal{C}_{m}(S)) . \label{PropEq4}
						\end{align}
						Moreover, the solution $S$ is unique if we impose condition that every term in $S$ has strictly smaller order than $r_{0}$.
					\end{prop}
					
					\begin{proof}
						Note that:
						\begin{align*}
							L & = \boldsymbol{\ell }_{m}^{n_{m}}P_{m+1} + P_{m} , \quad n_{m}\geq 1, 
						\end{align*}
						where $P_{m+1}\in \mathcal{B}_{m+1}\subseteq \mathcal{L}_{k}^{\infty }$ and $P_{m}\in \mathcal{B}_{m}^{+}\subseteq \mathcal{L}_{k}$, $\mathrm{ord}_{\boldsymbol{\ell }_{m}}(P_{m}) \geq n_{m}+1$.
						Put:
						\begin{align}
							\mathcal{S}_{m}(S) & := \mathcal{P}_{<r_{0}}\Big( V+\mathcal{C}_{m}(S)-h(L)\cdot D_{m}(S)+h(S)\cdot D_{m}(L)-SV\Big) , \nonumber \\
							\mathcal{T}_{m}(S) & := \mathcal{P}_{<r_{0}}\Big( LD_{m}(S) - SD_{m}(L) \Big) . \label{PropSTPart2}
						\end{align}
						Note that \eqref{PropEq4}, for $S\in \mathcal{B}_{m}^{+}\subseteq \mathcal{L}_{k}$, is equivalent to:
						\begin{align*}
							\mathcal{T}_{m}(S) &= \mathcal{S}_{m}(S) .
						\end{align*}
						Since
						\begin{align}
							& \mathrm{ord}_{\boldsymbol{\ell }_{m}}(\mathcal{C}_{m}(S)) \geq \mathrm{ord}_{\boldsymbol{\ell }_{m}}(S)+n_{m}+2, \nonumber \\
							& \mathrm{ord}_{\boldsymbol{\ell }_{m}}(SV) \geq \mathrm{ord}_{\boldsymbol{\ell }_{m}}(S)+n_{m}+2, \nonumber \\
							& \mathrm{ord}_{\boldsymbol{\ell }_{m}}(h(L)\cdot D_{m}(S)) \geq \mathrm{ord}_{\boldsymbol{\ell }_{m}}(S)+2n_{m}+1, \nonumber \\
							& \mathrm{ord}_{\boldsymbol{\ell }_{m}}(h(S)D_{m}(L)) \geq 2\cdot \mathrm{ord}_{\boldsymbol{\ell }_{m}}(S)+n_{m}+1, \label{PropCaseB2FixedEqTS2}
						\end{align}
						and $n_{m}\geq 1,\mathrm{ord}_{\boldsymbol{\ell }_{m}}(S)\geq 1$, it follows that $\mathcal{S}_{m}$ is a $\frac{1}{2^{n_{m}+2}}$-contraction on the space $(\mathcal{B}_{m}^{+},d_{m})$.
						
						Let $\widetilde{\mathcal{B}}\subseteq \mathcal{B}_{m}^{+}\subseteq \mathcal{L}_{k}$ be the space that contains $0$ and all logarithmic transseries in $\mathcal{B}_{m}^{+}$ that contain only terms of order strictly smaller than $(\mathbf{0}_{m},n_{m},\ldots ,n_{k})=\mathrm{ord}(L)$. Note that the restriction $\mathcal{T}_{m}|_{\widetilde{\mathcal{B}}}:\widetilde{\mathcal{B}}\to \mathcal{B}_{m}^{+}$ is a linear operator. For $S\in \widetilde{\mathcal{B}}\setminus \left\lbrace 0\right\rbrace $ we get:
						\begin{align}
							\mathrm{ord}_{\boldsymbol{\ell }_{m}}(\mathcal{T}_{m}(S)) & =\mathrm{ord}_{\boldsymbol{\ell }_{m}}(S) + n_{m}+1 , \label{PropCaseB2FixedEqTS3}
						\end{align}
						if $\mathcal{T}_{m}(S) \neq 0$. Suppose that $\mathcal{T}_{m}(S) = 0$. Then there exists $N\in \mathcal{B}_{m}^{+}\subseteq \mathcal{L}_{k}$, $\mathrm{ord} \, (N) \geq r_{0}$, such that $LD_{m}(S)-SD_{m}(L)=N$. Dividing by $L$ and solving the linear differential equation, we get
						\begin{align*}
							S &=L\cdot \Big( C+\int \frac{N}{L^{2}}\cdot \frac{d\boldsymbol{\ell }_{m}}{\boldsymbol{\ell }_{m}^{2}}\Big) ,
						\end{align*}
						for $C\in \mathbb{R}$. Since $S\in \widetilde{\mathcal{B}}$, we get that $N=0$ and $C=0$, which implies that $S=0$. Consequently, if $S\in \widetilde{\mathcal{B}}\setminus \left\lbrace 0\right\rbrace $, then $\mathcal{T}_{m}(S) \neq 0$, and therefore, \eqref{PropCaseB2FixedEqTS3} holds. Since $\mathcal{T}_{m}|_{\widetilde{\mathcal{B}}}$ is linear, it follows that $\mathcal{T}_{m}|_{\widetilde{\mathcal{B}}}$ is a $\frac{1}{2^{n_{m}+1}}$-homothety, with respect to the metric $d_{m}$.
						
						Let us prove that $\mathcal{S}_{m}(\widetilde{\mathcal{B}})\subseteq \mathcal{T}_{m}(\widetilde{\mathcal{B}})$. Let $K\in \mathcal{S}_{m}(\widetilde{\mathcal{B}})\subseteq \mathcal{B}_{m}^{+} \subseteq \mathcal{L}_{k}$ be arbitrary. By \eqref{PropSTPart2} and \eqref{PropCaseB2FixedEqTS2}, we get $\mathrm{ord}_{\boldsymbol{\ell }_{m}}(K)\geq n_{m}+2$ and $\mathrm{ord}(K)<r_{0}$. We find $S\in \widetilde{\mathcal{B}}$, such that $\mathcal{T}_{m}(S) = K$, i.e.
						\begin{align}
							\mathcal{P}_{<r_{0}}\left( LD_{m}(S) - SD_{m}(L) \right) & = K . \label{PropCaseB2FixedEqTS4}
						\end{align}
						We have the following decomposition:
						\begin{align*}
							K & =K_{m}+\boldsymbol{\ell }_{m}^{2n_{m}+1}K_{m+1}+\boldsymbol{\ell }_{m}^{2n_{m}+1}\boldsymbol{\ell }_{m+1}^{2n_{m+1}+1}K_{m+2}+\cdots + \boldsymbol{\ell }_{m}^{2n_{m}+1}\cdots \boldsymbol{\ell }_{k-1}^{2n_{k-1}+1}K_{k} ,
						\end{align*}
						where $K_{m}\in \mathcal{B}_{m}^{+}\subseteq \mathcal{L}_{k}$, $n_{m}+2\leq \mathrm{ord}_{\boldsymbol{\ell }_{m}}(K_{m}) < 2n_{m}+1$, and $K_{m+i}\in \mathcal{B}_{m+i}\subseteq \mathcal{L}_{k}^{\infty }$, $\mathrm{ord}_{\boldsymbol{\ell }_{m+i}}(K_{m+i}) < 2n_{m+i}+1$, $1\leq i\leq k-m$. \\
						
						Now we proceed to solving \eqref{PropCaseB2FixedEqTS4} inductively in $k-m+1$ steps: \\
						\noindent \emph{Step} $1$. We first solve the equation $\mathcal{P}_{<(\mathbf{0}_{m},2n_{m}+1)}(\mathcal{T}_{m}(S)) = K_{m}$, i.e.
						\begin{align*}
							\mathcal{P}_{<(\mathbf{0}_{m},2n_{m}+1)}(LD_{m}(S)-SD_{m}(L)) & = K_{m}.
						\end{align*}
						By Lemma~\ref{LemaInductiveStep}, there exists a solution $S_{m}\in \mathcal{B}_{m}\subseteq \mathcal{L}_{k}^{\infty }$ of the previous equation, such that order (in $\boldsymbol{\ell }_{m}$) of every term in $S_{m}$ is strictly smaller than $n_{m}$. Since $\mathrm{ord}_{\boldsymbol{\ell }_{m}}(K_{m}) \geq n_{m}+2$, it follows, by \eqref{LemaEquationLinearDiff}, that $S_{m}\in \mathcal{B}_{m}^{+}\subseteq \mathcal{L}_{k}$. Now, put:
						\begin{align}
							\boldsymbol{\ell }_{m}^{2n_{m}+1}G_{m+1} & := \mathcal{P}_{<r_{0}}(\mathcal{T}_{m}(S_{m}) - K_{m}) , \label{PropCaseB2FixedEqTS7}
						\end{align}
						where $G_{m+1}\in \mathcal{B}_{m+1}\subseteq \mathcal{L}_{k}^{\infty }$ is such that $\mathrm{ord}(\boldsymbol{\ell }_{m}^{2n_{m}+1}G_{m+1})<r_{0}$. \\
						
						\noindent \emph{Step} $2$. Now, with $S_{m}\in \mathcal{B}_{m}\subseteq \mathcal{L}_{k}^{\infty }$ from the previous step, we solve in the variable $S\in \mathcal{B}_{m+1}\subseteq \mathcal{L}_{k}^{\infty }$ the equation:
						\begin{align}
							& \mathcal{P}_{<(0,\mathbf{0}_{m-1},2n_{m}+1,2n_{m+1}+1)} \Big( LD_{m}(S_{m}+\boldsymbol{\ell }_{m}^{n_{m}}S) - (S_{m}+\boldsymbol{\ell }_{m}^{n_{m}}S)D_{m}(L) \Big) \nonumber \\
							& = K_{m} + \boldsymbol{\ell }_{m}^{2n_{m}+1}K_{m+1} . \label{PropCaseB2FixedEqTS6}
						\end{align}
						Let us decompose $L$ as:
						\begin{align*}
							L & =L_{m}+\boldsymbol{\ell }_{m}^{n_{m}}L_{m+1},
						\end{align*}
						where $L_{m}\in \mathcal{B}_{m}^{+}\subseteq \mathcal{L}_{k}$, $n_{m}+1\leq \mathrm{ord}_{\boldsymbol{\ell }_{m}}(L_{m})$, and $L_{m+1}\in \mathcal{B}_{m+1}\subseteq \mathcal{L}_{k}^{\infty }$. By \eqref{PropCaseB2FixedEqTS7} from the previous step, we get:
						\begin{align}
							& \mathcal{P}_{<(\mathbf{0}_{m},2n_{m}+1,2n_{m+1}+1)}(LD_{m}(S_{m})-S_{m}D_{m}(L)) \nonumber \\
							& = K_{m} + \mathcal{P}_{<(\mathbf{0}_{m},2n_{m}+1,2n_{m+1}+1)}(\boldsymbol{\ell }_{m}^{2n_{m}+1}G_{m+1}) \nonumber \\
							& = K_{m} + \boldsymbol{\ell }_{m}^{2n_{m}+1}\mathcal{P}_{<(\mathbf{0}_{m+1},2n_{m+1}+1)}(G_{m+1}) . \label{PropCaseB2FixedEqTS5}
						\end{align}
						By \eqref{PropCaseB2FixedEqTS5}, it follows that equation \eqref{PropCaseB2FixedEqTS6} is equivalent to the equation:
						\begin{align*}
							& \mathcal{P}_{<(\mathbf{0}_{m+1},2n_{m+1}+1)} \big( L_{m+1}D_{m+1}(S) -SD_{m+1}(L_{m+1}) \big) \\
							& = K_{m+1} - \mathcal{P}_{<(\mathbf{0}_{m+1},2n_{m+1}+1)}(G_{m+1}) .
						\end{align*}
						By Lemma~\ref{LemaInductiveStep}, there exists a solution $S:=S_{m+1}\in \mathcal{B}_{m+1}\subseteq \mathcal{L}_{k}^{\infty }$ of the previous equation, such that each term in $S_{m+1}$ is of order in $\boldsymbol{\ell }_{m+1}$ strictly smaller than $n_{m+1}$. Now, put:
						\begin{align*}
							\boldsymbol{\ell }_{m}^{2n_{m}+1}\boldsymbol{\ell }_{m+1}^{2n_{m+1}+1}G_{m+2} := \mathcal{P}_{<r_{0}}\big( \mathcal{T}(S_{m}+S_{m+1}) - K_{m}-\boldsymbol{\ell }_{m}^{2n_{m}+1}K_{m+1} \big) ,
						\end{align*}
						where $G_{m+2}\in \mathcal{B}_{m+2}\subseteq \mathcal{L}_{k}^{\infty }$ such that $\mathrm{ord}(\boldsymbol{\ell }_{m}^{2n_{m}+1}\boldsymbol{\ell }_{m+1}^{2n_{m+1}+1}G_{m+2})<r_{0}$. \\
						
						\noindent Inductively, in $k-m+1$ steps, we find $S_{m},\ldots ,S_{k}$. Now, put:
						\begin{align*}
							S & :=S_{m}+\boldsymbol{\ell }_{m}^{n_{m}}S_{m+1}\cdots +\boldsymbol{\ell }_{m}^{n_{m}}\cdots \boldsymbol{\ell }_{k-1}^{n_{k-1}}S_{k} .
						\end{align*}
						Note that $S\in \widetilde{\mathcal{B}}$, and, by the induction, $S$ is a solution of the equation $\mathcal{T}_{m}|_{\widetilde{\mathcal{B}}}(S)=K$. Therefore, $\mathcal{S}_{m}(\widetilde{\mathcal{B}})\subseteq \mathcal{T}_{m}(\widetilde{\mathcal{B}})$.
						
						Now we conclude by the fixed point theorem from Proposition~\ref{KorBanach}.
					\end{proof}
				
					\begin{proof}[Proof of step (b.2)]
						For simplicity of notation we denote again by $f$ logarithmic transseries $\varphi _{1}\circ f\circ \varphi _{1}^{-1}$, with $\varphi _{1}\in \mathcal{L}_{k}^{0}$ obtained in step $(b.1)$. Now, put
						\begin{align*}
							& f=\mathrm{id}+zR+\mathrm{h.o.b.}(z) , \quad R\in \mathcal{B}_{\geq 1}^{+}\subseteq \mathcal{L}_{k} .
						\end{align*}
						Let $\mathrm{ord} \, (f - \mathrm{id}):=(1,\mathbf{0}_{m-1},n_{m},\ldots ,n_{k})$, for $n_{m}\geq 1$, $1\leq m\leq k$, and let $L$ be as defined in the Main Theorem. Suppose additionally that $R\neq L$ and $\mathrm{ord}\, (z(R-L)) < \mathrm{ord}\, (\mathrm{Res}\, (f))$. Otherwise, we skip step $(b.2)$ and proceed to step $(b.3)$.
						
						Now, $f$ satisfies the assumptions of Lemma~\ref{CaseBLemma6}. Therefore, we transform the conjugacy equation $\mathcal{P}_{<r}(\varphi _{2}\circ f\circ \varphi _{2}^{-1})=\mathrm{id}+zL$ to equation \eqref{B2LemaEquation}.
						
						Now, put $V:=\boldsymbol{\ell }_{m}T_{m}\in \mathcal{B}_{m}^{+}\subseteq \mathcal{L}_{k}$, where $T_{m}$ is from decomposition \eqref{EquationLemaDecomposition} in Lemma~\ref{CaseBLemma6}. Since step (b.1) has been applied on $f$, note that $\mathcal{P}_{\leq (\mathbf{1}_{m},n_{m}+1)}(f)=\mathrm{id}+zL$, and therefore, $\mathrm{ord}_{\boldsymbol{\ell }_{m}}(V) \geq \mathrm{ord}_{\boldsymbol{\ell }_{m}}(L)+2$. Let $h:=(1+x)\log (1+x)-x\in x^{2}\mathbb{R}\left[ \left[ x\right] \right] $. If $m=1$, we take $\mathcal{C}_{1}:=\mathcal{C}:\mathcal{B}_{1}^{+}\to \mathcal{B}_{1}^{+}$, for $\frac{1}{2^{2+n_{1}}}$-contraction $\mathcal{C}$, with respect to the metric $d_{1}$, from equation \eqref{B2LemaEquation}. If $m>1$, take $\mathcal{C}_{m}:\mathcal{B}_{m}^{+}\to \mathcal{B}_{m}^{+}$ such that $\mathcal{C}_{m}(S):=0$, for each $S\in \mathcal{B}_{m}^{+}$. Evidently, for $m\geq 1$, $\mathcal{C}_{m}:\mathcal{B}_{m}^{+}\to \mathcal{B}_{m}^{+}$ is a $\frac{1}{2^{2+n_{m}}}$-contraction, with respect to the metric $d_{m}$. Therefore, we can apply Proposition~\ref{Prop4}. So there exists a unique solution $S\in \widetilde{\mathcal{B}}\subseteq \mathcal{B}_{m}^{+}\subseteq \mathcal{L}_{k}$ (for $\widetilde{\mathcal{B}}$ as defined in the proof of Proposition~\ref{Prop4}) of equation \eqref{B2LemaEquation}. Now, by Lemma~\ref{CaseBLemma6}, it follows that $\varphi _{2}:=\mathrm{id}+zS$ (here, we take the solution with the simplest choice $\varepsilon :=0$ in Lemma~\ref{CaseBLemma6}) is one solution of the equation
						\begin{align*}
							\mathcal{P}_{<r}(\varphi _{2}\circ f\circ \varphi _{2}^{-1}) & =\mathrm{id}+zL .
						\end{align*}
						Therefore, there exists $c\in \mathbb{R}$ such that:
						\begin{align*}
							\varphi _{2}\circ f\circ \varphi _{2}^{-1} & =\mathrm{id}+zL+c\mathrm{Res} \, (f) + \mathrm{h.o.t.}
						\end{align*}
						This completes step (b.2).
					\end{proof}
				
					\begin{remark}[Non-uniqueness of conjugacy $\varphi _{2}$ in step (b.2)]\label{RemCaseB2}
						Let $\varphi _{2}:=\mathrm{id}+zS+\varepsilon $, for $S\in \mathcal{B}_{\geq 1}^{+}\subseteq \mathcal{L}_{k}$ and $\varepsilon \in \mathcal{L}_{k}$, $\mathrm{ord}_{z} \, (\varepsilon )>1$, be a parabolic solution of conjugacy equation \eqref{SecondConj}. We do not claim the uniqueness of the solution $\varphi _{2}$. Let us decompose $S$ as $S=S_{m}+S_{m-1}+\cdots +S_{1}$, for $S_{i}\in \mathcal{B}_{i}^{+}\subseteq \mathcal{L}_{k}$, $1\leq i\leq m-1$, and $S_{m}\in \mathcal{B}_{\geq m}^{+}\subseteq \mathcal{L}_{k}$. By Lemma~\ref{CaseBLemma6} it follows that $\varphi _{2}$ is a solution of equation \eqref{SecondConj} if and only if $S_{m}\in \mathcal{B}_{m}^{+}\subseteq \mathcal{L}_{k}$ and if $S_{m}$ is a solution of differential equation given in \eqref{B2LemaEquation}. Therefore, we can choose arbitrary $S_{i}\in \mathcal{B}_{i}^{+}$, $1\leq i\leq m-1$, and $\varepsilon \in \mathcal{L}_{k}$, $\mathrm{ord}_{z}(\varepsilon ) >1$, such that $\varphi _{2}$ is still a solution of \eqref{SecondConj}. Although $\varphi _{2}$ is not unique, if we request the \emph{canonical form} of $\varphi _{2}$, that is, $\varphi _{2}:=\mathrm{id}+zS_{m}$, for $S_{m}\in \widetilde{\mathcal{B}}\subseteq \mathcal{B}_{m}^{+}\subseteq \mathcal{L}_{k}$ (for $\widetilde{\mathcal{B}}$ as defined in the proof of Proposition~\ref{Prop4}), then, by Lemma~\ref{CaseBLemma6}, $S_{m}$ satisfies differential equation \eqref{B2LemaEquation} and, by Proposition~\ref{Prop4}, such $S_{m}$ is unique in $\widetilde{\mathcal{B}}$. Therefore, if $\varphi _{2}$ is a \emph{canonical} solution of conjugacy equation \eqref{SecondConj}, then it is unique. 
					\end{remark}

					\subsubsection{Proof of step (b.3)}
					
					We first transform the conjugacy equation \eqref{ThirdConj} to the fixed point equation in Lemma~\ref{Lem1CaseB}. Then we apply Proposition~\ref{PropCaseBTS} and Proposition~\ref{KorBanach} to solve the mentioned fixed point equation. 
					
					\begin{lem}[Transforming conjugacy equation \eqref{ThirdConj} into a fixed point equation]\label{Lem1CaseB}
						Let $f\in \mathcal{L}_{k}^{0}$, $k\in \mathbb{N}_{\geq 1}$, such that $f:=\mathrm{id}+zR+\mu $, $R\in \mathcal{B}_{\geq 1}^{+}\subseteq \mathcal{L}_{k}$, $\mu \in \mathcal{L}_{k}$, $\mathrm{ord}_{z} \, (\mu ) > 1$, and $g:=\mathrm{id}+zT$, $T\in \mathcal{B}_{\geq 1}^{+}\subseteq \mathcal{L}_{k}$. The conjugacy equation
						\begin{align*}
							& \varphi \circ f\circ \varphi ^{-1}=g,
						\end{align*}
						where we write $\varphi :=\mathrm{id}+zS+\varepsilon $, for $S\in \mathcal{B}_{\geq 1}^{+}\subseteq \mathcal{L}_{k}$, $\varepsilon  \in \mathcal{L}_{k}$, $\mathrm{ord}_{z} \, (\varepsilon ) > 1$, is equivalent to the fixed point equation
						\begin{align*}
							\mathcal{T}_{f}(zS+\varepsilon) &=\mathcal{S}_{f}(zS+\varepsilon)
						\end{align*}
						on the space $\mathcal{L}_{k}^{1}$, where:
						{\small \begin{align}\label{OperST}
							\mathcal{S}_{f}(zS+\varepsilon) &:= \sum _{i\geq 1}\frac{(zS+\varepsilon )^{(i)}(\mathrm{id}+zR)}{i!}\mu ^{i}-\sum _{i\geq 2}\frac{(zT)^{(i)}(\mathrm{id}+zS)}{i!}\varepsilon^{i}+z(R-T)+\mu , \nonumber \\
							\mathcal{T}_{f}(zS+\varepsilon) &:=(zT)'\circ (\mathrm{id}+zS)\cdot \varepsilon - \sum _{i\geq 1}\frac{\varepsilon ^{(i)}}{i!}(zR)^{i}+\sum _{i\geq 1}\frac{(zT)^{(i)}}{i!}(zS)^{i}-\sum _{i\geq 1}\frac{(zS)^{(i)}}{i!}(zR)^{i}. 
						\end{align} }
					\end{lem}
				
					\begin{proof}
						The conjugacy equation $\varphi \circ f\circ \varphi ^{-1}=g$ is equivalent to the equation:
						\begin{align}
							\varphi \circ f & =g\circ \varphi  , \nonumber
						\end{align}
						for $\varphi := \mathrm{id}+zS+\varepsilon $, $S\in \mathcal{B}_{\geq 1}^{+}\subseteq \mathcal{L}_{k}$, $\varepsilon \in \mathcal{L}_{k}$, $\mathrm{ord}_{z}(\varepsilon ) >1$. By the Taylor Theorem (see \cite[Proposition 3.3]{prrs21}) we get that:
						{\small \begin{align}
							& zS+\varepsilon+\sum _{i\geq 1}\frac{(zS+\varepsilon )^{(i)}(\mathrm{id}+zR)}{i!}\mu ^{i}+\mathrm{id}+zR+\mu  + \sum _{i\geq 1}\frac{(zS)^{(i)}}{i!}(zR)^{i} + \sum _{i\geq 1}\frac{\varepsilon ^{(i)}}{i!}(zR)^{i} \nonumber \\
							&= \mathrm{id}+zS+\varepsilon+zT+\sum _{i\geq 1}\frac{(zT)^{(i)}(\mathrm{id}+zS)}{i!}\varepsilon^{i}+\sum _{i\geq 1}\frac{(zT)^{(i)}}{i!}(zS)^{i}. \label{LemmaBEquationTaylor}
						\end{align}
						}Let us define the operators $\mathcal{S}_{f}, \mathcal{T}_{f} : \mathcal{L}_{k}^{1}\to \mathcal{L}_{k}^{1}$ as in \eqref{OperST}. By \eqref{LemmaBEquationTaylor}, it follows that the conjugacy equation is equivalent to the equation $\mathcal{S}_{f}(zS+\varepsilon)=\mathcal{T}_{f}(zS+\varepsilon)$, for $S\in \mathcal{B}_{\geq 1}^{+}\subseteq \mathcal{L}_{k}$, $\varepsilon\in \mathcal{L}_{k}$, $\mathrm{ord}_{z}(\varepsilon ) >1$.
					\end{proof}
				
					\begin{prop}[Properties of operators $\mathcal{S}_{f}$ and $\mathcal{T}_{f}$]\label{PropCaseBTS}
						Let $f\in \mathcal{L}_{k}^{0}$, $k\in \mathbb{N}_{\geq 1}$, such that $f:=\mathrm{id}+zR+\mu $, for $R\in \mathcal{B}_{\geq 1}^{+}\setminus \left\lbrace 0\right\rbrace \subseteq \mathcal{L}_{k}$, and let $\alpha := \mathrm{ord}_{z}(\mu )>1$. Let $g:=\mathrm{id}+zT$, $T\in \mathcal{B}_{\geq 1}^{+}\subseteq \mathcal{L}_{k}$ and let $\mathcal{S}_{f}, \mathcal{T}_{f}:z\mathcal{B}_{\geq 1}^{+}\oplus \mathcal{L}_{k}^{\alpha }\to z\mathcal{B}_{\geq 1}^{+}\oplus \mathcal{L}_{k}^{\alpha }$ be operators defined in \eqref{OperST}. If $\mathrm{ord}(z(R-T))>\mathrm{ord}(\mathrm{Res}(f))$, then:
						\begin{enumerate}[1., font=\textup, nolistsep, leftmargin=0.6cm]
							\item the operator $\mathcal{S}_{f}$ is a $\frac{1}{2^{\alpha -1}}$-contraction on the space $(z\mathcal{B}_{\geq 1}^{+}\oplus \mathcal{L}_{k}^{\alpha },d_{z})$,
							\item the restriction of the operator $\mathcal{T}_{f}$ on the space
							\begin{align*}
								\widetilde{\mathcal{L}} & := \left\lbrace h\in z\mathcal{B}_{\geq 1}^{+}\oplus \mathcal{L}_{k}^{\alpha } : \mathrm{ord} \, (h) > \mathrm{ord} \, (f-\mathrm{id}) \right\rbrace 
							\end{align*}
							is an isometry, with respect to the metric $d_{z}$,
							\item $\mathcal{S}_{f}(\widetilde{\mathcal{L}}) \subseteq \mathcal{T}_{f}(\widetilde{\mathcal{L}})$.
						\end{enumerate}
					\end{prop}
				
					\begin{proof}
						1. First, note that $\mathcal{S}_{f}(z\mathcal{B}_{\geq 1}^{+}\oplus \mathcal{L}_{k}^{\alpha }), \mathcal{T}_{f}(z\mathcal{B}_{\geq 1}^{+}\oplus \mathcal{L}_{k}^{\alpha }) \subseteq z\mathcal{B}_{\geq 1}^{+}\oplus \mathcal{L}_{k}^{\alpha }$. Indeed, let $S\in \mathcal{B}_{\geq 1}^{+}\subseteq \mathcal{L}_{k}$ and $\varepsilon \in \mathcal{L}_{k}^{\alpha }$. Since $(zT)^{(i)}(\mathrm{id}+zS)=\mathrm{Lt}\big( (zT)^{(i)}\big) +\mathrm{h.o.t.}$, $(zS+\varepsilon )^{(i)}(\mathrm{id}+zR)=\mathrm{Lt}\big( (zS)^{(i)}\big) +\mathrm{h.o.t.}$, it follows that:
						\begin{align*}
							& \mathrm{ord}_{z}\Big( \sum _{i\geq 1}\frac{(zS+\varepsilon )^{(i)}(\mathrm{id}+zR)}{i!}\mu ^{i} \Big) \geq \mathrm{ord}_{z}(\mu) = \alpha , \\
							& \mathrm{ord}_{z}\Big( \sum _{i\geq 2}\frac{(zT)^{(i)}(\mathrm{id}+zS)}{i!}\varepsilon^{i} \Big) \geq 2\cdot \mathrm{ord}_{z}(\varepsilon ) -1\geq \alpha .
						\end{align*}
						Therefore, $\mathcal{S}_{f}(zS+\varepsilon ) \in z\mathcal{B}_{\geq 1}^{+}\oplus \mathcal{L}_{k}^{\alpha }$, and, consequently, $\mathcal{S}_{f}(z\mathcal{B}_{\geq 1}^{+}\oplus \mathcal{L}_{k}^{\alpha }) \subseteq z\mathcal{B}_{\geq 1}^{+}\oplus \mathcal{L}_{k}^{\alpha }$. 
						
						Similarly, since $(zT)'\circ (\mathrm{id} +zS)\cdot \varepsilon = \mathrm{Lt}(T\cdot \varepsilon )+\mathrm{h.o.t.}$ and $\sum _{i\geq 1}\frac{\varepsilon ^{(i)}}{i!}(zR)^{i}=\mathrm{Lt}(\varepsilon '\cdot zR)+\mathrm{h.o.t.}$, it follows that:
						\begin{align*}
							& \mathrm{ord}_{z}\Big( (zT)'\circ (\mathrm{id} +zS)\cdot \varepsilon  \Big) = \mathrm{ord} _{z}(\varepsilon ) \geq \alpha , \\
							& \mathrm{ord}_{z}\Big( \sum _{i\geq 1}\frac{\varepsilon ^{(i)}}{i!}(zR)^{i} \Big) = \mathrm{ord}_{z}(\varepsilon ) \geq \alpha .
						\end{align*}
						Therefore, $\mathcal{T}_{f}(zS+\varepsilon ) \in z\mathcal{B}_{\geq 1}^{+}\oplus \mathcal{L}_{k}^{\alpha }$, and, consequently, $\mathcal{T}_{f}(z\mathcal{B}_{\geq 1}^{+}\oplus \mathcal{L}_{k}^{\alpha }) \subseteq z\mathcal{B}_{\geq 1}^{+}\oplus \mathcal{L}_{k}^{\alpha }$.
						
						Now we prove that $\mathcal{S}_{f}$ is a contraction. Let $zS_{1}+\varepsilon _{1}, zS_{2}+\varepsilon _{2} \in z\mathcal{B}_{\geq 1}^{+}\oplus \mathcal{L}_{k}^{\alpha }$, $S_{1},S_{2}\in \mathcal{B}_{\geq 1}^{+}$, $\varepsilon _{1}, \varepsilon _{2} \in \mathcal{L}_{k}^{\alpha }$, such that $zS_{1}+\varepsilon _{1} \neq zS_{2}+\varepsilon _{2}$. We distinguish two cases. \\
						(a) \emph{Case} $\varepsilon_{1} \neq \varepsilon _{2}$. Since $(zT)^{(i)}(\mathrm{id}+zS_{1})=\mathrm{Lt}\big( (zT)^{(i)}\big) +\mathrm{h.o.t.}$ and $\varepsilon_{1},\varepsilon_{2}\in \mathcal{L}_{k}^{\alpha }$, we get:
						\begin{align*}
							& \mathrm{ord}_{z}\Big( \sum _{i\geq 2}\frac{(zT)^{(i)}(\mathrm{id}+zS_{1})}{i!}\varepsilon_{1}^{i}-\sum _{i\geq 2}\frac{(zT)^{(i)}(\mathrm{id}+zS_{2})}{i!}\varepsilon_{2}^{i}\Big) \\
							& \geq \mathrm{ord}_{z}(\varepsilon _{1}-\varepsilon _{2})+(\alpha -1) \\
							& \geq \mathrm{ord}_{z}\big( (zS_{1}+\varepsilon _{1})-(zS_{2}+\varepsilon _{2})\big) +(\alpha -1) .
						\end{align*}
						(b) \emph{Case} $\varepsilon _{1}=\varepsilon _{2}$. Then $S_{1}\neq S_{2}$ and:
						\begin{align*}
							& \sum _{i\geq 2}\frac{(zT)^{(i)}(\mathrm{id}+zS_{1})}{i!}\varepsilon_{1}^{i}-\sum _{i\geq 2}\frac{(zT)^{(i)}(\mathrm{id}+zS_{2})}{i!}\varepsilon_{1}^{i} \\
							& = -\mathrm{Lt}\Big( \frac{1}{2}z^{-1}D_{1}(T)(S_{1}-S_{2})\varepsilon_{1}^{2} \Big) + \mathrm{h.o.t.}
						\end{align*}
						That is:
						\begin{align*}
							& \mathrm{ord}_{z} \Big( \sum _{i\geq 2}\frac{(zT)^{(i)}(\mathrm{id}+zS_{1})}{i!}\varepsilon_{1}^{i}-\sum _{i\geq 2}\frac{(zT)^{(i)}(\mathrm{id}+zS_{2})}{i!}\varepsilon_{1}^{i} \Big) \\
							& \geq 2\alpha -1 \\
							& \geq \mathrm{ord}_{z}\big( (zS_{1}+\varepsilon _{1})-(zS_{2}+\varepsilon _{1})\big) +(\alpha -1) .
						\end{align*}
						Therefore, in both cases, it follows that:
						\begin{align}
							& \mathrm{ord}_{z}\Big( \sum _{i\geq 2}\frac{(zT)^{(i)}(\mathrm{id}+zS_{1})}{i!}\varepsilon_{1}^{i}-\sum _{i\geq 2}\frac{(zT)^{(i)}(\mathrm{id}+zS_{2})}{i!}\varepsilon_{2}^{i} \Big) \nonumber \\
							& \geq \mathrm{ord}_{z}\big( (zS_{1}+\varepsilon _{1}\big) -(zS_{2}+\varepsilon _{2}))+(\alpha -1) . \label{InequalityOrder3}
						\end{align}
						Note that:
						\begin{align}
							& \mathrm{ord}_{z}\Big( \sum _{i\geq 1}\frac{(zS_{1}+\varepsilon _{1})^{(i)}(\mathrm{id}+zR)}{i!}\mu ^{i} - \sum _{i\geq 1}\frac{(zS_{2}+\varepsilon _{2})^{(i)}(\mathrm{id}+zR)}{i!}\mu ^{i} \Big) \nonumber \\
							& = \mathrm{ord}_{z}\Big( \sum _{i\geq 1}\frac{(z(S_{1}-S_{2})+(\varepsilon _{1}-\varepsilon _{2}))^{(i)}(\mathrm{id}+zR)}{i!}\mu ^{i} \Big) \nonumber \\
							& \geq \mathrm{ord}_{z}\big( z(S_{1}-S_{2})+(\varepsilon _{1}-\varepsilon _{2})\big) + \mathrm{ord}_{z}(\mu) -1 \nonumber \\
							& \geq \mathrm{ord}_{z}\big( (zS_{1}+\varepsilon _{1}\big) -(zS_{2}+\varepsilon _{2})) + \alpha -1 . \label{InequalityOrder4}
						\end{align}
						
						\noindent Finally, statement 1 follows from \eqref{InequalityOrder3} and \eqref{InequalityOrder4}. \\
						
						3. Let $h\in \mathcal{S}_{f}(\widetilde{\mathcal{L}})$, where $\widetilde{\mathcal{L}}\subseteq z\mathcal{B}_{\geq 1}^{+}\oplus \mathcal{L}_{k}^{\alpha }$ is as defined in Proposition~\ref{PropCaseBTS}. Since $\mathcal{S}_{f}(\widetilde{\mathcal{L}}) \subseteq z\mathcal{B}_{\geq 1}^{+}\oplus \mathcal{L}_{k}^{\alpha }$ and $\mathrm{ord}(z(R-T))>\mathrm{ord}(\mathrm{Res}(f))$, by formula \eqref{OperST} for $\mathcal{S}_{f}$, $h$ can be written in the \emph{form}:
						\begin{align*}
							h &=zM+\sum _{\beta \geq \alpha }z^{\beta }M_{\beta },
						\end{align*}
						where $M\in \mathcal{B}_{\geq 1}^{+}\subseteq \mathcal{L}_{k}$, $\mathrm{ord}(zM)>\mathrm{ord}(\mathrm{Res}(f))$, $M_{\beta }\in \mathcal{B}_{1}\subseteq \mathcal{L}_{k}^{\infty }$, $\beta \geq \alpha $. We prove that there exists $zS+\varepsilon \in \widetilde{\mathcal{L}}$ such that:
						\begin{align}
							& \mathcal{T}_{f}(zS+\varepsilon ) =h=zM+\sum _{\beta \geq \alpha }z^{\beta }M_{\beta } . \label{PropClaim3Eq}
						\end{align}
						Therefore, we search for solution of \eqref{PropClaim3Eq} in the \emph{form}:
						\begin{align}
							& zS+\sum _{\beta \geq \alpha }z^{\beta }S_{\beta }, \quad S\in \mathcal{B}_{\geq 1}^{+}\subseteq \mathcal{L}_{k}, \, \mathrm{ord} \, (S)>\mathrm{ord} \, (R) , \, S_{\beta }\in \mathcal{B}_{1}\subseteq \mathcal{L}_{k}^{\infty } . \label{PropCaseB3EqDecomp}
						\end{align}
						From \eqref{OperST} and by comparing the blocks with the same order in $z$, it follows that equation \eqref{PropClaim3Eq} is equivalent to the following system of equations:
						\begin{align}
							& (zT)'\circ (\mathrm{id}+zS)\cdot z^{\beta }S_{\beta }-\sum _{i\geq 1}\frac{(z^{\beta }S_{\beta })^{(i)}}{i!}(zR)^{i} = z^{\beta }M_{\beta } , \quad \beta \geq \alpha , \label{CaseBIdent8}
						\end{align}
						\begin{align}
							& \sum _{i\geq 1}\frac{(zT)^{(i)}}{i!}(zS)^{i}-\sum _{i\geq 1}\frac{(zS)^{(i)}}{i!}(zR)^{i} = zM . \label{CaseBIdent1}
						\end{align}
						By \eqref{Identitet1} and \eqref{Identitet2} it follows that:
						\begin{align}
							& \sum _{i\geq 1}\frac{(z^{\beta }S_{\beta })^{(i)}}{i!}(zR)^{i} = z^{\beta }\cdot \bigg( S_{\beta }\cdot \sum _{i\geq 1}{\beta \choose i}R^{i}+ \mathcal{C}_{\beta }(S_{\beta })\bigg) , \nonumber \\
							& \sum _{i\geq 1}\frac{(zT)^{(i)}}{i!}(zS)^{i} = z\left( T\cdot S+D_{1}(T)\cdot (1+S)\cdot \log (1+S) + \mathcal{C}_{1}(S) \right) , \nonumber \\
							& \sum _{i\geq 1}\frac{(zS)^{(i)}}{i!}(zR)^{i} =z\left( S\cdot R+D_{1}(S)\cdot (1+R)\cdot \log (1+R) + \mathcal{K}_{1}(S) \right) , \label{CaseBIdent2}
						\end{align}
						for linear $\frac{1}{2^{1+\mathrm{ord}_{\boldsymbol{\ell }_{1}}(R)}}$-contractions $\mathcal{C}_{\beta }:(\mathcal{B}_{1},d_{1})\to (\mathcal{B}_{1},d_{1})$, $\beta \geq \alpha $, and $\frac{1}{2^{2+\mathrm{ord}_{\boldsymbol{\ell }_{1}}(R)}}$-contractions $\mathcal{C}_{1},\mathcal{K}_{1}:(\mathcal{B}_{\geq 1}^{+},d_{1})\to (\mathcal{B}_{\geq 1}^{+},d_{1})$. By \eqref{CaseBIdent2}, after eliminating the variable $z$, we get that solving \eqref{CaseBIdent8} and \eqref{CaseBIdent1} is equivalent to solving:
						\begin{align}
							& \mathcal{S}_{\beta }(S_{\beta }) = S_{\beta } , \quad S_{\beta }\in \mathcal{B}_{1}\subseteq \mathcal{L}_{k}^{\infty } , \, \beta \geq \alpha , \label{PropCaseB3EquatFixed1}
						\end{align}
						and
						\begin{align}
							& \mathcal{T}_{1}(S) = \mathcal{S}_{1}(S) , \quad S\in \mathcal{B}_{\geq 1}^{+}\subseteq \mathcal{L}_{k}, \, \mathrm{ord} \, (S) > \mathrm{ord} \, (R) , \label{PropCaseB3EquatFixed2}
						\end{align}
						where:
						\begin{align}
							& \mathcal{S}_{\beta }(S_{\beta }):= \frac{M_{\beta }-S_{\beta }\cdot D_{1}(T)\circ (\mathrm{id}+zS)+\mathcal{C}_{\beta }(S_{\beta })}{T\circ (\mathrm{id}+zS)-\sum _{i\geq 1}{\beta \choose i}R^{i}}, \label{CaseBIdent6}
						\end{align}
						and
						\begin{align}
							\mathcal{S}_{1}(S) &:= M+\mathcal{K}_{1}(S)-\mathcal{C}_{1}(S) -(1+S)\cdot \log (1+S)\cdot D_{1}(T-R), \nonumber \\
							\mathcal{T}_{1}(S) &:=-(1+R)\cdot \log (1+R)\cdot D_{1}(S)+(1+S)\cdot \log(1+S)\cdot D_{1}(R)+S\cdot (T-R) . \label{CaseBIdent5}
						\end{align}
						
						Since fixed point equations $\mathcal{S}_{\beta }(S_{\beta }) = S_{\beta }$, $\beta \geq \alpha $, depend on solutions $S\in \mathcal{B}_{\geq 1}^{+}\subseteq \mathcal{L}_{k}$, $\mathrm{ord} \, (S) > \mathrm{ord} \, (R)$, of the fixed point equation $\mathcal{T}_{1}(S)=\mathcal{S}_{1}(S)$, we first determine the unique solution $S$ of equation \eqref{PropCaseB3EquatFixed2}. \\
						
						\noindent \emph{Solving equation} \eqref{PropCaseB3EquatFixed2}. Let
						\begin{align*}
							\widetilde{\mathcal{B}} & := \left\lbrace K\in \mathcal{B}_{\geq 1}^{+}\subseteq \mathcal{L}_{k} : \mathrm{ord} \, (K) > \mathrm{ord} \, (R) \right\rbrace . 
						\end{align*}
						Now, for the purpose of applying Proposition~\ref{LemaCaseB2OperatorsTS} on operators $\mathcal{T}_{1},\mathcal{S}_{1}:\widetilde{\mathcal{B}} \to \mathcal{B}_{\geq 1}^{+}$ defined by \eqref{CaseBIdent5}, put $V:=T-R$. Since $\mathrm{ord} \, \big( z(R-T) \big) , \mathrm{ord} \, (zM)>\mathrm{ord} \, (\mathrm{Res} \, (f))$, we get $\mathrm{ord} \, (V), \mathrm{ord} \, (M)> 2\cdot \mathrm{ord} \, (R)+(0,\mathbf{1}_{k})$.
						
						By Proposition~\ref{LemaCaseB2OperatorsTS} operators $\mathcal{T}_{1}$ and $\mathcal{S}_{1}$ satisfy the assumptions of Proposition~\ref{KorBanach} on the complete space $(\widetilde{\mathcal{B}},d_{1})$. Therefore, by Proposition~\ref{KorBanach}, there exists a unique $S\in \widetilde{\mathcal{B}}$ such that $\mathcal{T}_{1}(S)=\mathcal{S}_{1}(S)$. \\
						
						\noindent \emph{Solving equations} \eqref{PropCaseB3EquatFixed1}. For a unique solution $S\in \widetilde{\mathcal{B}}$ of the fixed point equation $\mathcal{T}_{1}(S)=\mathcal{S}_{1}(S)$ we prove that operators $\mathcal{S}_{\beta }:\mathcal{B}_{1}\to \mathcal{B}_{1}$, $\beta \geq \alpha $, defined by \eqref{CaseBIdent6}, are $\frac{1}{2}$-contractions on the space $(\mathcal{B}_{1},d_{1})$.
						
						Since $\mathrm{ord}\big( z(R-T) \big) >\mathrm{ord}(\mathrm{Res}(f))$, it follows that $\mathrm{Lt}(T)=\mathrm{Lt}(R)$. Using ${\beta \choose 1}=\beta \geq \alpha >1$ and $T\circ (\mathrm{id}+zS)=T+\mathrm{h.o.t.}$, we get:
						\begin{align*}
							\frac{S_{\beta }\cdot D_{1}(T)\circ (\mathrm{id}+zS)+\mathcal{C}_{\beta }(S_{\beta })}{T\circ (\mathrm{id}+zS)-\sum _{i\geq 1}{\beta \choose i}R^{i}} & = \frac{S_{\beta }\cdot D_{1}(T)\circ (\mathrm{id}+zS)+\mathcal{C}_{\beta }(S_{\beta })}{(1- \beta )\mathrm{Lt}(R)+ \mathrm{h.o.t.}} ,
						\end{align*}
						for $S_{\beta }\in \mathcal{B}_{1}\subseteq \mathcal{L}_{k}^{\infty }$ and each $\beta \geq \alpha $. From this, using the facts that $\mathcal{C}_{\beta }$ is a linear $\frac{1}{2^{1+\mathrm{ord}_{\boldsymbol{\ell }_{1}}(R)}}$-contraction on $(\mathcal{B}_{1},d_{1})$ and that $\mathrm{ord}_{\boldsymbol{\ell }_{1}} \big( D_{1}(T)\circ (\mathrm{id}+zS) \big) =\mathrm{ord}_{\boldsymbol{\ell }_{1}}(R)+1$, we get:
						\begin{align*}
							\mathrm{ord}_{\boldsymbol{\ell}_{1}} \bigg( \frac{S_{\beta }\cdot D_{1}(T)\circ (\mathrm{id}+zS)+\mathcal{C}_{\beta }(S_{\beta })}{T\circ (\mathrm{id}+zS)-\sum _{i\geq 1}{\beta \choose i}R^{i}}\bigg) & \geq \mathrm{ord}_{\boldsymbol{\ell }_{1}}(S_{\beta })+1,
						\end{align*}
						for $S_{\beta }\in \mathcal{B}_{1}\subseteq \mathcal{L}_{k}^{\infty }$ and each $\beta \geq \alpha$. Since $\mathcal{S}_{\beta }$, $\beta \geq \alpha $, are affine operators, it follows that $\mathcal{S}_{\beta }$, $\beta \geq \alpha $, are $\frac{1}{2}$-contractions on the space $(\mathcal{B}_{1},d_{1})$. \\
						By the Banach Fixed Point Theorem, there exists a unique solution $S_{\beta }\in \mathcal{B}_{1}\subseteq \mathcal{L}_{k}^{\infty }$ of equation \eqref{CaseBIdent8}, for each $\beta \geq \alpha $. \\
						
						Finally, by putting $\varepsilon:=\sum _{\beta \geq \alpha }z^{\beta }S_{\beta }$, where $S_{\beta }\in \mathcal{B}_{1}\subseteq \mathcal{L}_{k}^{\infty }$ are unique solutions of \eqref{CaseBIdent8}, for $\beta \geq \alpha $, and taking $S\in \widetilde{\mathcal{B}}$ the unique solution of \eqref{PropCaseB3EquatFixed2}, the system of equations \eqref{CaseBIdent8}, \eqref{CaseBIdent1} is satisfied, and therefore, we get $\mathcal{T}_{f}(zS+\varepsilon)=h$, for $h\in \mathcal{S}_{f}(\widetilde{\mathcal{L}})$ chosen in \eqref{PropClaim3Eq}. Since $\mathrm{ord} \, (S) > \mathrm{ord} \, (R)$, it follows that $zS+\varepsilon \in \widetilde{\mathcal{L}}$. Consequently, it follows that $\mathcal{S}_{f}(\widetilde{\mathcal{L}}) \subseteq \mathcal{T}_{f}(\widetilde{\mathcal{L}}) $. \\
						
						2. Let $zS_{i}+\varepsilon_{i} \in \widetilde{\mathcal{L}}$, for $i=1,2$, be distinct, and written in \emph{form} \eqref{PropCaseB3EqDecomp}, i.e.
						\begin{align*}
							zS_{i}+\varepsilon _{i} & = zS_{i} + \sum _{\beta \geq \alpha }z^{\beta }S_{\beta ,i} ,
						\end{align*}
						where $S_{i}\in \widetilde{\mathcal{B}}$, and $S_{\beta ,i}\in \mathcal{B}_{1}\subseteq \mathcal{L}_{k}^{\infty }$, $\beta \geq \alpha $, for $i=1,2$. By putting such decompositions in \eqref{OperST} and using \eqref{CaseBIdent2}, we get:
						{\small \begin{align}\label{CaseBIdent3}
								& \mathcal{T}_{f}(zS_{i}+\varepsilon _{i} ) \nonumber \\
								& = z\big( (T-R)S_{i}+D_{1}(T)(1+S_{i})\log (1+S_{i}) -D_{1}(S_{i})(1+R)\log (1+R) + \mathcal{C}_{1}(S_{i}) - \mathcal{K}_{1}(S_{i}) \big) \nonumber \\
								& + \sum _{\beta \geq \alpha }z^{\beta }\bigg( S_{\beta ,i} \cdot \Big( T\circ (\mathrm{id}+zS_{i})-\sum _{i\geq 1}{\beta \choose i}R^{i} + D_{1}(T)\circ (\mathrm{id}+zS_{i})\Big) -\mathcal{C}_{\beta }(S_{\beta ,i}) \bigg) , 
						\end{align}
						}for $i=1,2$. Now, we consider two cases: $S_{1}\neq S_{2}$ and $S_{1}=S_{2}$, and prove that the restriction $\mathcal{T}_{f}|_{\widetilde{\mathcal{L}}}$ is an isometry.
						
						If $S_{1}\neq S_{2}$, since $\mathrm{Lt}(T)=\mathrm{Lt}(R)$, we get that:
						\begin{align}
							\mathcal{T}_{f}(zS_{1}+\varepsilon _{1}) - \mathcal{T}_{f}(zS_{2}+\varepsilon _{2}) &=z\mathrm{Lt} \big( (S_{1}-S_{2})D_{1}(R)-RD_{1}(S_{1}-S_{2}) \big) + \mathrm{h.o.t.} ,
						\end{align}
						assuming that $(S_{1}-S_{2})D_{1}(R)-RD_{1}(S_{1}-S_{2}) \neq 0$. Now, suppose that $(S_{1}-S_{2})D_{1}(R)-RD_{1}(S_{1}-S_{2}) =0$. Solving the linear differential equation, we get $S_{1}-S_{2}=C\cdot R$, for $C\in \mathbb{R}$. Since $\mathrm{ord}(S_{i}) > \mathrm{ord}(R)$, $i=1,2$, we conclude that $C=0$, i.e., $S_{1}=S_{2}$, which is a contradiction. This implies that:
						\begin{align*}
							& \mathrm{ord}_{z}\big( \mathcal{T}_{f}(zS_{1}+\varepsilon _{1}) - \mathcal{T}_{f}(zS_{2}+\varepsilon _{2})\big) = 1 = \mathrm{ord}_{z}\big( (zS_{1}+\varepsilon _{1})-(zS_{2}+\varepsilon _{2})\big) .
						\end{align*}
					
						If $S_{1}=S_{2}$, then $\varepsilon _{1}\neq \varepsilon _{2}$. Put $\beta _{0}:=\mathrm{ord}_{z}(\varepsilon _{1}-\varepsilon _{2})$. By \eqref{CaseBIdent3}, and since $\mathcal{C}_{\beta }$ is a linear $\frac{1}{2^{1+\mathrm{ord}_{\boldsymbol{\ell }_{1}}(R)}}$-contraction (with respect to the metric $d_{1}$), we get that:
						\begin{align*}
							\mathcal{T}_{f}(zS_{1}+\varepsilon _{1})-\mathcal{T}_{f}(zS_{1}+\varepsilon _{2}) &= z^{\beta _{0}}(1-\beta _{0})\mathrm{Lt}(R)\mathrm{Lt}(S_{\beta _{0},1}-S_{\beta _{0},2}) + \mathrm{h.o.t.}
						\end{align*}
						Since $S_{\beta _{0},1} \neq S_{\beta _{0},2}$, $R\neq 0$, and $\beta _{0}\geq \alpha >1$, we have:
						\begin{align*}
							& \mathrm{ord}_{z}\big( \mathcal{T}_{f}(zS_{1}+\varepsilon _{1}) - \mathcal{T}_{f}(zS_{1}+\varepsilon _{2}) \big) =\beta _{0}=\mathrm{ord}_{z}\big( (zS_{1}+\varepsilon _{1}) -(zS_{1}+\varepsilon _{2})\big) .
						\end{align*}
						Therefore, the restriction $\mathcal{T}_{f}|_{\widetilde{\mathcal{L}}}$ is an isometry, with respect to the metric $d_{z}$.
					\end{proof}
				
					\begin{proof}[Proof of step (b.3)]
						In steps $(b.1)$ and $(b.2)$ we obtained logarithmic transseries $\varphi _{1},\varphi _{2} \in \mathcal{L}_{k}^{0}$, such that:
						\begin{align*}
							\varphi _{2}\circ \varphi _{1}\circ f\circ \varphi _{1}^{-1}\circ \varphi _{2}^{-1} = \mathrm{id}+zL +c\mathrm{Res}(f) + \mathrm{h.o.t.} ,
						\end{align*}
						where $L$ is as defined in the Main Theorem and the unique $c$ is given by \eqref{EqRes}. For a simpler notation denote by \emph{new} $f$ the whole composition $\varphi _{2}\circ \varphi _{1}\circ f\circ \varphi _{1}^{-1}\circ \varphi _{2}^{-1}$. Note that now $f=\mathrm{id}+zL +c\mathrm{Res}(f) + \mathrm{h.o.t.}$ Put $zR:=zL+c\mathrm{Res}(f)+\mathrm{h.o.t.}$ and $zT:=zL +c\mathrm{Res}(f)$. Let $g:=\mathrm{id}+zL +c\mathrm{Res}(f)$. By Lemma~\ref{Lem1CaseB}, we define the operators $\mathcal{T}_{f},\mathcal{S}_{f}:\mathcal{L}_{k}^{1} \to \mathcal{L}_{k}^{1}$ and transform the conjugacy equation
						\begin{align}
							\varphi _{3}\circ f \circ \varphi _{3}^{-1} &= g \label{ProofCaseB3Eq}
						\end{align}
						to the equivalent fixed point equation
						\begin{align}
							\mathcal{T}_{f}(zS+\varepsilon ) & =\mathcal{S}_{f}(zS+\varepsilon ) , \label{ProofCaseB3EqFixed}
						\end{align}
						where $\varphi _{3}:=\mathrm{id}+zS+\varepsilon \in \mathcal{L}_{k}^{0}$. Put $\alpha :=\mathrm{ord}_{z}(f-\mathrm{id}-zR)$ and consider the restrictions of operators $\mathcal{T}_{f}$ and $\mathcal{S}_{f}$ on the subspace $z\mathcal{B}_{\geq 1}^{+}\oplus \mathcal{L}_{k}^{\alpha }\subseteq \mathcal{L}_{k}^{1}$. Since $\mathrm{ord}\big( z(R-T) \big) > \mathrm{ord}\, (\mathrm{Res}\, (f))$, by Proposition~\ref{PropCaseBTS} and Proposition~\ref{KorBanach}, there exists a unique solution $\varphi _{3}\in z\mathcal{B}_{\geq 1}^{+}\oplus \mathcal{L}_{k}^{\alpha }$, $\varphi _{3}:=\mathrm{id}+zS+\varepsilon $ of \eqref{ProofCaseB3EqFixed} satisfying $\mathrm{ord} \, (S)> \mathrm{ord} \, (R)$. As a consequence, by Lemma~\ref{Lem1CaseB}, there exists a solution $\varphi _{3}\in \mathcal{L}_{k}^{0}$ of conjugacy equation \eqref{ProofCaseB3Eq}. This completes step $(b.3)$.
						
						Note that $\varphi _{3}$ is a unique solution of conjugacy equation \eqref{ProofCaseB3Eq}, if we additionally impose the condition that $\mathrm{ord} \, (\varphi _{3}-\mathrm{id}) > \mathrm{ord} \, (f-\mathrm{id})$ and $\mathrm{ord}_{z}(\varepsilon ) \geq \mathrm{ord}_{z}(f-\mathrm{id}-zR)$.
					\end{proof}

					\subsection{Proofs of statements 2 and 3 of the Main Theorem}\label{sec:ProofMinimality}
					
					\begin{proof}[Proof of statement 2 of the Main Theorem]
						The idea of the proof of statement 2 is adapted from \cite[Proposition 9.3]{mrrz19tubular}. Let $f\in \mathcal{L}_{k}$, $k\in \mathbb{N}$, be such that $f:=\mathrm{id}+z^{\beta }L+\mathrm{h.o.t.}$, for $\beta \geq 1$, and let $L:=a_{\mathbf{n}}\boldsymbol{\ell }_{1}^{n_{1}}\cdots \boldsymbol{\ell }_{k}^{n_{k}}+\mathrm{h.o.t.}$, $a_{\mathbf{n}} \neq 0$, $\mathbf{n}:=(n_{1},\ldots ,n_{k})$, be defined as in the Main Theorem.
						
						By statement 1 of the Main Theorem, it follows that there exists $c\in \mathbb{R}$ and $\varphi \in \mathcal{L}_{k}^{0}$, such that $\varphi \circ f\circ \varphi ^{-1}=f_{c}$, where $f_{c}=\mathrm{id}+z^{\beta }L+c\mathrm{Res}(f)$. \\
						
						Note that:
						\begin{align}
							\frac{1}{f_{c}-\mathrm{id}} & = \frac{1}{z^{\beta }L+c\mathrm{Res}(f)} \nonumber \\
							& = \frac{1}{z^{\beta }L}\cdot \frac{1}{1+\frac{c\mathrm{Res}(f)}{z^{\beta }L}} \nonumber \\
							& = \frac{1}{z^{\beta }L}\cdot \bigg( 1+\sum _{i\geq 1}\Big(\frac{c\mathrm{Res}(f)}{z^{\beta }L}\Big)^{i}\bigg) \nonumber \\
							& = \frac{1}{z^{\beta }L} + \big( \frac{c}{a_{\mathbf{n}}^{2}}z^{-1}\boldsymbol{\ell }_{1}\cdots \boldsymbol{\ell }_{k} + \mathrm{h.o.t.} \big) . \label{ResEq1} 
						\end{align}
						From \eqref{ResEq1} we get:
						\begin{align}
							c & = \left[ \frac{a_{\mathbf{n}}^{2}}{f_{c}-\mathrm{id}}\right] _{-1,\mathbf{1}_{k}}- \left[ \frac{a_{\mathbf{n}}^{2}}{z^{\beta }L}\right] _{-1,\mathbf{1}_{k}} . \label{EqResProofCeoffic}
						\end{align}
						For every $f\in \mathcal{L}_{k}^{0}$, by Lemma~\ref{Lemma1}, it follows that:
						\begin{align}
							\left[ \frac{1}{f-\mathrm{id}}\right] _{-1,\mathbf{1}_{k}} & = \left[ \int \frac{dz}{f-\mathrm{id}}\right] _{\mathbf{0}_{k+1},-1} \label{EqResProof1}
						\end{align}
						and, in particular, for its normal form $f_{c}$,
						\begin{align}
							\left[ \frac{1}{f_{c}-\mathrm{id}}\right] _{-1,\mathbf{1}_{k}} & = \left[ \int \frac{dz}{f_{c}-\mathrm{id}}\right] _{\mathbf{0}_{k+1},-1} . \label{EqResProof2}
						\end{align}
						We prove that, for every $f\in \mathcal{L}_{k}^{0}$ and its normal form $f_{c}$ given by \eqref{DefnOfFc},
						\begin{align}\label{ResEq2}
							\left[ \int \frac{dz}{f_{c}-\mathrm{id}}\right] _{\mathbf{0}_{k+1},-1} & = \left[ \int \frac{dz}{f-\mathrm{id}}\right] _{\mathbf{0}_{k+1},-1} .
						\end{align}
						Put $g:=f-\mathrm{id}$. Let us use the following notation
						\begin{align*}
							& \left. \int ^{\varphi ^{-1}}h(s)ds:=\int h(s)ds \; \right| _{s=\varphi ^{-1}}, \quad h\in \mathcal{L}_{k} , \, \varphi \in \mathcal{L}_{k}^{0} .
						\end{align*}
						From $\varphi \circ f\circ \varphi ^{-1}=f_{0}$, it follows that $\varphi \circ f=f_{c}\circ \varphi $. Then, by the change of variable of the integration $z=\varphi (s)$ and the Taylor Theorem (see \cite[Proposition 3.3]{prrs21}), it follows that:
						\begin{align*}
							\int \frac{dz}{f_{c}-\mathrm{id}} & = \int ^{\varphi ^{-1}} \frac{\varphi '(s)ds}{(f_{c}\circ \varphi )(s)-\varphi (s)} \\
							& = \int ^{\varphi ^{-1}} \frac{\varphi '(s)ds}{\varphi (f(s))-\varphi (s)} \\
							& = \int ^{\varphi ^{-1}} \frac{\varphi '(s)ds}{\sum _{i\geq 1}\frac{\varphi ^{(i)}(s)}{i!}(g(s))^{i}} \\
							& = \int ^{\varphi ^{-1}} \frac{\varphi '(s)ds}{\varphi '(s)g(s)\cdot \Big( 1+\sum _{i\geq 2}\frac{\varphi ^{(i)}(s)}{i!\varphi '(s)}(g(s))^{i-1}\Big) } \\
							& = \int ^{\varphi ^{-1}} \frac{ds}{g(s)}\cdot \Big( 1-\mathrm{Lt}\Big( \frac{\varphi ''(s)}{2\varphi '(s)}g(s)\Big) +\mathrm{h.o.t.}\Big) \\
							& = \int ^{\varphi ^{-1}} \frac{ds}{g(s)} + \int ^{\varphi ^{-1}}\Big( - \frac{1}{2}\mathrm{Lt}\Big( \frac{\varphi ''(s)}{\varphi '(s)}\Big) +\mathrm{h.o.t.}\Big) ds .
						\end{align*}
						Put $\varepsilon :=\varphi -\mathrm{id}$. Note that:
						\begin{align*}
							\frac{\varphi ''(s)}{\varphi '(s)} & = \frac{d}{ds}(\log (1+\varepsilon '(s))).
						\end{align*}
						Since $\mathrm{ord}(\varepsilon ) > (1,\mathbf{0}_{k})$, it follows that $\mathrm{ord}(\log (1+\varepsilon '(s))) > \mathbf{0}_{k+1}$. Therefore,
						\begin{align*}
							\mathrm{ord}\Big( \frac{d}{ds}(\log (1+\varepsilon '(s)))\Big) & > (-1,\mathbf{1}_{k}) .
						\end{align*}
						By Lemma~\ref{Lemma1}, it follows that
						\begin{align*}
							& \int ^{\varphi ^{-1}}\Big( - \frac{1}{2}\mathrm{Lt}\Big( \frac{\varphi ''(s)}{\varphi '(s)}\Big) +\mathrm{h.o.t.}\Big) ds 
						\end{align*}
						is an element of $\mathcal{L}_{k}$. Thus, we proved:
						\begin{align}\label{ResEq4}
							& \left[ \int \frac{dz}{f_{c}-\mathrm{id}}\right] _{\mathbf{0}_{k+1},-1} = \left[ \int ^{\varphi ^{-1}} \frac{ds}{g(s)}\right] _{\mathbf{0}_{k+1},-1} = \left[ \int ^{\varphi ^{-1}} \frac{ds}{(f-\mathrm{id})(s)}\right] _{\mathbf{0}_{k+1},-1} .
						\end{align}
						Put $h(s):= \int \frac{ds}{(f-\mathrm{id})(s)}$. Note that $h=h(z)=\int ^{z} \frac{ds}{(f-\mathrm{id})(s)}=\int \frac{dz}{f-\mathrm{id}}$. Now, we get:
						\begin{align}
							& \int ^{\varphi ^{-1}} \frac{ds}{(f-\mathrm{id})(s)} - \int ^{z} \frac{ds}{(f-\mathrm{id})(s)} = h\left( \varphi ^{-1} \right) - h = \sum _{i\geq 1}\frac{h^{(i)}}{i!}(\varphi ^{-1}-\mathrm{id})^{i} . \label{ResidualEqProof}
						\end{align}
						Note that $\left[ h^{(i)}\right] _{\mathbf{0}_{k+1},-1}=0$ and $(\varphi ^{-1}-\mathrm{id})^{i}\in \mathcal{L}_{k}$, for $i\geq 1$. Therefore, by \eqref{ResidualEqProof}, it follows that:
						\begin{align*}
							& \left[ \int ^{\varphi ^{-1}} \frac{ds}{(f-\mathrm{id})(s)}\right] _{\mathbf{0}_{k+1},-1} - \left[ \int ^{z}\frac{ds}{(f-\mathrm{id})(s)}\right] _{\mathbf{0}_{k+1},-1}=0, 
						\end{align*}
						which implies that:
						\begin{align}
							& \left[ \int ^{\varphi ^{-1}} \frac{ds}{(f-\mathrm{id})(s)}\right] _{\mathbf{0}_{k+1},-1} = \left[ \int ^{z} \frac{ds}{(f-\mathrm{id})(s)}\right] _{\mathbf{0}_{k+1},-1} = \left[ \int \frac{dz}{f-\mathrm{id}}\right] _{\mathbf{0}_{k+1},-1} . \label{ResEq3}
						\end{align}
						Now, \eqref{ResEq2} follows from \eqref{ResEq4} and \eqref{ResEq3}. By \eqref{EqResProof1}, \eqref{EqResProof2} and \eqref{ResEq2}, it follows that
						\begin{align*}
							\left[ \frac{1}{f-\mathrm{id}}\right] _{-1,\mathbf{1}_{k}} & = \left[ \frac{1}{f_{c}-\mathrm{id}}\right] _{-1,\mathbf{1}_{k}} .
						\end{align*}
						Therefore, \eqref{EqRes} follows from \eqref{EqResProofCeoffic}. If $\beta >1$, note that the second term in \eqref{EqResProofCeoffic} vanishes, which implies \eqref{EqRes2}.
						
						Since $c$ is explicitely given by formula \eqref{EqRes}, it is unique.
					\end{proof}
					
					\begin{proof}[Proof of statement 3 of the Main Theorem]
						The minimality of the normal form $f_{c}$ in $\mathcal{L}_{k}^{0}$ follows directly from the uniqueness of $c\in \mathbb{R} $ and Proposition~\ref{LeadingBlockLemma}.
					\end{proof}

				\section{Proofs of Remark~\ref{Remark2} and Corollary~\ref{Remark3}}\label{sec:proofRemark}
				
					\begin{proof}[Proof of Remark~\ref{Remark2}]
						By a simple calculation it can be shown that the time-one map of the vector field $X_{c}$ is an element of $\mathcal{L}_{k}^{0}$ that has the initial part equal to $f_{c}$. Therefore, by the Main Theorem, the time-one map of the vector field $X_{c}$ can be reduced to $f_{c}$ by a change of variables form $\mathcal{L}_{k}^{0}$. Since $\mathcal{L}_{k}^{0}$ is a group, it follows that $f$ can be reduced to the time-one map of the vector field $X_{c}$ by a change of variables from $\mathcal{L}_{k}^{0}$.
					\end{proof}
					
					\begin{proof}[Proof of Corollary~\ref{Remark3}]
						Let $k\in \mathbb{N}$ be minimal such that $f\in \mathcal{L}_{k}^{0}$. Let $f:=\mathrm{id}+z^{\beta }L+\mathrm{h.o.t.}$, for $\beta \geq 1$, where $L$ is as defined in the Main Theorem. By the Main Theorem, for $m\geq k+1$, there exists the unique $c\in \mathbb{R}$ such that $f$ can be reduced to $f_{c}=\mathrm{id}+zL+c\mathrm{Res}(f)$, where $\mathrm{Res}(f)$ is the residual monomial of $f$ in the differential algebra $\mathcal{L}_{m}$. By statement 2 of the Main Theorem, it follows that:
						\begin{align}
							c & = \left[ \frac{a_{\mathbf{n}}^{2}}{f-\mathrm{id}} \right] _{-1,\mathbf{1}_{m}} - \left[ \frac{a_{\mathbf{n}}^{2}}{zL} \right] _{-1,\mathbf{1}_{m}} , \label{ResidualCoeffEq1}
						\end{align}
						if $\beta =1$, and
						\begin{align}
							c & =\left[ \frac{a_{\mathbf{n}}^{2}}{f-\mathrm{id}}\right] _{-1,\mathbf{1}_{m}} , \label{ResidualCoeffEq2}
						\end{align}
						if $\beta >1$. Since $f\in \mathcal{L}_{k}^{0}\subseteq \mathcal{L}_{k}^{\infty }$ and $\mathcal{L}_{k}^{\infty }$ is an algebra and a field, it follows that $ \frac{a_{\mathbf{n}}^{2}}{f-\mathrm{id}}- \frac{a_{\mathbf{n}}^{2}}{zL}$ if $\beta =1$, and $ \frac{a_{\mathbf{n}}^{2}}{f-\mathrm{id}}$ if $\beta >1$, are elements of $\mathcal{L}_{k}^{\infty }$. Since $m\geq k+1$, by \eqref{ResidualCoeffEq1} and \eqref{ResidualCoeffEq2}, it follows that $c=0$ in both cases. Therefore, $f_{c}=\mathrm{id}+zL$, for every $m\geq k+1$. Now, we put $f_{0}:=\mathrm{id}+zL$. It therefore suffices to take a normalization in $\mathcal{L}_{k+1}^{0}$ to eliminate the residual term.
						
						The minimality of $f_{0}$ in $\mathfrak L$ now follows directly by Proposition~\ref{LeadingBlockLemma}.
					\end{proof}

	\appendix
	
		\section{Proofs of auxiliary results from Section~\ref{sec:proofA}}
		
			\begin{proof}[Proof of Lemma~\ref{Lema2}]
				Since $\varepsilon \in \mathcal{L}_{k}^{\delta }$ and $\varphi \in \mathcal{L}_{k}^{\alpha }$, it follows that $\mathrm{ord}(\varepsilon ), \mathrm{ord}(\varphi )>(1,\mathbf{0}_{k})$. By the Neumann Lemma (see \cite{Neumann49}) it is easy to see that the series on the right-hand side of \eqref{LemmaDefinitionS} converges in the \emph{weak topology} defined in \cite[Subsection 2.3]{prrs21}. Therefore, the operator $\mathcal{S}$ is well-defined.
				
				Let $\varepsilon _{1}, \varepsilon _{2}\in \mathcal{L}_{k}^{\delta }$ be arbitrary, such that $\varepsilon _{1} \neq \varepsilon _{2}$. Then:
				\begin{align*}
					\mathcal{S}(\varepsilon _{1}) - \mathcal{S}(\varepsilon _{2}) & = \sum _{i\geq 2}\Big( \varphi ^{(i)}\cdot (\varepsilon _{1}^{i}- \varepsilon _{2}^{i}) \Big) \\
					& = \sum _{i\geq 2} \Big( \varphi ^{(i)}\cdot (\varepsilon _{1}-\varepsilon _{2})\cdot \Big( \sum _{j=0}^{i-1}\varepsilon _{1}^{j}\varepsilon _{2}^{i-1-j}\Big) \Big) .
				\end{align*}
				Now, $\mathrm{ord}(\mathcal{S}(\varepsilon _{1}) - \mathcal{S}(\varepsilon _{2}) ) \geq \mathrm{ord}(\varepsilon _{1}-\varepsilon _{2})+\delta +\alpha -2$, which implies that $\mathcal{S}$ is a $\frac{1}{2^{\delta + \alpha -2}}$-Lipschitz operator.
				
				If $\varepsilon \in \mathcal{L}_{k}^{\delta }$ such that $\mathrm{ord}_{z}(\varepsilon )=\delta $, then $\mathrm{ord}(\mathcal{S}(\varepsilon ) - \mathcal{S}(0) ) = \mathrm{ord}(\varepsilon -0)+\delta +\alpha -2$. Thus, $\frac{1}{2^{\delta + \alpha -2}}$ is the minimal Lipschitz coefficient of $\mathcal{S}$.
			\end{proof}
		
			Lemma~\ref{Lema1} is needed in the proof of Lemma~\ref{Lem5}.
		
			\begin{lem}\label{Lema1}
				Let $k\in \mathbb{N}_{\geq 1}$, $1\leq m\leq k$ and $R_{i} \in \mathcal{B}_{m} \setminus \left\lbrace 0\right\rbrace \subseteq \mathcal{L}_{k}^{\infty }$, such that $\mathrm{ord}(R_{i}) \geq \mathbf{0}_{k+1}$, for every $i\geq 2$. Let $\mathcal{S}:\mathcal{B}_{m}^{+} \to \mathcal{B}_{m}^{+}$ be defined as:
				\begin{align*}
					\mathcal{S}(Q) & :=\sum _{i\geq 2}R_{i}\cdot Q^{i} , 
				\end{align*}
				for $Q \in \mathcal{B}_{m}^{+}\subseteq \mathcal{L}_{k}$. The map $\mathcal{S}$ is a $\frac{1}{2}$-contraction on the space $(\mathcal{B}_{m}^{+}, d_{m})$. 
			\end{lem}
			\begin{proof}
				Let $Q_{1}$ and $Q_{2}$ be distinct elements of $\mathcal{B}_{m}^{+}\subseteq \mathcal{L}_{k}$. We have:
				\begin{align*}
					\mathcal{S}(Q_{1}) - \mathcal{S}(Q_{2}) &= \sum _{i\geq 2}R_{i}\cdot (Q_{1}^{i} - Q_{2}^{i}) \\
					&= \sum _{i\geq 2}\Big( R_{i} \cdot \left( Q_{1} -Q_{2} \right) \cdot \sum_{j=0}^{i-1}Q_{1}^{j}\cdot Q_{2}^{i-1-j} \Big) .
				\end{align*}
				Using the facts that
				\begin{align*}
					& \mathrm{ord}(R_{i})\geq \mathbf{0}_{k+1} , \quad \mathrm{ord}_{\boldsymbol{\ell}_{m}}(Q_{1}) , \, \mathrm{ord}_{\boldsymbol{\ell}_{m}}(Q_{2}) \geq 1,
				\end{align*}
				and $i\geq 2$, we conclude that
				\begin{align*}
					\mathrm{ord}_{\boldsymbol{\ell}_{m}}(\mathcal{S}(Q_{1})-\mathcal{S}(Q_{2})) & \geq \mathrm{ord}_{\boldsymbol{\ell}_{m}}(Q_{1} - Q_{2}) +1.
				\end{align*}
				This implies that
				\begin{align*}
					d_{m} \big( \mathcal{S}(Q_{1}), \mathcal{S}(Q_{2})\big) & \leq \frac{1}{2} \cdot d_{m}(Q_{1} , Q_{2}) .
				\end{align*}
				Thus, $\mathcal{S}$ is a $\frac{1}{2}$-contraction.
			\end{proof}
		
			\begin{proof}[Proof of Lemma~\ref{Lem5}]
				Let $\mathcal{T}, \mathcal{S}:\mathcal{B}_{m}^{+} \to \mathcal{B}_{m}^{+}$, such that
				\begin{align*}
					\mathcal{T}(Q) & :=R_{1} \cdot Q 
				\end{align*}
				and
				\begin{align*}
					\mathcal{S}(Q) & :=M - R_{2} \cdot h(Q) ,
				\end{align*}
				for every $Q \in \mathcal{B}_{m}^{+}\subseteq \mathcal{L}_{k}$. The equation \eqref{Eq10} is equivalent to the equation $\mathcal{T}(Q)= \mathcal{S}(Q)$. Let $\displaystyle h:=\sum _{i\geq 2}H_{i}\cdot x^{i}$, where $H_{i}\in \mathcal{B}_{\geq m}^{+}\subseteq \mathcal{L}_{k}$. Since $\mathrm{ord}(R_{1})=\mathbf{0}_{k+1}$, it follows that
				\begin{align*}
					& \mathrm{ord}_{\boldsymbol{\ell}_{m}} (\mathcal{T}(Q) )= \mathrm{ord}_{\boldsymbol{\ell}_{m}} (Q) + \mathrm{ord}_{\boldsymbol{\ell}_{m}} (R_{1})= \mathrm{ord}_{\boldsymbol{\ell}_{m}} (Q) .
				\end{align*}
				Therefore, $\mathcal{T}$ is an isometry on the space $\left( \mathcal{B}_{m}^{+},d_{m}\right) $. By Lemma~\ref{Lema1}, it follows that $\mathcal{S}$ is a $\frac{1}{2}$-contraction on the space $\left( \mathcal{B}_{m}^{+},d_{m}\right) $. It is obvious that $\mathcal{T}$ is a surjection, since $\frac{1}{R_{1}}\in \mathcal{B}_{\geq m}^{+}\subseteq \mathcal{L}_{k}$ (due to $\mathrm{ord}(R_{1}) = \mathbf{0}_{k+1}$). By the fixed point theorem from Proposition~\ref{KorBanach}, there exists a unique $Q \in \mathcal{B}_{m}^{+}\subseteq \mathcal{L}_{k}$, such that $\mathcal{T}(Q)= \mathcal{S}(Q)$.
			\end{proof}

		\section{The proofs of auxiliary results from Section~\ref{sec:proofAb}}
			
			\begin{lem}\label{LemaInductiveStep}
				Let $K,P\in \mathcal{B}_{i}\subseteq \mathcal{L}_{k}^{\infty }$, $P\neq 0$, $1\leq i\leq k$, $k\in \mathbb{N}_{\geq 1}$, and let $n_{i}:=\mathrm{ord}_{\boldsymbol{\ell }_{i}}(P)$. If the order in $\boldsymbol{\ell }_{i}$ of each term in $K$ is strictly smaller than $2n_{i}+1$, then there exists $S\in \mathcal{B}_{i}\subseteq \mathcal{L}_{k}^{\infty }$ such that:
				\begin{align}
					\mathcal{P}_{<(\mathbf{0}_{i},2n_{i}+1)}\big( PD_{i}(S)-SD_{i}(P)\big) &= K . \label{LemaInductiveEquation}
				\end{align}
				Moreover, if we impose the condition that order in $\boldsymbol{\ell }_{i}$ of each term in $S$ is strictly smaller than $n_{i}$, then such $S$ is unique.
			\end{lem}
			
			\begin{proof}
				Let
				\begin{align*}
					P & = \boldsymbol{\ell }_{i}^{n_{i}}P_{i+1} + P_{i},
				\end{align*}
				where $P_{i+1}\in \mathcal{B}_{i+1}\subseteq \mathcal{L}_{k}^{\infty }$, and $P_{i}\in \mathcal{B}_{i}\subseteq \mathcal{L}_{k}^{\infty }$ such that $\mathrm{ord}_{\boldsymbol{\ell }_{i}}(P_{i}) \geq n_{i}+1$. Let $\mathcal{S},\mathcal{T}:\mathcal{B}_{i}\to \mathcal{B}_{i}$ be operators defined by:
				\begin{align*}
					\mathcal{S}(S) & := \mathcal{P}_{<(\mathbf{0}_{i},2n_{i}+1)}\big( K-P_{i}D_{i}(S)+SD_{i}(P_{i}) \big) , \\
					\mathcal{T}(S) & := \mathcal{P}_{<(\mathbf{0}_{i},2n_{i}+1)} \big( \boldsymbol{\ell }_{i}^{n_{i}}P_{i+1}D_{i}(S)-SD_{i}(\boldsymbol{\ell }_{i}^{n_{i}}P_{i+1}) \big) .
				\end{align*}
				Note that $\mathcal{S}$ is an affine $\frac{1}{2^{n_{i}+2}}$-contraction on the space $(\mathcal{B}_{i},d_{i})$ and $\mathcal{T}$ is linear. Furthermore,
				\begin{align}
					\mathrm{ord}_{\boldsymbol{\ell }_{i}} \big( \boldsymbol{\ell }_{i}^{n_{i}}P_{i+1}D_{i}(S)-SD_{i}(\boldsymbol{\ell }_{i}^{n_{i}}P_{i+1}) \big) & = \mathrm{ord}_{\boldsymbol{\ell }_{i}}(S)+n_{i}+1 , \label{LemaCaseB2Equati10}
				\end{align}
				if and only if $\boldsymbol{\ell }_{i}^{n_{i}}P_{i+1}D_{i}(S)-SD_{i}(\boldsymbol{\ell }_{i}^{n_{i}}P_{i+1}) \neq 0$. Now, by solving the linear differential equation, we get that $\boldsymbol{\ell }_{i}^{n_{i}}P_{i+1}D_{i}(S)-SD_{i}(\boldsymbol{\ell }_{i}^{n_{i}}P_{i+1}) =0$ if and only if $S=C\cdot \boldsymbol{\ell }_{i}^{n_{i}}P_{i+1}$, $C\in \mathbb{R}$. Therefore, \eqref{LemaCaseB2Equati10} holds if and only if $S\neq C\cdot \boldsymbol{\ell }_{i}^{n_{i}}P_{i+1}$, for each $C\in \mathbb{R}$.
				
				Let $\widetilde{\mathcal{B}}$ be the space of all $S\in \mathcal{B}_{i}\subseteq \mathcal{L}_{k}^{\infty }$ such that $S=0$, or every term in $S$ has order in $\boldsymbol{\ell }_{i}$ strictly smaller than $n_{i}$. By \eqref{LemaCaseB2Equati10}, the order in $\boldsymbol{\ell }_{i}$ of each term in $ \boldsymbol{\ell }_{i}^{n_{i}}P_{i+1}D_{i}(S)-SD_{i}(\boldsymbol{\ell }_{i}^{n_{i}}P_{i+1}) $, $S\in \widetilde{\mathcal{B}}$, is strictly smaller than $2n_{i}+1$. Consequently, it follows that:
				\begin{align}
					\mathcal{T}|_{\widetilde{\mathcal{B}}}(S) & = \boldsymbol{\ell }_{i}^{n_{i}}P_{i+1}D_{i}(S)-SD_{i}(\boldsymbol{\ell }_{i}^{n_{i}}P_{i+1}) , \quad S\in \widetilde{\mathcal{B}} . \label{LemaCaseB2Equati11}
				\end{align}
				Let $S_{1},S_{2}\in \widetilde{\mathcal{B}}$ such that $S_{1}\neq S_{2}$. Suppose that $\mathcal{T}|_{\widetilde{\mathcal{B}}}(S_{1}-S_{2}) =0$. Solving the linear differential equation as above, we get that $S_{1}-S_{2}=C\cdot \boldsymbol{\ell }_{i}^{n_{i}}P_{i+1}$, for some $C\in \mathbb{R}$. Since, $S_{1},S_{2}\in \widetilde{\mathcal{B}}$, we get $S_{1}=S_{2}$, which is a contradiction. Therefore, $\mathcal{T}|_{\widetilde{\mathcal{B}}}(S_{1}-S_{2}) \neq 0$. By \eqref{LemaCaseB2Equati10} and \eqref{LemaCaseB2Equati11}, it follows that:
				\begin{align*}
					\mathrm{ord}_{\boldsymbol{\ell }_{i}}(\mathcal{T}|_{\widetilde{\mathcal{B}}}(S_{1}-S_{2})) &= \mathrm{ord}_{\boldsymbol{\ell }_{i}}(S_{1}-S_{2})+n_{i}+1 .
				\end{align*}
				Now we get that $\mathcal{T}|_{\widetilde{\mathcal{B}}}:\widetilde{\mathcal{B}}\to \mathcal{B}_{i}$ is a linear $\frac{1}{2^{n_{i}+1}}$-homothety, with respect to the metric $d_{i}$. Now, equation \eqref{LemaInductiveEquation} is equivalent to the fixed point equation
				\begin{align*}
					& \mathcal{T}|_{\widetilde{\mathcal{B}}}(S)=\mathcal{S}|_{\widetilde{\mathcal{B}}}(S) , \quad S\in \widetilde{\mathcal{B}} .
				\end{align*}
				We now prove that $\mathcal{S}(\widetilde{\mathcal{B}})\subseteq \mathcal{T}(\widetilde{\mathcal{B}})$. Then we conclude by the fixed point theorem from Proposition~\ref{KorBanach}.
				
				Suppose that $M\in \mathcal{S}(\widetilde{\mathcal{B}})$. Then, by definition of $\mathcal{S}$, each term in $M$ is of order in $\boldsymbol{\ell }_{i}$ strictly smaller than $2n_{i}+1$. We solve the equation $\mathcal{T}|_{\widetilde{\mathcal{B}}}(S)=M$, i.e.
				\begin{align*}
					& \boldsymbol{\ell }_{i}^{n_{i}}P_{i+1}D_{i}(S)-SD_{i}(\boldsymbol{\ell }_{i}^{n_{i}}P_{i+1}) = M , \quad S\in \widetilde{\mathcal{B}} .
				\end{align*}
				This is a linear ordinary differential equation whose solutions are given by:
				\begin{align}
					S & = \exp \Big( \int \frac{D_{i}(\boldsymbol{\ell }_{i}^{n_{i}}P_{i+1})}{\boldsymbol{\ell }_{i}^{n_{i}}P_{i+1}} \frac{d\boldsymbol{\ell }_{i}}{\boldsymbol{\ell }_{i}^{2}}\Big) \bigg( C+\int \Big( \frac{M}{\boldsymbol{\ell }_{i}^{n_{i}}P_{i+1}}\cdot \exp \Big( -\int \frac{D_{i}(\boldsymbol{\ell }_{i}^{n_{i}}P_{i+1})}{\boldsymbol{\ell }_{i}^{n_{i}}P_{i+1}}\frac{d\boldsymbol{\ell }_{i}}{\boldsymbol{\ell }_{i}^{2}} \Big) \frac{d\boldsymbol{\ell }_{i}}{\boldsymbol{\ell }_{i}^{2}}\Big) \bigg) \nonumber \\
					& = \exp \big( \log (\boldsymbol{\ell }_{i}^{n_{i}}P_{i+1}) \big) \Big( C+\int \frac{M}{\boldsymbol{\ell }_{i}^{n_{i}}P_{i+1}}\cdot \big( \exp (\log (\boldsymbol{\ell }_{i}^{n_{i}}P_{i+1})) \big) ^{-1} \frac{d\boldsymbol{\ell }_{i}}{\boldsymbol{\ell }_{i}^{2}} \Big) \nonumber \\
					& = \boldsymbol{\ell }_{i}^{n_{i}}P_{i+1}\Big( C+\int \frac{M}{(\boldsymbol{\ell }_{i}^{n_{i}}P_{i+1})^{2}} \frac{d\boldsymbol{\ell }_{i}}{\boldsymbol{\ell }_{i}^{2}}\Big) , \quad C\in \mathbb{R} . \label{LemaEquationLinearDiff}
				\end{align}
				By Lemma~\ref{Lemma1}, it follows that $S\in \mathcal{B}_{i}\subseteq \mathcal{L}_{k}$, for each $C\in \mathbb{R}$. Moreover, taking $C=0$, we get $S\in \widetilde{\mathcal{B}}\subseteq \mathcal{B}_{i}$.
			\end{proof}
		
			The following Proposition~\ref{LemaCaseB2OperatorsTS} is needed in the proof of Proposition~\ref{PropCaseBTS}.
			
			\begin{prop}\label{LemaCaseB2OperatorsTS}
				Let $M,R,V \in \mathcal{B}_{\geq 1}^{+}\subseteq \mathcal{L}_{k}$, $k\in \mathbb{N}_{\geq 1}$, such that $R\neq 0$ and $\mathrm{ord}\, (V),\mathrm{ord}\, (M)>2\cdot \mathrm{ord} \, (R)+(0,\mathrm{1}_{k})$. Let
				\begin{align*}
					\widetilde{\mathcal{B}} & := \left\lbrace K\in \mathcal{B}_{\geq 1}^{+}\subseteq \mathcal{L}_{k} : \mathrm{ord} \, (K) > \mathrm{ord} \, (R)\right\rbrace ,
				\end{align*}
				and let $\mathcal{T},\mathcal{S}:\widetilde{\mathcal{B}}\to \mathcal{B}_{\geq 1}^{+}$ be the operators defined by:
				\begin{align}
					\mathcal{T}(S) &:=-(1+R)\cdot \log (1+R)\cdot D_{1}(S)+(1+S)\cdot \log(1+S)\cdot D_{1}(R)+S\cdot V , \nonumber \\
					\mathcal{S}(S) &:= M+\mathcal{K}_{1}(S)-\mathcal{C}_{1}(S) -(1+S)\cdot \log (1+S)\cdot D_{1}(V) , \quad S\in \widetilde{\mathcal{B}} , \label{CaseBIdent4}
				\end{align}
				where $\mathcal{C}_{1},\mathcal{K}_{1}:\mathcal{B}_{\geq 1}^{+}\to \mathcal{B}_{\geq 1}^{+}$ are $\frac{1}{2^{2+\mathrm{ord}_{\boldsymbol{\ell }_{1}}(R)}}$-contractions on the space $(\mathcal{B}_{\geq 1}^{+},d_{1})$. Then:
				\begin{enumerate}[1., font=\textup, nolistsep, leftmargin=0.6cm]
					\item operator $\mathcal{S}$ is a $\frac{1}{2^{2+\mathrm{ord}_{\boldsymbol{\ell }_{1}}(R)}}$-contraction, with respect to the metric $d_{1}$,
					\item operator $\mathcal{T}$ is a $\frac{1}{2^{1+\mathrm{ord}_{\boldsymbol{\ell }_{1}}(R)}}$-homothety, with respect to the metric $d_{1}$,
					\item $\mathcal{S}(\widetilde{\mathcal{B}})\subseteq \mathcal{T}(\widetilde{\mathcal{B}})$.
				\end{enumerate}
			\end{prop}
		
			Proposition~\ref{Lemma3} below gives a solution of a particular differential equation which we use in the proof of Proposition~\ref{LemaCaseB2OperatorsTS}.
			
			\begin{prop}[Solution of a differential equation in $\mathcal{B}_{\geq 1}^{+}$]\label{Lemma3}
				Let $K,M,N,T\in \mathcal{B}_{\geq 1}^{+}\subseteq \mathcal{L}_{k}$, $N\neq 0$, such that $\mathrm{ord}(T)\geq \mathrm{ord}(\frac{D_{1}(N)}{N})$, $\mathrm{ord}(K)>(0,\mathbf{1}_{k})$, and $\mathrm{ord}(M)-\mathrm{ord}(N) > (0,\mathbf{1}_{k})$. Let $h\in x^{2}\mathbb{R}\left[ \left[ x\right] \right] $ be a power series in the variable $x$, with real coefficients, such that $h(0)=h'(0)=0$. Then there exists a unique solution $S\in \mathcal{B}_{\geq 1}^{+}\subseteq \mathcal{L}_{k}$ satisfying $\mathrm{ord} \, (S) > \mathrm{ord} \, (N)$ of the differential equation:
				\begin{align}
					D_{1}(S) - \Big( \frac{D_{1}(N)}{N}+K\Big) S + T\cdot h(S) & = M . \label{PropEq8}
				\end{align}
			\end{prop}
			
			We first prove Lemmas~\ref{Lemma4} and~\ref{Lemma4Second} which are auxiliary lemmas in the proof of Proposition~\ref{Lemma3}.
			
			\begin{lem}[Solution of a differential equation in $\mathcal{B}_{\geq m}^{+}$]\label{Lemma4}
				Let $M,N,P,T\in \mathcal{B}_{m}^{+}\subseteq \mathcal{L}_{k}$, $N\neq 0$, such that $\mathrm{ord}(P) > (\mathbf{0}_{m},\mathbf{1}_{k-m+1})$, for $1\leq m\leq k$, $k\in \mathbb{N}_{\geq 1}$, and let $h\in x^{2}\mathcal{B}_{\geq m}^{+}\left[ \left[ x\right] \right] $ be a power series in the variable $x$, with coefficients in $\mathcal{B}_{\geq m}^{+}\subseteq \mathcal{L}_{k}$, such that $h(0)=h'(0)=0$. Suppose that
				\begin{align*}
					\mathrm{ord}(M)-\mathrm{ord}(N) & > (\mathbf{0}_{m},\mathbf{1}_{k-m+1}) .
				\end{align*}
				Then there exists a unique solution $S\in \mathcal{B}_{m}^{+}\subseteq \mathcal{L}_{k}$ satisfying $\mathrm{ord} \, (S)> \mathrm{ord} \, (N)$ of the differential equation:
				\begin{align}
					D_{m}(S) - \Big( \frac{D_{m}(N)}{N}+P\Big) \cdot S + T\cdot h(S) & = M . \label{LemaCaseB2Equation1}
				\end{align}
			\end{lem}
			
			\begin{proof}
				Let $\mathcal{T}_{m},\mathcal{S}_{m}:\mathcal{B}_{m}^{+}\to \mathcal{B}_{m}^{+}$ be operators defined by:
				\begin{align}
					\mathcal{T}_{m}(S) &:= D_{m}(S)-\Big( \frac{D_{m}(N)}{N}+P\Big) S , \nonumber \\
					\mathcal{S}_{m}(S) &:= M-T\cdot h(S) , \quad S\in \mathcal{B}_{m}^{+}\subseteq \mathcal{L}_{k} . \label{DefOfTmSm}
				\end{align}
				Now, the equation $\mathcal{T}_{m}(S)=\mathcal{S}_{m}(S)$ becomes equation \eqref{LemaCaseB2Equation1}, for $S\in \mathcal{B}_{m}^{+}\subseteq \mathcal{L}_{k}$. \\
				
				We prove that operators $\mathcal{T}_{m}$ and $\mathcal{S}_{m}$ satisfy the assumptions of Proposition~\ref{KorBanach}. \\
				For arbitrary $S_{1},S_{2}\in \mathcal{B}_{m}^{+}$, note that:
				\begin{align*}
					& \mathrm{ord}_{\boldsymbol{\ell }_{m}}(\mathcal{S}_{m}(S_{1})-\mathcal{S}_{m}(S_{2})) \\
					& \geq \mathrm{ord}_{\boldsymbol{\ell }_{m}}(S_{1}-S_{2}) + \mathrm{ord}_{\boldsymbol{\ell }_{m}}(T) + \min \left\lbrace \mathrm{ord}_{\boldsymbol{\ell }_{m}}(S_{1}),\mathrm{ord}_{\boldsymbol{\ell }_{m}}(S_{2})\right\rbrace \\
					& \geq \mathrm{ord}_{\boldsymbol{\ell }_{m}}(S_{1}-S_{2}) + 2.
				\end{align*}
				Therefore, $\mathcal{S}_{m}$ is a $\frac{1}{4}$-contraction on space $(\mathcal{B}_{m}^{+},d_{m})$. \\
				Suppose that $\mathcal{T}_{m}(S)=0$, i.e., $D_{m}(S)-\Big( \frac{D_{m}(N)}{N}+P\Big) \cdot S = 0$, for some $S\in \mathcal{B}_{m}^{+}\subseteq \mathcal{L}_{k}$. Solving the linear differential equation, we get:
				\begin{align}
					S & = C\exp \big( \log N+\int P\frac{d\boldsymbol{\ell }_{m}}{\boldsymbol{\ell }_{m}^{2}}\big) \nonumber \\
					& =C(N\cdot \exp (\int P\frac{d\boldsymbol{\ell }_{m}}{\boldsymbol{\ell }_{m}^{2}})) , \quad C\in \mathbb{R} . \nonumber 
				\end{align}
				Since $\mathrm{ord}(P) > (\mathbf{0}_{m},\mathbf{1}_{k-m+1})$, using supstitution $z:=\boldsymbol{\ell }_{m}$, by Lemma~\ref{Lemma1}, it follows that
				\begin{align}
					& \int P\frac{d\boldsymbol{\ell }_{m}}{\boldsymbol{\ell }_{m}^{2}} \in \mathcal{B}_{\geq m}^{+} . \nonumber
				\end{align}
				This implies that
				\begin{align}\label{LemmaEq2}
					\exp (\int P\frac{d\boldsymbol{\ell }_{m}}{\boldsymbol{\ell }_{m}^{2}}) & = 1+\mathrm{h.o.t.} 
				\end{align}
				We conclude that $S\in \ker (\mathcal{T}_{m})$ implies that $\mathrm{ord}(S)=\mathrm{ord}(N)$. To avoid terms in the kernel of $\mathcal{T}_{m}$, we restrict ourselves on the space
				\begin{align*}
					\widetilde{\mathcal{B}} & := \left\lbrace V\in \mathcal{B}_{m}^{+}\subseteq \mathcal{L}_{k} : \mathrm{ord} \, (V) > \mathrm{ord} \, (N)\right\rbrace .
				\end{align*}
				Then, for every $S_{1},S_{2}\in \widetilde{\mathcal{B}}$, $S_{1} \neq S_{2}$, it follows that $S_{1}-S_{2}\notin \ker (\mathcal{T}_{m})$. Now, by \eqref{DefOfTmSm}, it follows that
				\begin{align*}
					\mathrm{ord}_{\boldsymbol{\ell }_{m}}(\mathcal{T}_{m}(S_{1}-S_{2})) & =\mathrm{ord}_{\boldsymbol{\ell }_{m}}(S_{1}-S_{2}) +1 .
				\end{align*}
				By the linearity of operator $\mathcal{T}_{m}$, it follows that the restriction $\mathcal{T}_{m}|_{\widetilde{\mathcal{B}}}$ is a $\frac{1}{2}$-homothety on the space $\widetilde{\mathcal{B}}$, with respect to the metric $d_{m}$.
				
				It is left to prove that $\mathcal{S}_{m}(\widetilde{\mathcal{B}})\subseteq \mathcal{T}_{m}(\widetilde{\mathcal{B}})$. By definition \eqref{DefOfTmSm} of $\mathcal{S}_{m}$, note that $\mathcal{S}_{m}(\widetilde{\mathcal{B}})$ is contained in the space
				\begin{align*}
					\overline{\mathcal{B}} & := \left\lbrace K\in \mathcal{B}_{m}^{+} \subseteq \mathcal{L}_{k} : \mathrm{ord} \, (K)>\mathrm{ord} \, (N)+(\mathbf{0}_{m},\mathbf{1}_{k-m+1}) \right\rbrace .
				\end{align*}
				
				We prove that $\overline{\mathcal{B}}\subseteq \mathcal{T}_{m}(\widetilde{\mathcal{B}})$. Let $K\in \overline{\mathcal{B}}$ be arbitrary. We solve in $\widetilde{\mathcal{B}}$ the equation $\mathcal{T}_{m}(S)=K$, i.e.
				\begin{align}
					D_{m}(S)-\Big( \frac{D_{m}(N)}{N}+P\Big) \cdot S & = K . \nonumber 
				\end{align}
				Solving the linear differential equation above, we get:
				\begin{align}\label{LemmaEq1}
					S & = \exp \Big( \log N + \int P\frac{d\boldsymbol{\ell }_{m}}{\boldsymbol{\ell }_{m}^{2}}\Big) \cdot \Big( C+\int \Big( \frac{K}{\exp \big( \log N + \int P\frac{d\boldsymbol{\ell }_{m}}{\boldsymbol{\ell }_{m}^{2}}\big) }\Big) \frac{d\boldsymbol{\ell }_{m}}{\boldsymbol{\ell }_{m}^{2}} \Big) \nonumber \\
					& = \Big( N\cdot \exp \int P\frac{d\boldsymbol{\ell }_{m}}{\boldsymbol{\ell }_{m}^{2}}\Big) \cdot \Big( C+\int \Big( \frac{K}{N\cdot \exp \int P\frac{d\boldsymbol{\ell }_{m}}{\boldsymbol{\ell }_{m}^{2}}}\Big) \frac{d\boldsymbol{\ell }_{m}}{\boldsymbol{\ell }_{m}^{2}} \Big) , \quad C\in \mathbb{R} .
				\end{align}
				Since $S\in \widetilde{\mathcal{B}}$ we choose the solution with $C=0$. Now, from \eqref{LemmaEq2} and \eqref{LemmaEq1} we get:
				\begin{align}
					S & = (\mathrm{Lt}(N)+\mathrm{h.o.t.})\int \Big( \frac{K}{\mathrm{Lt}(N)+\mathrm{h.o.t.}}\Big) \frac{d\boldsymbol{\ell }_{m}}{\boldsymbol{\ell }_{m}^{2}} . \label{LemaEquationSolution}
				\end{align}
				Since $K\in \mathcal{B}_{m}^{+}\subseteq \mathcal{L}_{k}$, $\mathrm{ord}(K) > \mathrm{ord}(N)+(\mathbf{0}_{m},\mathbf{1}_{k-m+1})$, it follows that
				\begin{align}
					\mathrm{ord}\Big( \frac{K}{\mathrm{Lt}(N)+\mathrm{h.o.t.}}\Big) & > (\mathbf{0}_{m},\mathbf{1}_{k-m+1}) , \nonumber 
				\end{align}
				and, by Lemma~\ref{Lemma1},
				\begin{align*}
					& \int \Big( \frac{K}{\mathrm{Lt}(N)+\mathrm{h.o.t.}}\Big) \frac{d\boldsymbol{\ell }_{m}}{\boldsymbol{\ell }_{m}^{2}} \in \mathcal{B}_{\geq m}^{+}\subseteq \mathcal{L}_{k} . 
				\end{align*}
				By \eqref{LemaEquationSolution}, $\mathrm{ord}(S) > \mathrm{ord}(N)$, i.e., the solution $S$ given by \eqref{LemaEquationSolution} belongs to $\widetilde{\mathcal{B}}$.
				
				Thus, we proved that $\mathcal{S}_{m}(\widetilde{\mathcal{B}}) \subseteq \overline{\mathcal{B}}\subseteq \mathcal{T}_{m}(\widetilde{\mathcal{B}})$. \\
				Finally, by Proposition~\ref{KorBanach} applied to the operators $\mathcal{T}_{m}|_{\widetilde{\mathcal{B}}}:\widetilde{\mathcal{B}} \to \mathcal{B}_{\geq m}^{+}$ and $\mathcal{S}_{m}|_{\widetilde{\mathcal{B}}}:\widetilde{\mathcal{B}} \to \mathcal{B}_{\geq m}^{+}$, there exists a unique $S\in \widetilde{\mathcal{B}}$ such that $\mathcal{T}_{m}(S)=\mathcal{S}_{m}(S)$.
			\end{proof}
			
			\begin{lem}[Solution of a differential equation in $\mathcal{B}_{m}^{+}$]\label{Lemma4Second}
				Let $N\in \mathcal{B}_{\geq m+1}^{+}\subseteq \mathcal{L}_{k}$, $N\neq 0$, $M,P,T\in \mathcal{B}_{m}^{+}\subseteq \mathcal{L}_{k}$ such that $\mathrm{ord}(P) > (\mathbf{0}_{m},\mathbf{1}_{k-m+1})$ and $\mathrm{ord}_{\boldsymbol{\ell }_{m}}(M)\geq 2$, for $1\leq m\leq k$, $k\in \mathbb{N}_{\geq 1}$. Let $h\in x^{2}\mathcal{B}_{\geq m}^{+}\left[ \left[ x\right] \right] $ be a power series in the variable $x$, with coefficients in $\mathcal{B}_{\geq m}^{+}\subseteq \mathcal{L}_{k}$, such that $h(0)=h'(0)=0$. There exists a unique solution $S\in \mathcal{B}_{m}^{+}\subseteq \mathcal{L}_{k}$ of the differential equation:
				\begin{align*}
					D_{m}(S) - \Big( \frac{D_{m}(N)}{N}+P\Big) \cdot S + T\cdot h(S) & = M .
				\end{align*}
			\end{lem}
			
			The proof of Lemma~\ref{Lemma4Second} is similar as the proof of Lemma~\ref{Lemma4} and we omit it.
			
			\begin{proof}[Proof of Proposition~\ref{Lemma3}]
				We first transform equation \eqref{PropEq8} to an equivalent \emph{simpler} differential equation. Let
				\begin{align*}
					N & = N_{k}+\cdots + N_{1} ,
				\end{align*}
				where $N_{i}\in \mathcal{B}_{i}^{+}\subseteq \mathcal{L}_{k}$, $1\leq i\leq k$. Let $1\leq m\leq k$ be the biggest $m$ such that $N_{m}\neq 0$. Note that $N=N_{m}+\cdots +N_{1}$ and
				\begin{align}
					\frac{D_{1}(N)}{N} & = \frac{D_{1}(N)}{N_{m}\Big( 1+\frac{N_{m-1}}{N_{m}}+\cdots + \frac{N_{1}}{N_{m}}\Big) } \nonumber \\
					& = \left( \sum _{i=1}^{m}\frac{D_{1}(N_{i})}{N_{m}} \right) \cdot \Bigg( \sum _{i\geq 0}(-1)^{i}\Big( \frac{N_{m-1}}{N_{m}}+\cdots + \frac{N_{1}}{N_{m}}\Big) ^{i} \Bigg) \nonumber \\
					& = \Bigg( \frac{D_{1}(N_{m})}{N_{m}}+\sum _{i=1}^{m-1}\frac{D_{1}(N_{i})}{N_{m}} \Bigg) \cdot \Bigg( 1+\sum _{i\geq 1}(-1)^{i}\Big( \frac{N_{m-1}}{N_{m}}+\cdots + \frac{N_{1}}{N_{m}}\Big) ^{i} \, \Bigg) . \label{PropCaseB2Identit}
				\end{align}
				Note that:
				\begin{align}
					& \mathrm{ord}\Big( \frac{D_{1}(N_{i})}{N_{m}}\Big) > (0,\mathbf{1}_{k}) , \nonumber \\
					& \mathrm{ord} \left( \frac{N_{i}}{N_{m}}\right) \geq (\mathbf{0}_{i},1,v_{i+1},\ldots ,v_{k}) , \label{PropCaseB2Eq1}
				\end{align}
				where $(v_{i+1},\ldots ,v_{k}) \in \mathbb{Z}^{k-i}$, for $1\leq i\leq m-1$. Since, in addition, $\mathrm{ord}\Big( \frac{D_{1}(N_{m})}{N_{m}} \Big) =(0,\mathbf{1}_{m},\mathbf{0}_{k-m})$, we get that:
				\begin{align}
					\mathrm{ord} \left( \frac{D_{1}(N_{m})}{N_{m}}\cdot \sum _{i\geq 1}(-1)^{i}\Big( \frac{N_{m-1}}{N_{m}}+\cdots + \frac{N_{1}}{N_{m}}\Big) ^{i}\right) & > (0,\mathbf{1}_{k}) . \label{PropCaseB2Eq2}
				\end{align}
				By \eqref{PropCaseB2Identit}, \eqref{PropCaseB2Eq1} and \eqref{PropCaseB2Eq2}, we get that:
				\begin{align}
					\mathrm{ord}\Big( \frac{D_{1}(N)}{N} - \frac{D_{1}(N_{m})}{N_{m}}\Big) & > (0, \mathbf{1}_{k}) . \label{PropCaseB2OrderInequal}
				\end{align}
				Put $P:=K+\frac{D_{1}(N)}{N} - \frac{D_{1}(N_{m})}{N_{m}}$. By \eqref{PropCaseB2OrderInequal} and since $\mathrm{ord}\, (K) > (0,\mathbf{1}_{k})$ (by assumption), it follows that $P\in \mathcal{B}_{\geq 1}^{+}\subseteq \mathcal{L}_{k}$ and $\mathrm{ord}\, (P) > (0,\mathbf{1}_{k})$. Note that it is also possible that $P=0$. In this notation, equation \eqref{PropEq8} is equivalent to the equation:
				\begin{align}
					D_{1}(S) - \Big( \frac{D_{1}(N_{m})}{N_{m}}+P\Big) \cdot S + T\cdot h(S) & = M . \label{PropCaseB2EqMain}
				\end{align}
				
				\noindent \emph{Proof of existence of a solution}: Now we prove existence of a solution $S\in \mathcal{B}_{\geq 1}^{+}\subseteq \mathcal{L}_{k}$ of equation \eqref{PropCaseB2EqMain}. Moreover, we prove that $S$ satisfies $\mathrm{ord} \, (S) > \mathrm{ord} \, (N)$.
				
				Note that $S:=0$ is a solution of \eqref{PropCaseB2EqMain} if $M=0$. Moreover, $S$ is the unique solution that satisfies $\mathrm{ord} \, (S) > \mathrm{ord} \, (N)$. 
				
				Suppose that $M\neq 0$. Since $M\neq 0$, $\mathrm{ord}(M)-\mathrm{ord}(N)>(0,\mathbf{1}_{k})$ and $N\in \mathcal{B}_{\geq 1}^{+}\subseteq \mathcal{L}_{k}$ (so $\mathrm{ord}\, (N) > \mathbf{0}_{k+1}$), we have the following decomposition of $M$:
				\begin{align}
					M &= \boldsymbol{\ell }_{1}\cdots \boldsymbol{\ell }_{m}M_{m}+\boldsymbol{\ell }_{1}\cdots \boldsymbol{\ell }_{m-1}M_{m-1}+\cdots + \boldsymbol{\ell }_{1}M_{1}, \label{CaseB3DecompM}
				\end{align}
				where $M_{i}\in \mathcal{B}_{i}^{+}\subseteq \mathcal{L}_{k}$, $1\leq i\leq m-1$, and $M_{m}\in \mathcal{B}_{\geq m}^{+}\subseteq \mathcal{L}_{k}$, $\mathrm{ord}(\boldsymbol{\ell }_{m}M_{m}) > \mathrm{ord}(N_{m})+(\mathbf{0}_{m},\mathbf{1}_{k-m+1})$.
				
				By \eqref{PropCaseB2OrderInequal}, and since $\mathrm{ord}(K)>(0,\mathbf{1}_{k})$ by assumption, it follows that $P\in \mathcal{B}_{\geq 1}^{+}$ and $\mathrm{ord}(P)>(0,\mathbf{1}_{k})$. Now, decompose:
				\begin{align}
					P &= \boldsymbol{\ell }_{1}\cdots \boldsymbol{\ell }_{m}P_{m} + \boldsymbol{\ell }_{1}\cdots \boldsymbol{\ell }_{m-1}P_{m-1} + \cdots + \boldsymbol{\ell }_{1}P_{1} , \label{CaseB3DecompP}
				\end{align}
				where $P_{i}\in \mathcal{B}_{i}^{+}\subseteq \mathcal{L}_{k}$, $1\leq i\leq m-1$, and $P_{m}\in \mathcal{B}_{\geq m}^{+}\subseteq \mathcal{L}_{k}$, $\mathrm{ord}(P_{m})>(\mathbf{0}_{m+1},\mathbf{1}_{k-m})$. If $P=0$, we simply put $P_{i}:=0$, for each $1\leq i\leq m$.
				
				By \eqref{PropCaseB2Identit}, it follows that $\mathrm{ord}(\frac{D_{1}(N)}{N})=\mathrm{ord}(\frac{D_{1}(N_{m})}{N_{m}})$. Since
				\begin{align*}
					& \mathrm{ord}(T)\geq \mathrm{ord}\Big( \frac{D_{1}(N)}{N}\Big) =\mathrm{ord}\Big( \frac{D_{1}(N_{m})}{N_{m}} \Big) =(0,\mathbf{1}_{m},\mathbf{0}_{k-m}) ,
				\end{align*}
				we have the following decomposition of $T$:
				\begin{align}
					T &= \boldsymbol{\ell }_{1}\cdots \boldsymbol{\ell }_{m}T_{m}+\boldsymbol{\ell }_{1}\cdots \boldsymbol{\ell }_{m-1}T_{m-1}+\cdots + \boldsymbol{\ell }_{1}T_{1}, \label{CaseB3DecompT}
				\end{align}
				where $T_{i}\in \mathcal{B}_{i}^{+}\subseteq \mathcal{L}_{k}$, $1\leq i\leq m-1$, and $T_{m}\in \mathcal{B}_{m}\subseteq \mathcal{L}_{k}^{\infty }$, $\mathrm{ord}(T_{m})\geq \mathbf{0}_{k+1}$. If $T=0$, we simply put $T_{i}:=0$, for each $1\leq i\leq m$. \\
				
				The proof of existence of a solution of equation \eqref{PropCaseB2EqMain} is inductive. In the $i$-th step ($1\leq i\leq m$), instead of equation \eqref{PropCaseB2EqMain}, we consider the equation:
				{\small \begin{align*}
						& D_{1}(S_{m}+\cdots +S_{m-i+1}) \\
						& - \Big( \frac{D_{1}(N_{m})}{N_{m}}+\boldsymbol{\ell }_{1}\cdots \boldsymbol{\ell }_{m}P_{m}+\cdots +\boldsymbol{\ell }_{1}\cdots \boldsymbol{\ell }_{m-i+1}P_{m-i+1}\Big) \cdot (S_{m}+\cdots +S_{m-i+1}) \\
						& + (\boldsymbol{\ell }_{1}\cdots \boldsymbol{\ell }_{m}T_{m}+\cdots \boldsymbol{\ell }_{1}\cdots \boldsymbol{\ell }_{m-i+1}T_{m-i+1})\cdot h(S_{m}+\cdots +S_{m-i+1}) \\
						& = \boldsymbol{\ell }_{1}\cdots \boldsymbol{\ell }_{m}M_{m}+\cdots +\boldsymbol{\ell }_{1}\cdots \boldsymbol{\ell }_{m-i+1}M_{m-i+1} ,  
					\end{align*}
				}with appropriate partial sums from decompositions \eqref{CaseB3DecompM}, \eqref{CaseB3DecompP}, and \eqref{CaseB3DecompT} instead of whole $M,P,T$, where $S_{i}$ is an unknown \emph{variable} and $S_{m},\ldots ,S_{i-1}$ are obtained in the previous steps. For the proof of existence a solution $S_{m}\in \mathcal{B}_{\geq m}^{+}\subseteq \mathcal{L}_{k}$ in the first step, we use Lemma~\ref{Lemma4}, and for the proof of existence of a solution $S_{m-i+1}\in \mathcal{B}_{m-i+1}^{+}\subseteq \mathcal{L}_{k}$ in steps $2\leq i\leq m$, we use Lemma~\ref{Lemma4Second}. Finally, we put $S:=S_{m}+\cdots +S_{1}$, $S\in \mathcal{B}_{\geq 1}^{+}\subseteq \mathcal{L}_{k}$. \\
				
				\noindent \emph{Step} $1.$ Consider the equation:
				\begin{align}
					D_{1}(S) - \Big( \frac{D_{1}(N_{m})}{N_{m}}+\boldsymbol{\ell }_{1}\cdots \boldsymbol{\ell }_{m}P_{m}\Big) \cdot S + \boldsymbol{\ell }_{1}\cdots \boldsymbol{\ell }_{m}T_{m}\cdot h(S) & = \boldsymbol{\ell }_{1}\cdots \boldsymbol{\ell }_{m}M_{m} . \label{PropCaseBEq10}
				\end{align}
				where $S\in \mathcal{B}_{\geq m}^{+}\subseteq \mathcal{L}_{k}$. By \eqref{EqDer2}, it follows that $\frac{D_{1}(N_{m})}{N_{m}}=\boldsymbol{\ell }_{1}\cdots \boldsymbol{\ell }_{m-1}\frac{D_{m}(N_{m})}{N_{m}}$ and $D_{1}(S)=\boldsymbol{\ell }_{1}\cdots \boldsymbol{\ell }_{m-1}D_{m}(S)$. Dividing by $\boldsymbol{\ell }_{1}\cdots \boldsymbol{\ell }_{m-1}$, we get the equivalent equation:
				\begin{align}\label{PropCaseBEq9}
					D_{m}(S) - \Big( \frac{D_{m}(N_{m})}{N_{m}}+\boldsymbol{\ell }_{m}P_{m}\Big) \cdot S + \boldsymbol{\ell }_{m}T_{m}\cdot h(S) & = \boldsymbol{\ell }_{m}M_{m} .
				\end{align}
				Now, just for the purpose of applying Lemma~\ref{Lemma4}, put $N:=N_{m}$, $P:=\boldsymbol{\ell }_{m}P_{m}$, $T:=\boldsymbol{\ell }_{m}T_{m}$, and $M:=\boldsymbol{\ell }_{m}M_{m}$. Since $N\in \mathcal{B}_{m}^{+}\subseteq \mathcal{L}_{k}$, $\mathrm{ord} \, (P) > (\mathbf{0}_{m},\mathbf{1}_{k-m+1})$, $T\in \mathcal{B}_{m}^{+}\subseteq \mathcal{L}_{k}$, and $\mathrm{ord} \, (M) > \mathrm{ord} \, (N) + (\mathbf{0}_{m},\mathbf{1}_{k-m+1})$, by Lemma~\ref{Lemma4} it follows that there exists a unique solution $S:=S_{m}\in \mathcal{B}_{\geq m}^{+}\subseteq \mathcal{L}_{k}$ of equation \eqref{PropCaseBEq9}, i.e., equation \eqref{PropCaseBEq10}, satisfying $\mathrm{ord}(S_{m})>\mathrm{ord}(N_{m})$. \\
				
				\noindent \emph{Step} $2.$ Let $S_{m}\in \mathcal{B}_{\geq m}^{+}\subseteq \mathcal{L}_{k}$ be the solution of equation \eqref{PropCaseBEq10} obtained in \emph{Step} $1$. Now, consider the equation (in the variable $S$):
				\begin{align}
					& D_{1}(S_{m}+S) - \Big( \frac{D_{1}(N_{m})}{N_{m}}+\boldsymbol{\ell }_{1}\cdots \boldsymbol{\ell }_{m}P_{m}+\boldsymbol{\ell }_{1}\cdots \boldsymbol{\ell }_{m-1}P_{m-1}\Big) \cdot (S_{m}+S) \nonumber \\
					& + (\boldsymbol{\ell }_{1}\cdots \boldsymbol{\ell }_{m}T_{m}+\boldsymbol{\ell }_{1}\cdots \boldsymbol{\ell }_{m-1}T_{m-1})\cdot h(S_{m}+S) = \boldsymbol{\ell }_{1}\cdots \boldsymbol{\ell }_{m}M_{m}+\boldsymbol{\ell }_{1}\cdots \boldsymbol{\ell }_{m-1}M_{m-1} , \label{PropositCaseB3Equat1}
				\end{align}
				where $S\in \mathcal{B}_{m-1}^{+}\subseteq \mathcal{L}_{k}$. By the Taylor Theorem, it follows that:
				\begin{align*}
					h(S_{m}+S) & = h(S_{m})+h'(S_{m})\cdot S+\sum _{i\geq 2}\frac{h^{(i)}(S_{m})}{i!}S^{i} .
				\end{align*}
				Since $S_{m}$ is the solution of equation \eqref{PropCaseBEq10}, and by \eqref{EqDer2}, after dividing by $\boldsymbol{\ell }_{1}\cdots \boldsymbol{\ell }_{m-2}$, equation \eqref{PropositCaseB3Equat1} becomes:
				{\small \begin{align}
						& D_{m-1}(S) - \Big( \frac{D_{m-1}(N_{m})}{N_{m}}+\boldsymbol{\ell }_{m-1}\boldsymbol{\ell }_{m}P_{m}+ \boldsymbol{\ell }_{m-1}P_{m-1}-(\boldsymbol{\ell }_{m-1}\boldsymbol{\ell }_{m}T_{m}+\boldsymbol{\ell }_{m-1}T_{m-1})\cdot h'(S_{m})\Big) \cdot S \nonumber \\
						& + (\boldsymbol{\ell }_{m-1}\boldsymbol{\ell }_{m}T_{m}+\boldsymbol{\ell }_{m-1}T_{m-1})\cdot \sum _{i\geq 2}\frac{h^{(i)}(S_{m})}{i!}S^{i}  = \boldsymbol{\ell }_{m-1}\big( M_{m-1} + P_{m-1}\cdot S_{m}-T_{m-1}\cdot h(S_{m}) \big) . \label{PropCaseBEq12}
					\end{align}
				}Now, for the purpose of applying Lemma~\ref{Lemma4Second}, put:
				\begin{align}
					& N :=N_{m} , \nonumber \\
					& P :=\boldsymbol{\ell }_{m-1}\boldsymbol{\ell }_{m}P_{m}+ \boldsymbol{\ell }_{m-1}P_{m-1}-(\boldsymbol{\ell }_{m-1}\boldsymbol{\ell }_{m}T_{m}+\boldsymbol{\ell }_{m-1}T_{m-1})\cdot h'(S_{m}) , \nonumber \\
					& T :=\boldsymbol{\ell }_{m-1}\boldsymbol{\ell }_{m}T_{m}+\boldsymbol{\ell }_{m-1}T_{m-1} , \nonumber \\
					& M :=\boldsymbol{\ell }_{m-1}\big( M_{m-1} + P_{m-1}\cdot S_{m}-T_{m-1}\cdot h(S_{m}) \big) , \nonumber \\
					& h :=\sum _{i\geq 2}\frac{h^{(i)}(S_{m})}{i!}x^{i} . \label{PropositCaseB3Notat}
				\end{align}
				Note that $h\in x^{2}\mathcal{B}_{\geq m}^{+}\left[ \left[ x\right] \right] $. Since $N\in \mathcal{B}_{m}^{+}\subseteq \mathcal{B}_{\geq m}^{+}\subseteq \mathcal{L}_{k}$ and $\mathrm{ord}(h'(S_{m}))\geq \mathrm{ord}(S_{m})>\mathrm{ord}(N_{m})$, it follows that $P\in \mathcal{B}_{m-1}^{+}\subseteq \mathcal{L}_{k}$ with
				\begin{align*}
					\mathrm{ord}(P) &> (\mathbf{0}_{m-1},\mathbf{1}_{k-m+2}) .
				\end{align*}
				Note, from \eqref{PropositCaseB3Notat}, that $T\in \mathcal{B}_{m-1}^{+}\subseteq \mathcal{L}_{k}$ and $M\in \mathcal{B}_{m-1}^{+}\subseteq \mathcal{L}_{k}$, $\mathrm{ord}_{\boldsymbol{\ell }_{m-1}}(M) \geq 2$. By Lemma~\ref{Lemma4Second} (for $m-1$), it follows that there exists a unique solution $S:=S_{m-1}\in \mathcal{B}_{m-1}^{+}\subseteq \mathcal{L}_{k}$ of equation \eqref{PropCaseBEq12}. \\
				
				Inductively, we find $S_{m-2}\in \mathcal{B}_{m-2}^{+},\ldots ,S_{1}\in \mathcal{B}_{1}^{+}$. Put $S:=S_{m}+\cdots +S_{1}$. By construction, it follows that $S\in \mathcal{B}_{\geq 1}^{+}\subseteq \mathcal{L}_{k}$ is a solution of the equation \eqref{PropCaseB2EqMain}, and therefore, \eqref{PropEq8}. It satisfies the property $\mathrm{ord}(S)=\mathrm{ord}(S_{m}) > \mathrm{ord}(N_{m})=\mathrm{ord}(N)$. \\
				
				\noindent \emph{Proof of uniqueness of the solution}: Suppose that $S_{1},S_{2}\in \mathcal{B}_{\geq 1}^{+}\subseteq \mathcal{L}_{k}$ are distinct solutions of equation \eqref{PropEq8} satisfying $\mathrm{ord}\, (S_{1}),\mathrm{ord}\, (S_{2}) > \mathrm{ord}\, (N)$. Therefore, $S_{1}$ and $S_{2}$ are solutions of equation \eqref{PropCaseB2EqMain} satisfying $\mathrm{ord}\, (S_{i}) > \mathrm{ord}\, (N_{m})$, for $i=1,2$. Subtracting \eqref{PropCaseB2EqMain} for $S_{1}$ and $S_{2}$ and then multiplying by $N_{m}$, we get:
				\begin{align}
					& N_{m}\cdot D_{1}(S_{1}-S_{2})-(S_{1}-S_{2})\cdot D_{1}(N_{m}) \nonumber \\
					& = N_{m}\cdot P\cdot (S_{1}-S_{2}) - N_{m}\cdot T\cdot (h(S_{1})-h(S_{2})) . \label{DiffEqUniq2}
				\end{align}
				Since $\mathrm{ord}\, (S_{1}-S_{2}) \geq \min \left\lbrace \mathrm{ord}\, (S_{1}), \mathrm{ord}\, (S_{2}) \right\rbrace > \mathrm{ord}\, (N_{m})$, by solving the differential equation, it is easy to see that $\mathrm{Lt}(N_{m})\cdot D_{1}(\mathrm{Lt}(S_{1}-S_{2}))-\mathrm{Lt}(S_{1}-S_{2})\cdot D_{1}(\mathrm{Lt}(N_{m})) \neq 0$. Using this fact and by analyzing orders of terms in $N_{m}\cdot D_{1}(S_{1}-S_{2})-(S_{1}-S_{2})\cdot D_{1}(N_{m}) $, it can be proven that:
				\begin{align}
					& \mathrm{Lt}\big( N_{m}\cdot D_{1}(S_{1}-S_{2})-(S_{1}-S_{2})\cdot D_{1}(N_{m})\big) \nonumber \\
					& = \mathrm{Lt}\Big( \mathrm{Lt}(N_{m})\cdot D_{1}(\mathrm{Lt}(S_{1}-S_{2}))-\mathrm{Lt}(S_{1}-S_{2})\cdot D_{1}(\mathrm{Lt}(N_{m}))\Big) . \label{DiffEqUniq1}
				\end{align}
				By \eqref{DiffEqUniq1}, it follows that
				\begin{align}
					\mathrm{ord} \, \big( N_{m}\cdot D_{1}(S_{1}-S_{2})-(S_{1}-S_{2})\cdot D_{1}(N_{m}) \big) & \leq \mathrm{ord} \, (S_{1}-S_{2}) + \mathrm{ord} \, (N_{m}) + (0,\mathbf{1}_{k}) . \label{DiffEqUniq4}
				\end{align}
				Since $\mathrm{ord} \, (P) > (0,\mathbf{1}_{k})$, it follows that
				\begin{align}
					\mathrm{ord} \, \big( N_{m}\cdot P\cdot (S_{1}-S_{2}) \big) & > \mathrm{ord} \, (S_{1}-S_{2}) + \mathrm{ord} \, (N_{m}) + (0,\mathbf{1}_{k}) . \label{DiffEqUniq5}
				\end{align}
				By \eqref{DiffEqUniq4} and \eqref{DiffEqUniq5}, it follows that
				\begin{align}
					\mathrm{ord} \, \big( N_{m}\cdot D_{1}(S_{1}-S_{2})-(S_{1}-S_{2})\cdot D_{1}(N_{m}) \big) & < \mathrm{ord} \, \big( N_{m}\cdot P\cdot (S_{1}-S_{2}) \big) . \label{DiffEqUniq6}
				\end{align}
				On the other hand, since $N_{m}\in \mathcal{B}_{m}^{+}\subseteq \mathcal{L}_{k}$, $\mathrm{ord}\, (T)\geq \mathrm{ord}\, (\frac{D_{1}(N)}{N})=\mathrm{ord}\, (\frac{D_{1}(N_{m})}{N_{m}})$, it follows that $\mathrm{ord} \, (N_{m}\cdot T) \geq \mathrm{ord}\, (D_{1}(N_{m}))$. Since $\mathrm{ord}\, (S_{i})>\mathrm{ord}\, (N_{m})$, $i=1,2$, it follows that:
				\begin{align*}
					\mathrm{ord}\, (h(S_{1})-h(S_{2})) & \geq \mathrm{ord}\, (S_{1}-S_{2})+\min \left\lbrace \mathrm{ord}\, (S_{1}) , \mathrm{ord}\, (S_{2}) \right\rbrace \\
					& > \mathrm{ord}\, (S_{1}-S_{2})+\mathrm{ord} \, (N_{m}) .
				\end{align*}
				Now, we get:
				\begin{align}
					\mathrm{ord} \, \big( N_{m}\cdot T\cdot (h(S_{1})-h(S_{2})) \big) & > \mathrm{ord} \, (S_{1}-S_{2}) + \mathrm{ord} \, (N_{m}) + \mathrm{ord} \, (D_{1}(N_{m})) \nonumber \\
					& > \mathrm{ord} \, (S_{1}-S_{2}) + \mathrm{ord} \, (N_{m}) + (0,\mathbf{1}_{k}) . \label{DiffEqUniq7}
				\end{align}
				Now, by \eqref{DiffEqUniq4} and \eqref{DiffEqUniq7}, we get:
				\begin{align}
					\mathrm{ord} \, \big( N_{m}\cdot D_{1}(S_{1}-S_{2})-(S_{1}-S_{2})\cdot D_{1}(N_{m}) \big) & < \mathrm{ord} \, \big( N_{m}\cdot T\cdot (h(S_{1})-h(S_{2})) \big) . \label{DiffEqUniq8}
				\end{align}
				By \eqref{DiffEqUniq6} and \eqref{DiffEqUniq8}, we proved that the order of the left-hand side of equation \eqref{DiffEqUniq2} is strictly smaller than the order of the right-hand side of equation \eqref{DiffEqUniq2}, which is a contradiction. Therefore, the solution $S$ of equation \eqref{PropEq8}, satisfying $\mathrm{ord} \, (S)> \mathrm{ord} \, (N)$, obtained in the first part of the proof, is unique. 
			\end{proof}
		
			\begin{proof}[Proof of Proposition~\ref{LemaCaseB2OperatorsTS}]
				1. Let $S_{1},S_{2}\in \widetilde{\mathcal{B}}$, $S_{1}\neq S_{2}$, be arbitrary. Since $\mathrm{ord}\, (V) > 2\cdot \mathrm{ord}\, (R)+(0,\mathbf{1}_{k})$ and $(1+x)\log (1+x)=x+\sum _{i\geq 2}\frac{(-1)^{i}}{i(i-1)}x^{i}$, it follows that:
				\begin{align*}
					& \mathrm{ord}_{\boldsymbol{\ell }_{1}}\Big( (1+S_{1})\cdot \log (1+S_{1})\cdot D_{1}(V) - (1+S_{2})\cdot \log (1+S_{2})\cdot D_{1}(V)\Big) \\
					& \geq \mathrm{ord}_{\boldsymbol{\ell }_{1}}(S_{1}-S_{2})+2\cdot \mathrm{ord}_{\boldsymbol{\ell }_{1}}(R)+2 .
				\end{align*}
				Since $\mathcal{C}_{1}$ and $\mathcal{K}_{1}$ are $\frac{1}{2^{2+\mathrm{ord}_{\boldsymbol{\ell }_{1}}(R)}}$-contractions, by \eqref{CaseBIdent4}, it follows that $\mathcal{S}$ is a $\frac{1}{2^{2+\mathrm{ord}_{\boldsymbol{\ell }_{1}}(R)}}$-contraction on the space $(\widetilde{\mathcal{B}},d_{1})$. \\
				
				2. Let $S_{1},S_{2}\in \widetilde{\mathcal{B}}$, $S_{1}\neq S_{2}$, be arbitrary. Since $\mathrm{ord}(V) > 2\cdot \mathrm{ord}\, (R)+(0,\mathbf{1}_{k})$ and $(1+x)\log (1+x)=x+\sum _{i\geq 2}\frac{(-1)^{i}}{i(i-1)}x^{i}$ by \eqref{CaseBIdent4}, it follows that:
				\begin{align}
					\mathcal{T}(S_{1})-\mathcal{T}(S_{2}) & =-(1+R)\cdot \log (1+R)\cdot D_{1}(S_{1}-S_{2})+(S_{1}-S_{2})\cdot V \nonumber \\
					& +\big( (1+S_{1})\cdot \log(1+S_{1}) - (1+S_{2})\cdot \log(1+S_{2}) \big) \cdot D_{1}(R) \nonumber \\
					& = \mathrm{Lt}\big( (S_{1}-S_{2})\cdot D_{1}(R)-R\cdot D_{1}(S_{1}-S_{2})\big) + \mathrm{h.o.t.}, \label{PropCaseB3Equat5}
				\end{align}
				if $(S_{1}-S_{2})\cdot D_{1}(R)-R\cdot D_{1}(S_{1}-S_{2}) \neq 0$. Suppose that $(S_{1}-S_{2})\cdot D_{1}(R)-R\cdot D_{1}(S_{1}-S_{2}) =0$. By solving the linear differential equation, we get:
				\begin{align*}
					S_{1}-S_{2} & = C\cdot R, \quad C\in \mathbb{R} .
				\end{align*}
				Since $S_{1}-S_{2}\in \widetilde{\mathcal{B}}$, we get that $C=0$, i.e., $S_{1}=S_{2}$, which is a contradiction. Therefore, $(S_{1}-S_{2})D_{1}(R)-RD_{1}(S_{1}-S_{2}) \neq 0$, and \eqref{PropCaseB3Equat5} holds. By \eqref{PropCaseB3Equat5}, we get that:
				\begin{align*}
					\mathrm{ord}_{\boldsymbol{\ell }_{1}} \big( \mathcal{T}(S_{1})-\mathcal{T}(S_{2}) \big) & =\mathrm{ord}_{\boldsymbol{\ell }_{1}}(S_{1}-S_{2})+\mathrm{ord}_{\boldsymbol{\ell }_{1}}(R)+1 .
				\end{align*}
				Therefore, $\mathcal{T}$ is a $\frac{1}{2^{1+\mathrm{ord}_{\boldsymbol{\ell }_{1}}(R)}}$-homothety on the space $(\widetilde{\mathcal{B}},d_{1})$. \\
				
				3. Let $W\in \mathcal{S}(\widetilde{\mathcal{B}})$. Since $\mathrm{ord}\, (M), \mathrm{ord}\, (V) > 2\cdot \mathrm{ord} \, (R)+(0,\mathbf{1}_{k})$ and $\mathcal{C}_{1},\mathcal{K}_{1}:\mathcal{B}_{\geq 1}^{+}\to \mathcal{B}_{\geq 1}^{+}$ are $\frac{1}{2^{2+\mathrm{ord}_{\boldsymbol{\ell }_{1}}(R)}}$-contractions on $(\mathcal{B}_{\geq 1}^{+},d_{1})$, it follows that $W$ belongs to $\mathcal{B}_{\geq 1}^{+}\subseteq \mathcal{L}_{k}$ and satisfies $\mathrm{ord}(W)>2\cdot \mathrm{ord}(R)+(0,\mathbf{1}_{k})$. We solve in the space $\widetilde{\mathcal{B}}$ the equation $\mathcal{T}(S)=W$. Using \eqref{CaseBIdent4}, this equation becomes:
				\begin{align*}
					-(1+R)\cdot \log (1+R)\cdot D_{1}(S)+(1+S)\cdot \log(1+S)\cdot D_{1}(R)+S\cdot V & = W .
				\end{align*}
				Since $R\neq 0$, after dividing by $-(1+R)\log (1+R)$, we get the equivalent equation:
				{\small \begin{align}
						D_{1}(S)-S\cdot \frac{(D_{1}(R)+V)}{(1+R)\cdot \log (1+R)} - \frac{((1+S)\cdot \log (1+S)-S)\cdot D_{1}(R)}{(1+R)\cdot \log (1+R)} & = -\frac{W}{(1+R)\cdot \log (1+R)} . \label{PropCaseB3EquatApplicatProp}
					\end{align}
				}For the purpose of applying Proposition~\ref{Lemma3}, put:
				\begin{align*}
					& h :=(1+x)\log (1+x)-x , \\
					& N :=\log (1+R) , \\
					& K :=\frac{V}{(1+R)\cdot \log (1+R)} , \\
					& T :=-\frac{D_{1}(R)}{(1+R)\cdot \log (1+R)} = - \frac{D_{1}(N)}{N} , \\
					& M :=-\frac{W}{(1+R)\cdot \log (1+R)} .
				\end{align*}
				Since $\mathrm{ord}\, (V) > 2\cdot \mathrm{ord} \, (R)+(0,\mathbf{1}_{k})$, it follows that $\mathrm{ord} \, (K) > (0,\mathbf{1}_{k})$. Since $\mathrm{ord}(W)>2\cdot \mathrm{ord} \, (R)+(0,\mathbf{1}_{k})$ and $\mathrm{ord}(N)=\mathrm{ord}(R)$, it follows that $\mathrm{ord}(M)-\mathrm{ord}(N)>(0,\mathbf{1}_{k})$. By Proposition~\ref{Lemma3}, there exists a unique solution $S\in \widetilde{\mathcal{B}}$, of differential equation \eqref{PropCaseB3EquatApplicatProp}. Therefore, $S\in \widetilde{\mathcal{B}}$ is a unique solution of $\mathcal{T}(S)=W$ on the space $\widetilde{\mathcal{B}}$. This implies that $\mathcal{S}(\widetilde{\mathcal{B}})\subseteq \mathcal{T}(\widetilde{\mathcal{B}})$.
			\end{proof}

\end{document}